\documentclass[10pt]{amsart}
\usepackage{amsmath}
\usepackage{amscd}
\usepackage{amssymb}
\usepackage{amsthm}
\usepackage{enumerate}
\usepackage[normalem]{ulem}
\usepackage{mathtools}
\usepackage[usenames,dvipsnames]{xcolor}
\usepackage{stmaryrd}
\usepackage[margin=1.2in]{geometry}
\usepackage[T2A]{fontenc}
\usepackage[utf8]{inputenc}
\usepackage{mathrsfs}
\newtheorem{thm}{Theorem}[section]
\newtheorem{lem}[thm]{Lemma}
\newtheorem{prop}[thm]{Proposition}
\newtheorem{cor}[thm]{Corollary}

\theoremstyle{definition}
\newtheorem{defn}[thm]{Definition}

\newtheorem{cons}[thm]{Construction}

\numberwithin{equation}{section}

\newcommand{\bC}{{\mathbb C}}

\newcommand{\bE}{{\mathbb E}}
\newcommand{\bF}{{\mathbb F}}

\newcommand{\bM}{{\mathbb M}}
\newcommand{\bN}{{\mathbb N}}

\newcommand{\bR}{{\mathbb R}}

\newcommand{\bZ}{{\mathbb Z}}

\newcommand{\cA}{{\mathcal A}}
\newcommand{\cB}{{\mathcal B}}

\newcommand{\cF}{{\mathcal F}}
\newcommand{\cG}{{\mathcal G}}
\newcommand{\cH}{{\mathcal H}}

\newcommand{\cK}{{\mathcal K}}

\newcommand{\cM}{{\mathcal M}}

\newcommand{\cT}{{\mathcal T}}
\newcommand{\cU}{{\mathcal U}}
\newcommand{\cV}{{\mathcal V}}
\newcommand{\cW}{{\mathcal W}}

\renewcommand{\Re}{\operatorname{Re}}

\DeclareMathOperator{\actson}{\curvearrowright}

\DeclareMathOperator{\id}{id}

\DeclareMathOperator{\tr}{tr}

\DeclareMathOperator{\vol}{vol}

\DeclareMathOperator{\Int}{int}

\DeclareMathOperator{\ev}{ev}
\DeclareMathOperator{\Ball}{Ball}

\DeclareMathOperator{\diam}{diam}

\DeclareMathOperator{\Ext}{Ext}
\DeclareMathOperator{\Prob}{Prob}
\DeclareMathOperator{\co}{co}
\DeclareMathOperator{\spec}{spec}
\DeclareMathOperator{\Lip}{Lip}
\DeclareMathOperator{\Fix}{Fix}
\DeclareMathOperator{\Cent}{Cent}
\DeclareMathOperator{\md}{md}
\DeclareMathOperator{\Proj}{Proj}
\DeclareMathOperator{\Gr}{Gr}

\DeclarePairedDelimiter{\ip}{\langle}{\rangle}

\begin{document}

\title{Property (T) and strong 1-boundedness for von Neumann algebras}

\author{Ben Hayes}
\address{\parbox{\linewidth}{Department of Mathematics, University of Virginia, \\
141 Cabell Drive, Kerchof Hall,
P.O. Box 400137
Charlottesville, VA 22904}}
\email{brh5c@virginia.edu}
\urladdr{https://sites.google.com/site/benhayeshomepage/home}

\thanks{The author acknowledges support from NSF CAREER award DMS-2144739.
}
\title{On the generator problem for $W^{*}$-bundles}
\date{\today}

\begin{abstract}
We give an explicit example of a $W^{*}$-bundle $M$ over a compact, metrizable space $K$, which has each fiber $M_{p}$ a $\textrm{II}_{1}$-factor with separable predual, and which satisfies the following negation of the generator problem: given any finite family $a_{1},\cdots,a_{n}\in M$ of continuous sections, there is a $p\in K$ (depending upon that family) so that $a_{1,p},\cdots,a_{n,p}$ do not generate $M_{p}$ as a von Neumann algebra. More generally, if $N$ is a sub-bundle of $M$ with the property that each fiber is hyperfinite, or has a Cartan, or is generated by two commuting diffuse subalgebras, or has diffuse central sequence algebra, or is generated by a single sequential commutation orbit, then given any finite family $a_{1},\cdots,a_{n}\in M$ of continuous sections, there is a $p\in K$ (depending upon that family) so that $a_{1,p},\cdots,a_{n,p}$ together with $N_{p}$ do not generate $M_{p}$.
We discuss implications for the generator problem for von Neumann algebras: e.g. there is no ``continuous" way to take countably many generators for a tracial von Neumann algebra with separable predual and produce a single generator, at least if such a procedure works for all von Neumann algebras simultaneously. 
\end{abstract}
\maketitle 

\section{Introduction}

One of the most well known open problems in von Neumann algebras asks if every von Neumann algebra with separable predual is singly generated. This was answered affirmatively for type $I$ algebras in \cite{PearcySG}, for an algebra generated by countably many singly generated algebras in \cite{SSGen}, properly infinite algebras in \cite{WogenGen}, and non-prime algebras in \cite{BehnckeSG}. In \cite{WilligGen}, it was  noted that the generator problem reduces to the case of $\textrm{II}_{1}$-factors, and since the hyperfinite $\textrm{II}_{1}$-factor is singly generated, this means the problem is about \emph{nonamenable} $\textrm{II}_{1}$-factors (see \cite{CardinalInvariantsSherman} for a further historical discussion of the above developments). Since these initial papers on the subject, many other examples of $\textrm{II}_{1}$-factors with separable predual have been shown to be  singly generated, such as those which have a Cartan \cite{SorinCarSG}, those that are not full \cite{GePopaSG}, $L(SL_{n}(\bZ))$ for $n\geq 3$ \cite{GeShenGen}, those that are thin \cite{GePopaSG}, and those that are generated by a single sequential commutation orbit \cite{InternalSeqComm} in the sense of \cite{SeqCommutation} (this last class includes algebras with a Cartan, those that are not full, those that are not prime, as well as $L(SL_{n}(\bZ))$ for $n\geq 3$ and so provides a strict generalization of many prior results). Work of Shen \cite{ShenSingleGen} also provides a numerical invariant $\cG(M)$ which (roughly) counts the minimal number of generators weighted by what projections they lie under, additional work on this invariant was done in \cite{DSSW}.


It was recently announced \cite{elliott2026simpleseparablecalgebrasingly} that there is a simple, unital, separable $C^{*}$-algebra which is not singly generated. Previous results on $C^{*}$-algebra which are not singly generated all relied on some reduction to the abelian case. In the abelian case, counterexamples are abundant since for a compact Hausdorff space $K$ we have that $C(K)$ is $n$-generated if and only if $K$ embeds in $\bC^{n}$, and covering dimension theory allows one to rule this out in many cases (e.g. if $K$ is a compact manifold or CW complex of dimension larger than $2n$). The results of \cite{elliott2026simpleseparablecalgebrasingly} are a breakthrough in the field, but since the $C^{*}$-algebras given there are nuclear, it is unclear how to adapt the methods to give a von Neumann algebra which is not singly generated: every nuclear $C^{*}$-algebra will have all of its GNS completions hyperfinite von Neumann algebras with separable predual, and we know that all such algebras are singly generated. 

While there is a large literature proving that various classes of von Neumann algebras are singly generated, the general belief is that there should be a separable von Neumann algebra which is not singly generated. Many people believe that techniques from Voiculescu's free probability theory (e.g. his microstates free entropy dimension defined in \cite{Voiculescu1996}) should prove that the free group factors $L(\bF_{r})$ are not singly generated once $r$ is sufficiently large (see e.g. the papers \cite{SorinSSG,DimaNonMS}). However, the problem still remains open.

As the generator problem remains currently out of reach, it is natural to consider modified version of it for algebras that are ``in-between" $C^{*}$-algebras and von Neumann algebras. The class of $W^{*}$-bundles was introduced by Ozawa in \cite{OzawaBundle} as an operator algebraic means of describing a ``continuously varying" family of von Neumann algebras over a compact Hausdorff space. A $W^{*}$-bundle is a triple $(M,\bE,K)$ where $K$ is a compact Hausdorff space with $C(K)$ embedded in the  center of $M$, and $\bE\colon M\to C(K)$ is a faithful and tracial conditional expectation, together with extra hypotheses that imply that the images $M_{p}$ of $M$ under the GNS representations with respect to $\tau_{p}=\bE(x)(p)$ are all tracial von Neumann algebras (see Definition \ref{defn:W*bundle} for a precise definition). This allows us heuristically to view the algebras $M_{p}$ as von Neumann algebras which ``vary continuously in $p$". 
This concept was later generalized to the notion of a tracially complete $C^{*}$-algebra in \cite{TraciallyComplete}. The initial motivation for these notions was for classification of nuclear $C^{*}$-algebras, but it has become clear that they are important objects of study for their own sake (see e.g. \cite{ EvingtonTCC, Evington_thesis, EPBundles, mommaerts2026nonlocallytrivialmathrmwbundlefixed}). Specifically for us, tracially complete $C^{*}$-algebras can be viewed as objects that are a ``bridge" between $C^{*}$-algebras and tracial von Neumann algebras, and it is sensible to ask questions about them such as if they are hyperfinite, if they are non-full, etc (see \cite{TraciallyComplete, UnifGamma, UnifGamm3, UnifGamm2}). 
As a modification of the generator problem, it thus becomes natural to ask what the generator problem for $W^{*}$-bundles and tracially complete $C^{*}$-algebras would be. A naive formulation would be to find a factorial, tracially complete $C^{*}$-algebra which is not singly generated as a tracially complete $C^{*}$-algebra. However, every $C(K)$ can be viewed as tracially complete $C^{*}$-algebra (and its structure as a tracially complete $C^{*}$-algebra is the same as its $C^{*}$-algebra structure), and as previously mentioned examples of these that are not $n$-generated are abundant. In our opinion, a better analogue would be a separable $W^{*}$-bundle which is not ``singly generated over $C(K)$'' (to rule out abelian examples, since every $C(K)$ is singly generated over itself), and which has the property that all of its fibers are nonamenable $\textrm{II}_{1}$-factors (so that it better connects to the generator problem for von Neumann algebras, which reduces to the case of nonamenable $\textrm{II}_{1}$-factors). Our main theorem provides the first such example.

\begin{thm}\label{thm: kind of main thm}
There is a $W^{*}$-bundle $(M,\bE,K)$ over a compact metrizable space, which has each fiber a $\textrm{II}_{1}$-factor with separable predual, and which satisfies the following  negation of the generator problem: given any finite family $a_{1},\cdots,a_{n}\in M$ of continuous sections, there is a $p\in K$ (depending upon the family) so that $a_{1,p},\cdots,a_{n,p}$ do not generate $M_{p}$ as a von Neumann algebra.    
\end{thm}

In fact, we prove something more general. Let $(M,\bE,K)$ be a $W^{*}$-bundle and $N\leq M$ a sub-bundle. We say that a family $(a_{i})_{i\in I}$ in $M$ \textbf{fiberwise generates $M$ over $N$} if for every $p\in K$ we have that $W^{*}(\{a_{i,p}\}_{i\in I}\cup N_{p})=M_{p}$. 
One motivation for this consideration is the following: the statement that every $\textrm{II}_{1}$-factor is singly generated is equivalent to the statement that there is a $k\in \bN$ so that every $\textrm{II}_{1}$-factor is can be generated by at most $k$ elements (e.g. by \cite{PearcySG}, see also \cite[Theorem 6.3]{CardinalInvariantsSherman}). Thus, for connections to the generator problem for von Neumann algebras, it is sensible to take a class of von Neumann algebras which are known to be singly generated (e.g. hyperfinite algebras), and ask if there is a $W^{*}$-bundle $(M,\bE,K)$ which cannot be fiberwise finitely generated over any sub-bundle whose fibers belong to that class. E.g. if one could show that every separable $\textrm{II}_{1}$-factor was, say, singly generated over a hyperfinite subfactor (or singly generated over any of the class of algebras listed above which are known to be singly generated), then it would follow that every separable $\textrm{II}_{1}$-factor is singly generated. We show that we can produce a $W^{*}$-bundle which is not fiberwise finitely generated over any sub-bundle whose fibers fall into a large class of von Neumann algebras which are known to be singly generated.

\begin{thm}\label{thm: main thm intro}
Let $Q$ be a $\textrm{II}_{1}$-factor which has an embedding into an ultraproduct of matrices with trivial relative commutant (e.g. $Q=L(\bF_{r})$ for $r>1$, or $Q=L(G)$ where $G$ is an i.c.c., residually finite, Property (T) group). Then there is a compact, metrizable space $K$, and a $W^{*}$-bundle $M$over $K$ so that:
\begin{enumerate}[(i)]
    \item for each $p\in K$, we have that $M_{p}$ is a   $\textrm{II}_{1}$-factor with separable predual \footnote{$M$ is even separable in the uniform $2$-norm, which is a priori stronger than having fibers with separable predual}. Additionally, $M$ contains the trivial bundle over $K$ with fiber $Q$ as a sub-bundle, and $Q'\cap M_{p}=\bC$ for every $p\in K$. In particular, each $M_{p}$ is nonamenable. \label{item: intro rel commm}
    \item \label{item: intro classes} Suppose $N$ is any sub-bundle of $M$ and that for every $p\in K$, $N_{p}$ satisfies one of the following properties:
    \begin{itemize}
    \item $N_{p}$ is hyperfinite, 
        \item $N_{p}$ has a Cartan subalgebra,
        \item $N_{p}$ has diffuse central sequence algebra,
        \item  $N_{p}$ is generated by two commuting diffuse subalgebras,
        \item $N_{p}$ can be generated by a single sequential commutation orbit in the sense of \cite{SeqCommutation}.
    \end{itemize}
    Then $M$ cannot be fiberwise generated over $N$ by a finite family of continuous sections: i.e. given $n\in \bN$ and $a=(a_{1},\cdots,a_{n})\in M^{n}$ there is a $p\in K$ (depending upon $a$ and $N$) so that $M_{p}\ne W^{*}(\{a_{i,p}\}_{i=1}^{n}\cup N_{p})$. 
    \item \label{item: unfiromly S1B INTRO} More generally, suppose that $N$ is any sub-bundle of $M$ and that $\sup_{p\in K}h(N_{p},\tau_{p})<+\infty$ (with $h$ being the $1$-bounded entropy), then $M$ cannot be fiberwise generated by a finite family of continuous sections over $N$. 
\end{enumerate}
\end{thm}

The fact that we may take $Q=L(\bF_{r})$ follows form the existence  on quantum expanders due to Hastings \cite{Hastings} (see also \cite{PisierQuantumExpander} for a separate proof as well as  \cite{CollinsBordenave} for  a more general statement), combined with Voiculescu's asymptotic freeness theorem \cite[Theorem 3.8]{VoicAsyFree}. When $Q=L(G)$ with $G$ i.c.c., residually finite, Property (T) the existence of such an embedding uses that $\delta_{e}$ is approximated by characters of finite-dimensional representations, and extremality of $\delta_{e}$ in the character space forces that such finite-dimensional representations have irreducibility subrepresentations whose characters converge to $\delta_{e}$. See Theorem \ref{thm: Property T trivial rel comm} for more details.

The $M$ constructed in Theorem \ref{thm: main thm intro} is done by completing a $C^{*}$-algebra $A$ with respect to a family of traces, and $A$ has some sub-algebra which is (roughly) a universal $C^{*}$-algebra of countably many self-ajdoint contractions. However, the methods are flexible enough that we can replace self-adjoint contractions with unitaries or even projections. As such, we obtain the following refinement of our results.
\begin{thm}\label{thm: main theorem intro but with projections}
Let $Q$ be a $\textrm{II}_{1}$-factor which has an embedding into an ultraproduct of matrices with trivial relative commutant. Then there is a compact, metrizable space $K$ and a $W^{*}$-bundle $M$ over $K$ so that $(M,\bE,K)$ so that $M$ is separable in the uniform $2$-norm,  satisfies conclusions (\ref{item: intro rel commm})-(\ref{item: unfiromly S1B INTRO}) of Theorem \ref{thm: main thm intro}, and so that $M$ is generated (as a tracially complete $C^{*}$-algebra) by a countable family of projections.
\end{thm}

As an immediate corollary, we deduce the following (which appears to be new). For a group $G$, we use $C^{*}(G)$ for its full $C^{*}$-algebra. 

\begin{cor}
We have that $C^{*}((\bZ/2\bZ)^{*\infty})$ is not singly generated as a $C^{*}$-algebra. 
\end{cor}

We discuss the relation to the generator problem for von Neumann algebras, and the general sketch of the proof in the rest of this introduction.

\subsection*{Relation to the generator problem for von Neumann algebras}

Given a $\textrm{II}_{1}$-factor $Q$ with separable predual, we may always find countably many generators $a=(a_{j})_{j=1}^{\infty}$ for $Q$. Attempting to solve the generator problem for von Neumann algebra, essentially means that we have to produce an algorithm which inputs such a sequence $a$ and repackages it into a single generator for $Q$. Theorem \ref{thm: kind of main thm} implies that if such an algorithm exists, then it cannot vary continuous in the trace of $Q$ (or more properly speaking in the law of $a$). See Theorems \ref{thm: cont func of law} for a precise formulation.

As mentioned above, the generator problem for von Neumann algebras is equivalent to asking if there is a uniform bound on the minimal number of generators of any von Neumann algebra with separable predual. 
It thus makes sense to consider packaging hyperfinite subalgebras $N$ of $M$ (or subalgebras with a Cartan, non-prime subalgebras, etc.) as part of the ``data" for an algorithm alluded to above that would solve the generator problem. Theorem \ref{thm: main thm intro} implies that there is a large class $\mathcal{C}$ of singly generated algebras so that even if we package in  a subalgebra $N$ of $M$ with $N\in \mathcal{C}$ into such an algorithm, then it still cannot vary continuously into the law of $a$. See Theorem \ref{thm: no continuous algorithm} for a precise version of this.  Moreover, the fact that each fiber is a factor forces it to be the case that $M$ does not have abelian quotients (nor does any sub-bundle obtained by restricting to closed subsets with nonempty interior), and thus rules out that the lack of finite generation is occurring purely for abelian reasons.

We also think it is of significance that there is a trivial sub-bundle which is the trivial bundle over $K$ with fibers over a $\textrm{II}_{1}$-factor $Q$. The proofs that various classes of algebras are singly generated von Neumann algebras typically use partial isometries, projections, etc. It is unclear if a $W^{*}$-bundle always has nowhere vanishing continuous sections which are projection-valued, take values in non-unitary partial isometries, etc. However, since in our results we are assuming that we have a trivial sub-bundle over a $\textrm{II}_{1}$-factor $Q$, we have ample access to sections which are projection-valued, partial isometry valued etc. coming from norm-bounded, SOT-continuous functions $K\to Q$ (and we are even, by Theorem \ref{thm: main theorem intro but with projections}, allowed to assume that $M$ is generated as a tracially complete $C^{*}$-algebra by its projections).

Finally, let us note that by \cite{EPBundles} we may regard a $W^{*}$-bundle as the continuous sections of a topological bundle over $K$ with fibers tracial von Neumann algebras. Recall that if $X$ is a Polish space, then $f\colon K\to X$ is \emph{universally measurable} if for every open $U\subseteq X$ we have that $f^{-1}(U)$ is in the $\mu$-completion of the Borel sets for every Borel probability measures $\mu$ on $K$. This is more general than being a Borel function (see \cite[Theorems 14.2 and 21.10]{KechrisClassic}).
Using the topological structure given in \cite{EPBundles} one can also define Borel sections, universally measurable sections etc (see Section \ref{Sec:universmally measurable bundle} for more details). By standard descriptive set theory methods (i.e the Jankov-von Neumann theorem), if $M_{p}$ is singly generated for every $p\in K$, then we can find a universally measurable section $(a_{p})_{p\in K}$ so that $W^{*}(a_{p})=M_{p}$ for every $p\in K$. Thus if we could upgrade ``continuous" in Theorem \ref{thm: main thm intro} to ``universally measurable", then this would amount to a disproof of the generator problem. 


For the above reasons, we think our results give substantial evidence to the belief that there is a von Neumann algebra with separable predual which is not singly generated.
\subsection*{Discussion of methods involved}
Our methods rely on a mixture of operator algebraic machinery on $W^{*}$-bundles, classical results in Banach space geometry, results in strong convergence in random matrices, Voiculescu's free entropy dimension theory, and adapting ideas from the entropy theory in dynamical settings (specifically the entropy theory of actions of sofic groups on compact metric space or probability spaces). 

Each of these areas is of significant interest in its own right. The theory of strong convergence (initiated in \cite{HaagThorbNormBound}) in random matrices has seen tremendous advances lately (e.g. \cite{PTkilled,CollinsBordenave,  bordenave2023norm, Chen_2026, chen2024newapproachstrongconvergence, gao2026newsource, loudermagee2025limitgroups,   MdLSstrongasymptoticfreenesshaar, MageeThomas2023, Parraud2024strong}), see \cite{ magee2025strong, vHICM, vHCDM} for surveys in this direction. It also has strong connections to other areas of mathematics beyond operator algebras, such as geometry \cite{MageeHide, magee2025strongconvergenceuniformlyrandom, MageeThomas2023, song2025randomharmonicmapsspheres} and random graph theory \cite{CollinsBordenave, bordenave2023norm}. Voiculescu's free entropy dimension theory is the source of the first proofs for many important structural results on free group factors, such as absence of Cartan subalgebras \cite{Voiculescu1996}, and primeness \cite{GePrime}. In many cases free entropy dimension proves structural results on free group factors which are currently inaccessible by other means, e.g. thinness of free group factors \cite{GeNonPrimeSG}, the absence of finite multiplicity MASAs in free group factors \cite{DykemaFreeEntropy}, Property C' for free group factors \cite{Hayes2018}, or the fact that free group factors cannot be generated by two amenable subalgebras with diffuse intersection \cite{Jung2007}. This last property was generalized to
the Peterson-Thom property for free group factors (\cite{HayesPT} along with \cite{PTkilled, bordenave2023norm, chen2024newapproachstrongconvergence, magee2025strong, Parraud2024strong}). As a consequence of the Peterson-Thom property one may deduce solidity properties for free group factors \cite{hayes2024generalsolidityphenomenaanticoarse} beyond the celebrated solidity and strong solidity theorems due to Ozawa \cite{OzawaSolidActa}, and Ozawa-Popa \cite{OzPopaCartan}.
Entropy for dynamical systems is also a well-studied subject as well, and the case of sofic groups involved the resolution of long-standing problems as well as developments of new exciting phenomena not present for actions of amenable groups (see e.g. the surveys \cite{LewisICM, BowenExamples}).

We sketch the idea behind the proof of Theorem \ref{thm: main thm intro} when $Q=L(\bF_{2})$.
Ultimately, our aim is to produce an invariant $\delta_{MD}(M;X)$ and a relative version $\delta_{MD}(M|N;X)$ for tracially complete $C^{*}$-algebras $(M;X)$ satisfying that: 
\begin{itemize}
    \item $\delta_{MD}(M|N;X)$ is bounded by the minimal number of self-adjoints in $M$ which together with $N$ generate $M$,
    \item $\delta_{MD}(M|N;X)$ satisfies the chain inequality for conditioning: $\delta_{MD}(M)\leq \delta_{MD}(M|N)+\delta_{MD}(M)$,
    \item $\delta_{MD}(M)\leq 1$ if $\sup_{\tau\in X}h(\overline{\pi_{\tau}(M)}^{SOT},\tau)<+\infty$, with $h$ being the $1$-bounded entropy of \cite{Jung2007, Hayes2018} (equivalently if $h(M;X)<+\infty$ in the sense of \cite{macmahon20261boundedentropycalgebras}). In particular, we have that $\delta_{MD}(M)\leq 1$ if for each $\tau$ we have that $\overline{\pi_{\tau}(M)}^{SOT}$ is hyperfinite, or has a Cartan, or is generated by two commuting diffuse subalgebras, or has diffuse central sequence algebra, or is generated by a single sequential commutation orbit. 
\end{itemize}
We then produce a factorial $W^{*}$-bundle $M$ over a metrizable space $K$ which has a fixed sub-bundle which is a trivial bundle over $L(\bF_{2})$, with each fiber $M_{p}$ a non-amenable $\textrm{II}_{1}$-factor, and which satisfies $\delta_{MD}(M)=+\infty$. 

The idea behind the definition of $\delta_{MD}$ is that one can view free entropy theory as analogous to the theory of dynamical entropy for actions of sofic groups on compact spaces or probability spaces. In both settings, one has a version of ``finitary approximates" for some generating object and use various means to compute its ``size" (e.g. literal cardinality, or volume, or packing number), and then takes appropriate limits with the end goal of producing an invariant which is independent of the choice of generator. For a tracial von Neumann algebra $(M,\tau)$, the $1$-bounded entropy is analogous to the entropy of probability measure-preserving actions of sofic groups. For a $C^{*}$-algebra, the $1$-bounded entropy (due to MacMahon in \cite{macmahon20261boundedentropycalgebras}) of its universal tracial completion is analogous to the topological entropy of an action of a sofic group by homeomorphisms\footnote{A frequently proposed analogue of topological entropy is to use microstates for the operator norm instead of for the universal tracial completion. While certainly a very important notion, this does not match up with the classical setting. E.g. producing norm microstates for a $C^{*}$-algebra is, in practice, several orders of magnitude harder than producing weak$^{*}$-microstates for a tracial von Neumann algebra. Whereas in the setting of dynamical systems, producing topological microstates is strictly easier than producing ones that are weak$^{*}$-approximates for a given invariant measure. Said differently: for topological dynamical systems $G\actson K$, topological entropy measures some kind of size of equivariant embeddings of $C(K)$ into a natural Loeb measure space (which is an ultraproduct of probability spaces) associated with your sofic group, and does not measure the size of equivariant embeddings into some norm ultraproduct of finite-dimensional $C^{*}$-algebras. This is similar to considering $*$-homomorphisms from a $C^{*}$-algebra $A$ into a \emph{tracial} ultraproduct of matrices, instead of a \emph{norm} ultraproduct of matrices. The former is captured by the entropy of the universal tracial completion of $A$, whereas the latter captures a different notion without a precise analogue in dynamical systems.}, and
Voiculescu's free entropy dimension is analogous to metric mean dimension of topological dynamical systems (compare e.g the formulations \cite[Section 4]{LindWeiss},\cite[Section 4]{Li}  and \cite[Definition 2.1]{JungLemma}). 
Metric mean dimension is not an invariant for topological dynamical systems, and analogously there are difficulties showing that free entropy dimension is an invariant of the generating von Neumann algebra. There is a natural modification of metric mean dimension, known as mean dimension defined in \cite{Li, LindWeiss}, which \emph{is} an invariant of topological dynamical systems. Modifying the formula for mean dimension in the context of free entropy theory produces a quantity we can prove is invariant (see Theorem \ref{thm: invariance of MD}), and this is the desired $\delta_{MD}$ described above. The relative version $\delta_{MD}(M|N)$ follows in a similar manner, adapting known formulations on relative mean dimension given in \cite{li2025soficconditionalmeandimension, liang2021conditionalmeandimension,  TsukaHur}. We remark that one of the major initial applications of mean dimension in \cite{LindWeiss} was to exhibit a topological action  $\bZ\actson X$ which did not embed into the full shift over $[0,1]$, i.e. so that $C(X)$ is not generated by the translates of a single continuous function. It is thus natural to adapt mean dimension as a means to obstruct single/finite generation for operator algebras (and this relates nicely to established work relating minimal number of generators to free entropy dimension, see \cite{DSSW}). 

Computations for $\delta_{MD}$ are difficult in general. Indeed, as we will show in Theorem \ref{thm: nontrivial computations for tracial vNas are hard}, if there is a single tracial von Neumann algebra $(M,\tau)$ for which $\delta_{MD}(M,\tau)>1$, then it follows that there is a $\textrm{II}_{1}$-factor which is not singly generated. However, for tracially complete $C^{*}$-algebras, the situation is much nicer. 
Though it is not trivial, one can use classical results of Szarek \cite{SzarekAlmostEuc, STALmostEuc} in geometric Banach space theory to give many examples of tracially complete $C^{*}$-algebras $(M;X)$ with $\delta_{MD}(M;X)=+\infty$. 
The easiest examples, however, will be uniform tracial completions of certain universal $C^{*}$-algebras such as $C^{*}(\bF_{\infty})$ or the infinite full free product of $C([-1,1])$ etc. Unfortunately, these will not give rise to $W^{*}$-bundles with nonamenable factor fibers.

The trick is ultimately to start with a sequence of microstates $(U^{(k)}_{1},U^{(k)}_{2})\in \cU(\bM_{k}(\bC))^{2}$ for $L(\bF_{2})$ which have the property that \emph{ for every free ultrafilter $\omega$}, the resulting embedding of $L(\bF_{2})$ into the tracial ultraproduct of matrices with respect to $\omega$ has trivial relative commutant, e.g. using quantum expanders \cite{bordenave2023norm, Hastings,PisierQuantumExpander}.

 We then take some appropriately large $C^{*}$-algebra $A$ containing $C^{*}_{\lambda}(\bF_{2})$, such as the full free product of $C^{*}_{\lambda}(\bF_{2})$ with $C^{*}(\bF_{\infty})$ and consider the compact (non-convex) set of traces $K$ we get from homomorphisms on $A$ into tracial ultraproduct of matrices which when restricted to $C^{*}_{\lambda}(\bF_{2})$ produce the given embedding of $C^{*}_{\lambda}(\bF_{2})$ with trivial relative commutant. The trivial relative commutant condition guarantees all such traces are extremal, and a diagonal argument guarantees that $K$ is compact. Choquet theory implies that if $X$ is the closed convex hull of $K$, then $X$ is a face in $\cT(A)$ and is a Bauer simplex with extreme points $K$. It then follows by \cite[Theorem 3]{OzawaBundle} (see also \cite[Theorem 3.37]{TraciallyComplete}) that the tracial completion\footnote{strictly speaking $X$ may a priori fail to be a faithful set of traces. This presents no issue and one simply passes to the separation by $X$ and then the tracial completion.} $(M;X)$ of $(A,K)$ is a factorial $W^{*}$-bundle over $K$. This bundle will have the property that $\delta_{MD}(M;X)=+\infty$, and we naturally have a trivial sub-bundle of $M$ which is the trivial bundle over $K$ with fibers $L(\bF_{2})$. The construction of $(A,K)$ allows us to say that $(M;X)$ has many microstates and we are able to use this (together with the aforementioned geometric Banach space theory techniques) to prove that $\delta_{MD}(M;X)=+\infty$. Together with the chain inequality for conditional $\delta_{MD}$, we show that $M$ cannot be fiberwise finitely generated (more generally that it cannot be fiberwise finitely generated over a sub-bundle satisfying the hypotheses of Theorem \ref{thm: main thm intro} (\ref{item: intro classes}), (\ref{item: unfiromly S1B INTRO})).

\subsection*{Acknowledgments} I thank David Jekel and Srivatsav Kunnawalkam Elayavalli for insightful discussions at the early stages of this project. 
Part of this project was carried out at the 2025 ``C$^{*}$-Algebras" conference at Oberwolfach, where discussions with 
Jamie Gabe, Matt Kennedy, Jenny Pi, Chris Schafhauser, and Sven Raum were helpful. Conversations with Tim Austin at the ``Spectral gaps 2025" conference in Portoro\v{z} were insightful.
I also worked on this project at YMC*A 2026, where comments by Stuart White and Kiefer Mommaerts were useful. 
I thank David Sherman for clarifying the history of the generator problem to me, as well as for interesting discussions on this work, and I thank David Jekel and Stuart White for valuable comments on an earlier draft of this paper.

\subsection*{AI disclosure statement.} No LLM's were used at any stage during this project.

\tableofcontents

\section{Preliminaries}

\subsection{Notation}
We use $\bC^{*}\ip{(T_{i})_{i\in I}}$ for all noncommutative $*$-polynomials in abstract variables $(T_{i})_{i\in I}$ indexed  by $I$ (i.e. the free unital $*$-algebra on the set $I$). 
Given a unital $*$-algebra $A$ and $a\in A^{I}$, there is a unique unital $*$-homomorphism $\ev_{a}\colon \bC^{*}\ip{(T_{i})_{i\in I}}\to A$ so that $\ev_{a}(T_{i})=a_{i}$ for all $i\in I$. We set $P(a):=\ev_{a}(P)$ for $P\in \bC^{*}\ip{(T_{i})_{i\in I}}$. 
If $J$ is a set, and $P=(P_{j})_{j\in J}\in \bC^{*}\ip{(T_{i})_{i\in I}}^{J}$, we set $P(a)=(P_{j}(a))_{j\in j}$. If $J_{0}\subseteq J$, we use $P|_{J_{0}}\in \bC^{*}\ip{(T_{i})_{i\in I}}^{J_{0}}$ for $(P_{j})_{j\in J_{0}}$. Given $k\in \bN$, we use $[k]=\{1,\cdots,k\}$. 

 Recall that a \emph{Polish space} is a separable topological space whose topology is generated by a complete metric (but we do not regard the metric as part of the data). If $X$ is compact Hausdorff space, we use $\Prob(X)$ for the space of Radon probability measures on $X$. If $X$ is Polish, we use $\Prob(X)$ for the space of Borel probability measures on $X$. For a finite-dimensional Hilbert space $\cH$, we use $\tr$ for the \emph{normalized} trace on $B(\cH)$, i.e.
\[tr(T)=\frac{1}{\dim(\cH)}\sum_{j}\ip{Te_{j},e_{j}},\]
with $(e_{j})_{j}$ an orthonormal basis of $\cH$. 

For a $C^{*}$-algebra $A$, we use $A_{s.a.}$ for its self-adjoint elements and if $A$ is unital, we use $\cU(A)$ for its group of unitaries. We also use $\Proj(A)$ for the projections in $A$.
For a group $G$, we let $C^{*}(G)$ be its full $C^{*}$-algebra.
For unital $C^{*}$-algebras $A,B$ we use $A*_{u}B$ for their \emph{universal free product}, i.e. $A*_{u}B$ contains unital embeddings of $A,B$ and has the universal property that any two unital $*$-homomorphisms on $A,B$ extend uniquely to $A*_{u}B$. More generally, if $I$ is a set and $(A_{i})_{i\in I}$ are unital $C^{*}$-algebras, we use $*_{i\in I,u}A_{i}$ for their universal free product. If each $A_{i}$ is a fixed $C^{*}$-algebra $A$, we will typically use $A^{*_{u}I}$ instead. The reader may be surprised to learn that we will not use reduced free products much in this paper.

\subsection{Choquet theory}

Given a convex subset $C$ of a real vector space, we let $\Ext(C)$ denote its extreme points. If $F\subseteq C$ is convex, we say that $F$ is a \emph{face} of $C$, if whenever $p,q\in C$ and $t\in (0,1)$ with $(1-t)p+tq\in F$, then necessarily $p,q\in F$.

If $C$ is a compact convex subset of a locally convex space $V$, and $\mu\in \Prob(C)$, then by \cite[Theorem 3.27]{GrandpaRudin} there is a unique $x\in C$ so that 
\[\phi(x)=\int \phi(p)\,d\mu(p) \text{ for all $\phi\in V^{*}$}.\]
We use $\int p\,d\mu(p)$ for this unique point $x\in C$. If $C$ is metrizable, we call $C$ a \emph{Choquet simplex} if for every $x\in C$, there is a unique $\mu\in \Prob(\Ext(C))$ so that $x=\int p\,d\mu(p)$. By \cite[Theorem 3.1.18]{Sakai} the tracial state space of any unital $C^{*}$-algebra is a Choquet simplex. A Choquet simplex $C$ is called a \emph{Bauer simplex} if $\Ext(C)$ is compact (by \cite[comments before Proposition V.7.9]{Conway}, even finite-dimensional convex sets may fail to have a closed set of extreme points).

\subsection{Tracially complete $C^{*}$-algebras}

By a \emph{tracial von Neumann algebra} we mean a pair $(M,\tau)$ where $M$ is a von Neumann algebra and $\tau$ is a faithful, normal, tracial state on $M$.

\begin{defn}[Definition 3.4 of\cite{TraciallyComplete}]
A \emph{tracially complete $C^{*}$-algebra} is a pair $(M,X)$ where 
\begin{itemize}
    \item $X\subseteq \cT(M)$ is weak$^{*}$-compact and convex, and
    \item $\Ball(M)$ is complete under $\|a\|_{2;X}=\sup_{\tau\in X}\tau(a^{*}a)^{1/2}$. 
\end{itemize}
    
\end{defn}

The idea here is that a tracially complete $C^{*}$-algebra should be viewed as an object ``in between" a tracial von Neumann algebra and a $C^{*}$-algebra. Any pair $(A,X)$ of a unital $C^{*}$-algebra $A$ together with a weak$^{*}$-compact convex set of traces $X$ on $A$ generates a tracially complete $C^{*}$-algebra via a natural completion functor (see \cite{OzawaBundle}, \cite[Section 3.3]{TraciallyComplete}) 
by setting
\[\overline{A}^{X}=\frac{\{(a_{n})_{n=1}^{\infty}\in \ell^{\infty}(\bN,A):(a_{n})_{n=1}^{\infty} \text{ is $\|\cdot\|_{2,X}$-Cauchy}\}}{\{(a_{n})_{n=1}^{\infty}\in \ell^{\infty}(\bN,A):\|a_{n}\|_{2,X}\to 0\}},\]
which has a natural diagonal embedding of $A$. Further each $\tau\in X$ is a trace on $\overline{A}^{X}$ since $\tau(a_{n})$ is a Cauchy sequence if $(a_{n})_{n=1}^{\infty}$ is $\|\cdot\|_{2,X}$-Cauchy. We continue to use $X$ for this set of traces on $\overline{A}^{X}$. The pair $(\overline{A}^{X},X)$ is then a tracially complete $C^{*}$-algebra \cite[Proposition 3.23]{TraciallyComplete}.

Various important properties such as Property Gamma and hyperfiniteness have natural analogues for tracially complete $C^{*}$-algebras \cite{TraciallyComplete}, which amount to demanding that these properties hold ``uniformly" over all traces. For $\tau\in X$, we let $M_{\tau}=\overline{\pi_{\tau}(M)}^{SOT}$. Given a set $I$ and $a\in M^{I}$, we let $C^{*}_{X}(a)=\overline{\{P(a):P\in \bC^{*}\ip{(T_{i})_{i\in I}}\}}^{\|\cdot\|_{2;X}}.$

A tracially complete $C^{*}$-algebra is called \emph{factorial}
 if $X$ is a face in $\cT(M)$. By \cite[Proposition 3.14]{TraciallyComplete} this is equivalent to demanding that $M_{\tau}$ be a factor for every $\tau\in \Ext(X)$. 

\begin{defn}[Section 5 of \cite{OzawaBundle}] \label{defn:W*bundle}
A \emph{$W^{*}$-bundle} is a triple $(M,\bE,K)$ where:
\begin{itemize}
    \item $K$ is a compact Hausdorff space,
    \item $M$ is a unital $C^{*}$-algebra containing a unital copy of $C(K)$ with $C(K)\subseteq Z(M)$
    \item $\bE\colon M\to C(K)$ is a faithful conditional expectation,
    \item $\bE$ is tracial: for all $x,y\in M$ we have $\bE(xy)=\bE(yx)$. 
    \item $\Ball(M)$ is complete under the norm $\|x\|_{\bE}=\|\bE(x^{*}x)\|^{1/2}.$
\end{itemize}
\end{defn}

Note that each $p\in K$ gives rise to a trace $\tau_{p}$ on $M$ by $\tau_{p}(x)=\bE(x)(p)$. We let $\pi_{p}$ be GNS representation arising from $\tau_{p}$ and we set $M_{p}=\pi_{p}(M)$. By \cite[Theorem 11]{OzawaBundle}, we know that $M_{p}$ is already a von Neumann algebra. We also let $x_{p}=\pi_{p}(x)$. Thus we may think of $M$ as a continuous field $(M_{p})_{p\in K}$ of tracial von Neumann algebras over $K$, and $M$ as the continuous sections of this field. For a precise version of this, see \cite{EPBundles}.
Note that if $X=\overline{\co}(\{\tau_{p}:p\in K\})$, then 
\[\|a\|_{\bE}=\sup_{\tau\in X}\tau(a^{*}a)^{1/2}, \text{ for all $a\in M$.}\]
Hence, we may view every $W^{*}$-bundle as a tracially complete $C^{*}$-algebra. In fact, the definition of a tracially complete $C^{*}$-algebra was naturally motivated by Ozawa's definition of a $W^{*}$-bundle. 
It follows by \cite[Proposition 3.6]{TraciallyComplete} that a $W^{*}$-bundle $(M,\bE,K)$, when viewed as a tracially complete $C^{*}$-algebra, is factorial if and only if $M_{p}$ is a factor for each $p\in K$. If $I$ is set and $a=(a_{i})_{i\in I}\in M^{I}$, we define $a_{p}\in M_{p}^{I}$ for $p\in K$ by $a_{p}=(a_{i,p})_{i\in I}$, where $a_{i,p}:=\pi_{p}(a_{i})$.
A particular case of a bundle is the following: if $K$ is a compact Hausdorff space and $(Q,\tau)$ is a tracial von Neumann algebra, let $C_{\sigma}(K,Q)$ be all operator norm bounded and SOT-continuous functions $f\colon K\to Q$. 
Then $\|f\|=\sup_{x\in K}\|f(x)\|$ is a norm on $C_{\sigma}(K,Q)$ which makes it a $C^{*}$-algebra, and one can check that the conditional expectation $\bE\colon C_{\sigma}(K,Q)\to C(K)$ given by $\bE(f)(x)=\tau(f(x))$ turns $C_{\sigma}(K,Q)$ into a $W^{*}$-bundle. We call this the \emph{trivial bundle over $Q$ with fiber $K$.}

While in general it is difficult to axiomatize \emph{which} tracially complete $C^{*}$-algebras are $W^{*}$-bundles, it is easier in the case of factorial tracially complete $C^{*}$-algebras. Indeed, by \cite[Theorem 3]{OzawaBundle} (see also \cite[Theorem 3.37]{TraciallyComplete}) and \cite[Proposition 3.23 (ii) and (iv)]{TraciallyComplete} if $A$ is  unital $C^{*}$-algebra and $K\subseteq\cT(A)$ is a face in $\cT(A)$, and a Bauer simplex (i.e. $K$ is a Choquet simplex and $\Ext(K)$ is compact), then the completion of $(A,K)$ is a $W^{*}$-bundle over $K$. Since $\cT(A)$ is always a Choquet simplex when $A$ is a unital $C^{*}$-algebra, the following folklore proposition gives us a good means to producing $W^{*}$-bundles. 

\begin{prop}\label{prop: creating faces}
Let $V$ be a locally convex space and $C\subseteq V$ a metrizable Choquet simplex. If $K\subseteq \Ext(C)$ is compact, and $X=\overline{\co}(K)$, then $X$ is a face in $C$,  $\Ext(X)=K$, and $X=\left\{\int p\,d\mu(p):\mu \in \Prob(K)\right\}$. 
\end{prop}

\begin{proof}
Suppose $x_{1},x_{2}\in C$ and $t\in (0,1)$ satisfy that $x:=(1-t)x_{1}+tx_{2}\in C$. By \cite[Theorem 3.28]{GrandpaRudin} and compactness of $K$, we may write $x=\int p\,d\mu(x)$ for some $\mu\in \Prob(K)$. Since $C$ is metrizable, by the Choquet–Bishop de Leeuw theorem \cite[Theorem I.4.8 and Corollary I.4.9]{ChoquetTheory} we may write $x_{i}=\int p\,d\mu_{i}(x)$ with $\mu_{i}\in \Prob(\Ext(C))$. Since $C$ is Choquet, the representing measure for $x$ is unique. Since $K\subseteq \Ext(C)$, this forces $\mu=(1-t)\mu_{1}+t\mu_{2}$. Since $0<t<1$ and $\mu(K)=1$ we must have that $\mu_{i}(K)=1$ for $i=1,2$. This implies that $x_{i}\in \overline{\co}(K)=X$ for $i=1,2$. Thus $X$ is a face in $C$. Finally, since $K$ is compact, Milman's partial converse \cite[Theorem V.7.8]{Conway} tells us that $\Ext(X)=K$. 

\end{proof}

\subsection{Tracial von Neumann algebras and microstates}\label{sec: all praise the OG}

Given a unital $C^{*}$-algebra $A$ and a state $\varphi$ on $A$, and a normal element $a\in A$, we let $\mu_{a,\varphi}$ be the measure on $\spec(a)$ given by
\[\varphi(f(a))=\int f\,d\mu_{a,\varphi}, \text{ for all $f\in C(\spec(a))$.}\]
If $\varphi$ is understood from context, we often use $\mu_{a}$ instead of $\mu_{a,\varphi}$.



Suppose $(M,X)$ is a tracially complete $C^{*}$-algebra. If $I$ is a set and $F\subseteq I$ is finite, we use $\|\cdot\|_{2,F:X}$ for the seminorm on $M^{I}$ given by 
\[\|a\|_{2,F;X}=\left(\sum_{i\in I}\|a_{i}\|_{2;X}^{2}\right)^{1/2}.\]
If $I$ is finite itself, we will typically use $\|a\|_{2;X}$ in place of $\|a\|_{2,F;X}$ for $a\in M^{F}$. Finally, if $X=\{\tau\}$, we will typically use $\|\cdot\|_{2,F}$ instead of $\|\cdot\|_{2,F;\{\tau\}}$ (if $\tau$ is clear from context), and will also use 
\[\|a\|_{2}=\left(\sum_{i\in I}\|a_{i}\|_{2}^{2}\right)^{1/2},\]
if $I$ is finite. All of this applies for the case that $M=\bM_{n}(\bC)$ in which case $\|\cdot\|_{2}$ on $\bM_{n}(\bC)$ is assumed to be taken with respect to the \emph{normalized} trace. 
Namely, for $F\subseteq I$ finite, $p\in [1,+\infty]$, we define $\|\cdot\|_{p,F}$ on $\bM_{k}(\bC)^{I}$ by
\[\|A\|_{p,F}^{p}=\sum_{j\in F}\tr((A_{j}^{*}A_{j})^{p/2})^{1/p}, \text{ if $p\in [1,\infty)$},\]
\[\|A\|_{\infty,F}=\max_{j\in F}\|A_{j}\|,\]
and if $I$ itself is finite, we typically use $\|A\|_{p},\|A\|_{\infty}$ in place of $\|A\|_{F,p},\|A\|_{\infty,p}$.

Given $R\colon I\to [0,+\infty)$, we let $\Sigma_{R}$ consist of all linear functionals $\ell\colon \bC^{*}\ip{(T_{i})_{i\in I}}\to \bC$ so that:
\begin{itemize}
    \item $\ell(PQ)=\ell(QP)$ for all $P,Q\in \bC^{*}\ip{(T_{i})_{i\in I}}$,
    \item $\ell(P^{*}P)\geq 0$ for all $P\in \bC^{*}\ip{(T_{i})_{i\in I}}$,
    \item $\ell((T_{i}^{*}T_{i})^{k})\leq R_{i}^{2k}$ for all $i\in I,k\in \bN$. 
\end{itemize}
If $R$ is a constant function (say constantly equal to $C$), we typically use $\Sigma_{C,I}$ instead of $\Sigma_{R}$. Note that if $(M,\tau)$ is a tracial von Neumann algebra and $a\in \prod_{i\in I}R_{i}\Ball(M)$, then the functional $\ell_{a}\colon \bC^{*}\ip{(T_{i})_{i\in I}}\to\bC$ given by $\ell_{a}(P)=\tau(P(a))$ is in $\Sigma_{R}$. 

For $\ell\in\Sigma_{R}$ define an inner product $\ip{\cdot,\cdot}_{\ell}$ on $\bC^{*}\ip{(T_{i})_{i\in I}}$ by $\ip{P,Q}_{\ell}=\ell(P^{*}Q)$. We let $L^{2}(\ell)$ be the separation/completion of $\bC^{*}\ip{(T_{i})_{i\in I}}$ under this inner product. For $P\in \bC^{*}\ip{(T_{i})_{i\in I}}$, we let $[P]$ be its image in $L^{2}(\ell)$ under separation/completion. By the same argument as in \cite[Proposition 4.2]{BNLaws}, for each $\ell\in\Sigma_{R}$ there is a unique $*$-homomorphism $\pi_{\ell}\colon \bC^{*}\ip{(T_{i})_{i\in I}}\to B(L^{2}(\ell))$ with $\pi_{\ell}(P)[Q]=[PQ]$ for all $P,Q\in \bC^{*}\ip{(T_{i})_{i\in i}}$. Further we have that $\|\pi_{\ell}(T_{i})\|\leq R_{i}$ for all $i\in I$. We set $W^{*}(\ell)=\overline{\pi_{\ell}(\bC^{*}\ip{(T_{i})_{i\in I}})}^{SOT}$, and $\tau_{\ell}(x)=\ip{x[1],[1]}$ for $x\in W^{*}(\ell)$. Then $(W^{*}(\ell),\tau_{\ell})$ is a tracial von Neumann algebra, and setting $\vartheta_{\ell}=(\pi_{\ell}(T_{i}))_{i\in I}$ we have that $\ell_{\vartheta_{\ell}}=\ell$.

We endow the space of linear functionals (and hence $\Sigma_{R})$ on $\bC^{*}\ip{(T_{i})_{i\in I}}$ with the weak$^{*}$-topology given by the family of seminorms $\|L\|_{P}=|L(P)|$ for $P\in \bC^{*}\ip{(T_{i})_{i\in i}}$.
We may define a norm $\|P\|_{R}$ on $\bC^{*}\ip{(T_{i})_{i\in I}}$ by 
\[\|P\|_{R}=\sup_{a\in \prod_{i\in I}R_{i}\Ball(A)}\|P(a)\|,\]
where the supremum is over all unital $C^{*}$-algebras $A$ and all tuples $a\in \prod_{i}R_{i}\Ball(A)$. Letting $C^{*}(R)$ be the completion of $\bC^{*}\ip{(T_{i})_{i\in i}}$ under this norm, the above description shows that we have a homeomorphism $\cT(C^{*}(R))\cong \Sigma_{R}$, given by $\tau\mapsto \tau|_{\bC^{*}\ip{(T_{i})_{i\in I}}}$. In particular, $\Sigma_{R}$ is weak$^{*}$-compact. Again if $I$ is finite and $R$ is constantly equal to, say $C$, we typically use $\|\cdot\|_{C,I}$ and $C^{*}(C,I)$ instead of $\|\cdot\|_{R}$. If $I=[r]$, we tend to use $\|\cdot\|_{C,r},C^{*}(C,r)$ instead.
The case where we most commonly use such a function $R\in [0,\infty)^{I}$ is the following: let $A$ be a $C^{*}$-algebra and $a\in A^{I}$. We call $R\in [0,+\infty)^{I}$ a \emph{cutoff function for $a$} if $\|a_{i}\|\leq R_{i}$ for all $i\in I$. If $a$ is clear from context, we often call $R$ a \emph{cutoff function.}

Given $R\in [0,+\infty)^{I}$, and $\mathcal{O}\subseteq \Sigma_{R}$, and $k\in \bN$, we set
\[\Gamma_{R}^{(k)}(\mathcal{O})=\left\{A\in \prod_{i\in I}R_{i}\Ball(\bM_{k}(\bC)):\ell_{A}\in \mathcal{O}\right\}.\]
Here $\ell_{A}$ is the law with respect to the unique tracial state $\tr$ on $\bM_{k}(\bC)$. Typically, $\mathcal{O}$ will be weak$^{*}$-open, in which case $\Gamma^{(k)}_{R}(\mathcal{O})$ is open. If $R$ is a constant function, say constantly equal to $C$, we will typically use $\Gamma_{C}^{(k)}(\mathcal{O})$ instead of $\Gamma^{(k)}_{R}(\mathcal{O})$. We are often interested in compact convex sets $K\subseteq \Sigma_{R}$. In this case, we refer to the collection $(\Gamma^{(k)}_{R}(\mathcal{O}))_{k,\mathcal{O}}$ ranging over $k\in \bN$, and neighborhoods $\mathcal{O}$ of $K$ as \emph{Voiculescu's microstates spaces for $K$.}

It will be convenient to work with another norm on $\bC^{*}\ip{(T_{i})_{i\in I}}$ other than $\|\cdot\|_{2},\|\cdot\|_{R}$. It follows as in \cite[Lemma 3.1]{HJKECoho} that there is a function
\[L\colon [0,\infty)^{I}\times \bC^{*}\ip{(T_{i})_{i\in I}}\to [0,+\infty)\]
so that for all $P\in \bC^{*}\ip{(T_{i})_{i\in I}}$, every tracial von Neumann algebra $(M,\tau)$ and every $a,b\in \prod_{i}R_{i}\Ball(M)$, we have 
\[\|P(a)-P(b)\|_{2}\leq L(R,P)\|a-b\|_{2,F},\]
where $F$ is  the smallest finite subset of $I$ so that $P\in \bC^{*}\ip{(T_{i})_{i\in F}}$.
For a fixed $R\in [0,+\infty)^{I}$, we let $\|P\|_{R,\Lip}$ be minimal $L(R,P)$ for which the above inequality holds for  every tracial von Neumann algebra $(M,\tau)$ and every $a,b\in \prod_{i}R_{i}\Ball(M)$. More generally, if $P\in \bC^{*}\ip{(T_{i})_{i\in I}}^{G}$ where $G$ is a finite set, we let $\|P\|_{R,\Lip}$ be the minimal $L$ so that for all tracial von Neumann algebras $(M,\tau)$ and all $a,b\in \prod_{i\in I}R_{i}\Ball(M)$ we have 
\[\|P(a)-P(b)\|_{2}\leq \|a-b\|_{2,F}\]
where $F$ is the minimal finite subset of $I$ so that $P_{j}\in \bC^{*}\ip{(T_{i})_{i\in F}}$ for all $j\in G$.

Given a sequence $(M_{n},\tau_{n})$ of tracial von Neumann algebras, and a free ultrafilter $\omega\in\beta\bN\setminus\bN$, we let $\prod_{n\to\omega}(M_{n},\tau_{n})$ denote their tracial ultraproduct. If $\tau_{n}$ is understood from context, we will often instead write $\prod_{n\to\omega}M_{n}$. Given $(x_{n})_{n=1}^{\infty}\in \prod_{n\to\omega}M_{n}$ with $\sup_{n}\|x_{n}\|<+\infty,$ we let $(x_{n})_{n\to\omega}$ be its image in $\prod_{n\to\omega}M_{n}$. We use $\tau_{\omega}$ for the trace on $\prod_{n\to\omega}M_{n}$ defined by $\tau_{\omega}((x_{n})_{n\to\omega})=\lim_{n\to\omega}\tau_{n}(x_{n})$. If each $M_{n}$ is $\bM_{k_{n}}(\bC)$ for some sequence $(k_{n})_{n=1}^{\infty}$ in $\bN$ (so that necessarily $\tau_{n}=\tr$), we typically use $\tr_{\omega}$ instead of $\tau_{\omega}$. 

\section{Definition of the invariant $\delta_{MD}$: relative and nonrelative}

In this section, we define a modification (denoted $\delta_{MD}(a))$ of Voiculescu's microstates free entropy dimension for tuples $a$ in a tracial von Neumann algebra. This modification will have the property that:
\begin{itemize}
    \item it is defined for tuples in a tracially complete $C^{*}$-algebra (instead of just for tracial von Neumann algebras),
    \item it is an invariant for the tracially complete $C^{*}$-algebra generated by the tuple. Namely, if $(M,X)$ is any tracially complete $C^{*}$-algebra, if $I,J$ are sets, and $a\in M^{I},b\in M^{J}$ are two tuples with $C^{*}_{X}(a)=C^{*}_{X}(b)$, then $\delta_{MD}(a)=\delta_{MD}(b)$.
\end{itemize}
As discussed in the introduction, the core idea is to modify the definition of mean dimension to the context of Voiculescu's microstates spaces instead of topological microstates in the context of, say, actions by homeomorphisms of sofic groups. 
We proceed to give the definition. 
\begin{defn}
Let $X$ be a Hausdorff topological space and $\rho$ a continuous pseudometric on $X$. Given $\varepsilon>0$, and $\Xi$  a set we say that $f\colon X\to \Xi$ is \emph{$\varepsilon$-injective} with respect to $\rho$ if $\rho(x,y)<\varepsilon$ whenever $f(x)=f(y)$. 
We let $\dim_{\varepsilon}(X,\rho)$ be the minimal $d$ so that there is a $d$-dimensional compact Hausdorff space $\Delta$ (in the sense of covering dimension) and a continuous $f\colon X\to \Delta$ which is $\varepsilon$-injective (with respect to $\rho$).   
\end{defn}

Typically for use $X\subseteq \bM_{n}(\bC)^{I}$ for some set $I$, and $\rho(a,b)=\|a-b\|_{2,F}$ for some $F\subseteq I$. We will refer to an $\varepsilon$-injective function with respect to $\rho$ as an $(\varepsilon,F)$-injective function in this setting.

The reader may have similar notions in other contexts, e.g. the usage of mean dimension and sofic mean dimension in \cite[Section 4]{LindWeiss}, \cite[Section 4]{Li}. In these references, this notion is often defined more combinatorially, and we are following the convention given in \cite[Section 1.5]{Gro}. Though not needed for our work, we state for the convenience of the reader several equivalent formulations of $\dim_{\varepsilon}$ here. See e.g. \cite[Section 1.1]{UWEquiv} for the proof. 

\begin{prop}\label{prop: dim TFAE}
Let  $(X,d)$ be a compact metric space. Then the following are equivalent:
\begin{enumerate}[(i)]
    \item $\dim_{\varepsilon}(X,d)\leq d$,
    \item there is a finite open cover $\{U_{i}\}_{i=1}^{r}$ of $X$ of multiplicity at most $d+1$ (i.e with $\sum_{i=1}^{r}1_{U_{i}}\leq d+1$) and with $\diam(U_{i},\rho)<\varepsilon$ for all $i\in [r]$, 
    \item there is a simplicial complex $Z$ of dimension at most $d$ and an $\varepsilon$-injective, continuous map $f\colon X\to Z$,
    \item there is a metrizable topological space $Z$ of covering dimension at most $d$, and an $\varepsilon$-injective, continuous map $f\colon X\to Z$,
    \item there is a Hausdorff topological space of covering dimension at most $d$ and an $\varepsilon$-injective, continuous map $f\colon X\to Z$. 
\end{enumerate}
\end{prop}

 We will typically be interested in the quantity $\dim_{\varepsilon}(\Gamma^{(k)}_{R}(\mathcal{O}),\|\cdot\|_{2,F})$ where $\mathcal{O}$ is a neighborhood of some compact convex set of laws, and $R$ is a cutoff function on some arbitrary index set $I$, and $F$ is a finite subset of $I$. The reader might notice that:
\begin{itemize}
    \item $\Gamma^{(k)}_{R}(\mathcal{O})$ is not a compact space unless $\mathcal{O}$ is a closed neighborhood,
    \item $\|\cdot\|_{2,F}$ is not a metric unless $I$ is finite and $F=I$,
    \item $\Gamma^{(k)}_{R}(\mathcal{O})$ is not a metrizable space unless $I$ is a countable set (which we will not assume in general).
\end{itemize}
This first item is not a major issue: when we define $\delta_{MD}$ we will take an infimum over neighborhoods $\mathcal{O}$ of a compact set of laws. Every such neighborhood will contain a compact neighborhood. By monotonicity of the quantities involved, we could simply replace the infimum with a compact neighborhood if we wished to use one of the other items in Proposition \ref{prop: dim TFAE}. instead of $\dim_{\varepsilon}$. The other items are harder to fix, especially the last (if $I$ is countable, we could use a suitable metric defined by summing $2^{-n}\min(1,\|\cdot\|_{2,F_{n}})$ across an increasing sequence of finite subsets with union $I$, and the reader may check that this does not change the ultimate quantity $\delta_{MD}$ by the same argument as in, e.g \cite[Lemma 4.4]{Li}).

As mentioned above, we are largely stating Proposition \ref{prop: dim TFAE} for the convenience and intuition of the reader. It will be used once to get a lower bound on $\dim_{\varepsilon}$ for a compact subset of a finite-dimensional Hilbert space (where none of the above four bullet points present an issue). For concreteness, we will stick to $\dim_{\varepsilon}$ as formulated above and not use the other items in Proposition \ref{prop: dim TFAE}. Ultimately, it should be the case that replacing $\dim_{\varepsilon}$ with other quantities implicitly defined by the other items in Proposition \ref{prop: dim TFAE} will not change the invariant $\delta_{MD}$ we define (e.g. by copying the details provided in \cite[Section 1.1]{UWEquiv}), but we will not pursue this here. 

For a tracially complete $C^{*}$-algebra $(M;X)$, a set $I$ and $a\in M^{I}$, we let $L_{a;X}=\{\tau\circ \ev_{a}:\tau\in X\}$. Note that if $R\in [0,=\infty)^{I}$ is a cutoff function, then $L_{a;X}$ is a compact, convex subset of $\Sigma_{R}$. 

\begin{defn}
Let $(M;X)$ be a tracially complete $C^{*}$-algebra, $I,J$  sets, $R\in [0,\infty)^{I\sqcup J}$ and $a\in \prod_{i\in I}R_{i}\Ball(M)$, $b\in \prod_{j\in J}R_{j}\Ball(M)$. For a weak$^{*}$-neighborhood $\mathcal{O}$ of $L_{a,b;X}$, a finite $F\subseteq I$ and an $\varepsilon>0$, we iteratively define
\[\md_{\varepsilon,F,R}(\mathcal{O})=\limsup_{k\to\infty}\frac{1}{k^{2}}\dim_{\varepsilon}(\Gamma^{(k)}_{R}(\mathcal{O}),\|\cdot\|_{2,F}),\]
\[\md_{\varepsilon,F,R}(a,b)=\inf_{\mathcal{O}}\md_{\varepsilon,F,R}(\mathcal{O}),\]
\[\delta_{MD,R}(a:b)=\sup_{\varepsilon,F}\md_{\varepsilon,F,R}(a,b),\]
where the infimum is over all weak$^{*}$-neighborhoods of $L_{a.b;X}$ and the supremum is over all $\varepsilon>0$ and $F\subseteq I$ finite. 
\end{defn}
We will show later that $\delta_{MD}$ does not depend upon the choice of $R$ (in fact, we will show something much more general), so we typically drop $R$ and write e.g. $\md_{\varepsilon,F}(\mathcal{O}),\md_{\varepsilon,F}(a,b)$, $\delta_{MD}(a:b)$ if it is clear from context. 
Note that $\|\cdot\|_{2,F}$ is a seminorm on $\bM_{k}(\bC)^{I}$. It is tempting to replace $\Gamma^{(k)}_{R}(\mathcal{O})$ with its projection $\pi_{F}(\Gamma^{(k)}_{R}(\mathcal{O}))$ onto $\bM_{k}(\bC)^{F}$ where $\|\cdot\|_{2,F}$ is now an honest norm and take  $\dim_{\varepsilon}$ of the image. This is ultimately a different quantity than $\dim_{\varepsilon}(\Gamma^{(k)}_{R}(\mathcal{O}),\|\cdot\|_{2,F})$ since we demand that $\varepsilon$-injective maps with respect to $\|\cdot\|_{2,F}$ on $\Gamma^{(k)}_{R}(\mathcal{O})$ are continuous with respect to the induced topology coming from $\bM_{k}(\bC)^{I}$, and there may not be a continuous section $\pi_{F}(\Gamma^{(k)}_{R}(\mathcal{O}))\to \Gamma^{(k)}_{R}(\mathcal{O})$. As a result, if we were to replace $\dim_{\varepsilon}(\Gamma^{(k)}_{R}(\mathcal{O}),\|\cdot\|_{2,F})$ with $\dim_{\varepsilon}(\pi_{F}(\Gamma^{(k)}_{R}(\mathcal{O})),\|\cdot\|_{2})$ we would a priori lose desirable monotonicity properties that we can deduce once we prove that $\delta_{MD}(N:M)$ is independent of choice of generators. E.g. with this alternate definition of $\delta_{MD}$ it becomes unclear if the inequality $\delta_{MD}(N:M)\leq \delta_{MD}(M)$ holds.
There is also a relative version of mean dimension in the context of factor maps. We will need a similar relative version.

\begin{defn}
Let $(M,X)$ be a tracially complete $C^{*}$-algebra, $I_{i},i=1,2,3$ be sets and let $a\in M^{I_{1}},b\in M^{I_{2}},c\in M^{I_{3}}$ be given. For finite sets $F\subseteq I_{1}\sqcup I_{2}$, $G\subseteq I_{2}$, $\varepsilon,\theta>0$, $n\in \bN$ and $\Omega\subseteq \bM_{n}(\bC)^{I_{1}\sqcup I_{2}}$, we say that a function $f\colon \Omega\to \Delta$ is \textbf{$(\varepsilon,F|\theta,G)$-injective} if whenever $x,y\in \Omega$  satisfy that $f(x)=f(y)$ \emph{and} $\|x-y\|_{2,G}<\theta$, then we have $\|x-y\|_{2,F}<\varepsilon$. We let $\dim_{\varepsilon,F|\theta,G}(\Omega)$ be the smallest $d$ so that there is a $d$-dimensional simplicial complex and a $(\varepsilon,F|\theta,G)$-injective function $f\colon \Omega\to \Delta$. For a weak$^{*}$-neighborhood $\mathcal{O}$ of $L_{a,b,c:X}$, $\varepsilon,\theta>0$ and finite $F\subseteq I_{1},G\subseteq I_{2}$ we iteratively define

We define
\[\md_{\varepsilon,F|\theta,G}(\mathcal{O})=\limsup_{k\to\infty}\frac{1}{k^{2}}\dim_{\varepsilon,F|\theta,G}(\Gamma^{(k)}_{R}(\mathcal{O})),\]
\[\md_{\varepsilon,F|\theta,G}(a|b:c)=\inf_{\mathcal{O}}\md_{\varepsilon,F|\theta,G}(\mathcal{O}),\]
\[\md_{\varepsilon,F}(a|b:c)=\inf_{\theta,G}\md_{\varepsilon,F|\theta,G}(a|b:c),\]
\[\delta_{MD}(a|b:c)=\sup_{\varepsilon,F}\md(a|b:c),\]
where the first infimum is overall weak$^{*}$-neighborhoods $\mathcal{O}$ of $L_{a,b;X}$, the second infimum is over $\theta>0$ and finite $G\subseteq I_{2}$, and the supremum is over $\varepsilon>0$ and all finite $F\subseteq I_{1}$.

\end{defn}

The following inequality justifies calling the above a ``relative mean dimension".

\begin{prop}\label{prop: relative inequlaity}
Let $(M,X)$ be a tracially complete $C^{*}$-algebra, $I_{i},i=1,2,3$ be sets, and let $a\in M^{I_{1}},b\in M^{I_{2}},c\in M^{I_{3}}$ be given. Then
\[\delta_{MD}(a,b:c)\leq \delta_{MD}(b:a,c)+\delta_{MD}(a|b:c)\]
\end{prop}

\begin{proof}
Fix a cutoff function $R\in [0,\infty)^{I}$ where $I:=\bigsqcup_{i=1}^{3}I_{i}$. Let $\varepsilon>0$ and a finite $F\subseteq I_{1}\sqcup I_{2}$ be given. 
Fix a weak$^{*}$-neighborhood $\mathcal{O}$ of $L_{a,b,c;X}$. Suppose a finite $G\subseteq I_{2}$ and a $\theta>0$ are given. Let $g\colon \Gamma^{(n)}_{R}(\mathcal{O})\to \Delta_{1}$ be a $(\theta,G)$-injective function, and let $f\colon \Gamma^{(n)}_{R}(\mathcal{O})\to \Delta_{2}$ be a $(\varepsilon,F|\theta,G)$-injective function. Then $f\times g$ is an $(\varepsilon,F)$-injective function. This shows that 
\[\md_{\varepsilon,F}(a,b:c)\leq \md_{\varepsilon,F}(\mathcal{O})\leq \md_{\varepsilon,F|\theta,G}(\mathcal{O})+\md_{\theta,G}(\mathcal{O}).\]
Taking the infimum over $\mathcal{O}$ we see that
\[\md_{\varepsilon,F}(a,b:c)\leq \md_{\varepsilon,F|\theta,G}(a|b:c)+\md_{\theta,G}(b:a,c)\leq \md_{\varepsilon,F|\theta,G}(a|b:c)+\delta_{MD}(b:a,c).\]
Taking the infimum over $G,\theta$, then the supremum over $F,\varepsilon$ completes the proof. 

\end{proof}

\subsection{Proof of invariance}

In this section, we show that if $(M,X)$ is a tracially complete $C^{*}$-algebra, then $\delta_{MD}(a:b|c)=\delta_{MD}(a':b'|c')$ if $a,b,c,a',b',c'$ are tuples in $M$ with 
\[C^{*}_{X}(b)=C^{*}_{X}(b'),\]
\[C^{*}_{X}(a,b)=C^{*}_{X}(a',b'),\]
\[C^{*}_{X}(a,b,c)=C^{*}_{X}(a',b',c').\]
In fact, we prove something slightly more general.
As in \cite{Hayes2018, macmahon20261boundedentropycalgebras}, an important step in our proof is the following proposition (see \cite[Corollary 5.1]{macmahon20261boundedentropycalgebras} for the proof), which ultimately relies on \cite[Proposition 3.28]{TraciallyComplete}. 

\begin{prop}\label{prop: better KDT}
Let $(M,X)$ be a tracially complete $C^{*}$-algebra, $I$ a set and $a\in M^{I}$. Let $R\in [0,\infty)^{I}$ be a cutoff function for $a$. For any $b\in C^{*}_{X}(a)$, and any $\varepsilon>0$, there is a $P\in \bC^{*}\ip{(T_{i})_{i\in I}}$ so that $\|P\|_{R}\leq \|b\|$ and $\|P(a)-b\|_{2;X}<\varepsilon$. 
\end{prop}

As in \cite{macmahon20261boundedentropycalgebras}, we use the following lemma (note: in \cite{macmahon20261boundedentropycalgebras} this lemma is stated with $R$ being uniformly bounded, but the proof given in \cite{macmahon20261boundedentropycalgebras} works for general cutoff functions, mutatis mutandis). For a set $X$ equipped with a pseudometric $\rho$, and $A\subseteq X$, we let $N_{\delta}(A,\rho)=\{x\in X:\rho(a,x)<\delta \text{ for some $x\in X$}\}$.

\begin{lem}\label{lem: key lem for invariance}
Let $(M,X)$ be a tracially complete $C^{*}$-algebra, and let $I$ be a set. Fix $x\in M^{I}$, and a cutoff function $R\in [0,\infty)^{I}$. 
\begin{enumerate}[(i)]
\item Let $Q\in \bC^{*}\ip{T_{i}:i\in I}^{I'}$ for some set $I'$ and let $R'\in [0,+\infty)^{I'}$ satisfy that $\|Q_{i'}\|_{R,\infty}\leq R_{i}'$ for all $i'\in I'$.Given a weak$^{*}$-neighborhood $\mathcal{O}$ of $L_{Q(x);X}$, there is a weak$^{*}$-neighborhood $\mathcal{V}$ of $L_{x;X}$ so that $Q(\Gamma^{(n)}_{R}(\mathcal{V}))\subseteq \Gamma^{(n)}_{R'}(\mathcal{O})$ for all $n\in \bN$. \label{item: pushforward neighborhods}
\item For any weak$^{*}$-neighborhood $\mathcal{O}$ of $L_{x;X}$, there is a $\kappa>0$, a finite $F\subseteq I$, and a weak$^{*}$-neighborhood $\mathcal{V}$ of $L_{x;X}$ so that $N_{\kappa}(\Gamma^{(n)}(\mathcal{V}),\|\cdot\|_{2,F}))\cap \left(\prod_{i\in I}R_{i}\Ball(\bM_{n}(\bC))\right)\subseteq \Gamma^{(n)}_{R}(\mathcal{O})$ for all $n\in \bN$. \label{item: small perturbation neighborhood}
\item For any weak$^{*}$-neighborhood $\mathcal{O}$ of $L_{x;X}$, there is an $\eta>0$ and a finite $F\subseteq I$ with the following property. Whenever $b\in \prod_{i\in I}R_{i}\Ball(M)$ satisfies $\|x-b\|_{2,F;X}<\eta$, there is a weak$^{*}$-neighborhood $\mathcal{V}$ of $L_{b;X}$ with $\Gamma^{(n)}_{R}(\mathcal{V})\subseteq \Gamma^{(n)}_{R}(\mathcal{O})$.  \label{item: small perturbation neighborhood 2}
\end{enumerate}
\end{lem}

We now proceed to prove that the $\delta_{MD}(a|b:c)$ is an invariant of the algebras generated by $a,$,$(a,b)$, and $(a,b,c)$. In fact, we prove a more general statement which will give us good monotonicity properties for $\delta_{MD}$, analogous to those for $1$-bounded entropy (see Corollary \ref{cor:basic properties}).

\begin{thm}\label{thm: invariance of MD}
Let $(M,X)$ be a tracially complete $C^{*}$-algebra, and $(I_{i})_{i=1}^{3},(I_{i}')_{i=1}^{3}$ be sets. Suppose that $a\in M^{I_{1}},b\in M^{I_{2}},c\in M^{I_{3}}$, and $a'\in M^{I_{1}'},b'\in M^{I_{2}'},c'\in M^{I_{3}'}$ are given, and that 
\[C^{*}_{X}(b)\supseteq C^{*}_{X}(b'),\]
\[C^{*}_{X}(a,b)\subseteq C^{*}_{X}(a',b'),\]
\[C^{*}_{X}(a,b,c)\supseteq C^{*}_{X}(a',b',c').\]
Then
\[\delta_{MD}(a|b:c)\leq\delta_{MD}(a'|b':c').\]
\end{thm}

\begin{proof}
To ease notation somewhat, set $I=\sqcup_{i=1}^{3}I_{i}$ and $I'=\sqcup_{i=1}^{3}I_{i}'$.
 Fix cutoff functions $R\in [0,\infty)^{I}$ and $R'\in [0,\infty)^{I'}$. Let $\varepsilon>0$  and a finite $F\subseteq I_{1}\sqcup I_{2}$ be given. 
 Since $C^{*}_{X}(a,b)\subseteq C^{*}_{X}(a',b')$, we may apply Proposition \ref{prop: better KDT} to choose a $P\in \bC^{*}\ip{T_{i}':i\in I_{1}'\sqcup I_{2}'}^{I_{1}\sqcup I_{2}}$ so that $\|P(a',b')-(a,b)\|_{2,F;X}<\varepsilon$ and that $\|P_{i}\|_{R'}\leq R_{i}$ for all $i\in I_{1}\sqcup I_{2}$. Let $F'\subseteq I_{1}'\sqcup I_{2}'$ be a finite set so that $P_{i}\in \bC^{*}\ip{T_{i}':i\in F'}$ for all $i\in F$.  Set $L'=\|P|_{F}\|_{R',\Lip}$.  Set $\varepsilon'=\frac{\varepsilon}{L'+1}$. Since $C^{*}_{X}(a',b',c')\subseteq C^{*}_{X}(a,b,c)$, we may apply Proposition \ref{prop: better KDT} to find a $Q_{0}\in \bC^{*}\ip{T_{i'}:i'\in I'}^{I}$ so that 
 \[\|Q_{0}(a,b,c)-(a',b',c')\|_{2,F'}<\varepsilon'.\]
Note that
 \[\|P(Q_{0}|_{I_{1}'\sqcup I_{2}'}(a,b,c))-(a,b)\|_{2,F;X}<2\varepsilon.\]
 Thus 
 \[\cU_{1}=\left\{\ell\in \Sigma_{R}:\sum_{i\in F}\|P_{i}(Q_{0}|_{I_{1}'\sqcup I_{2}'}(T'))-T_{i}\|_{L^{2}(\ell)}^{2}<4\varepsilon^{2}\right\}\]
 is a weak$^{*}$-neighborhood of $L_{a,b,c;X}$. 

  Fix a weak$^{*}$-open neighborhood $\mathcal{O}$ of $L_{a',b',c';X}$, a finite $G'\subseteq I_{2}'$, and a $\theta'>0$. By Lemma \ref{lem: key lem for invariance} (\ref{item: small perturbation neighborhood}) we may find an $\eta>0$ and a finite $\widetilde{F}\subseteq I'$ so that if $y\in \prod_{i'\in I'}R_{i'}'\Ball(M)$ satisfies $\|(a',b'c')-y\|_{2,\widetilde{F};X}<\eta$, then there is a weak$^{*}$-neighborhood $\mathcal{V}$ of $L_{y;X}$ with $\Gamma^{(n)}_{R'}(\mathcal{V})\subseteq \Gamma^{(n)}_{R'}(\mathcal{O})$.
Without loss of generality we may assume that $\eta<\varepsilon'$ and $\widetilde{F}\supseteq F'\cup G'$.  Since $C^{*}_{X}(a',b',c')\subseteq C^{*}_{X}(a,b,c)$. Apply Proposition \ref{prop: better KDT} to find a $Q\in \bC^{*}\ip{T_{i}:i\in I}^{I'}$ so that $\|Q(a,b,c)-(a',b',c')\|_{2,\widetilde{F}}<\eta$ and $\|Q_{i}\|_{R,\infty}\leq R_{i}'$ for all $i\in I'$.
Since $C^{*}_{X}(b)\supseteq C^{*}_{X}(b')$ we may, and will, assume that $Q|_{I_{2}'}\in \bC^{*}\ip{T_{i}:i\in I_{2}}^{I_{2}'}$. Choose a finite $G\subseteq I_{2}$ so that $Q_{i'}\in \bC^{*}\ip{T_{i}:i\in G}$ for all $i'\in G'$. Let $L=\|Q|_{G'}\|_{R,\Lip}$. Set  $\theta=\frac{\theta'}{L+1}$. Since 
\[\|Q|_{I_{1}'\sqcup I_{2}'}(a,b,c)-Q_{0}|_{I_{1}'\sqcup I_{2}'}(a,b,c)\|_{2,F';X}<2\varepsilon',\]
we have that 
\[\cU_{2}=\left\{\ell\in \Sigma_{R}:\sum_{i'\in F'}\|Q_{i'}(T')-Q_{0,i'}(T)\|_{L^{2}(\ell)}^{2}<4\varepsilon'^{2}\right\}\]
is a weak$^{*}$-neighborhood of $L_{a,b,c;X}$. 

By Lemma \ref{lem: key lem for invariance} (\ref{item: pushforward neighborhods}) and  (\ref{item: small perturbation neighborhood 2}), we may choose a neighborhood $\mathcal{U}_{3}$ of $L_{a,b,c;X}$ so that $Q(\Gamma^{(n)}(\mathcal{U}_{3}))\subseteq \Gamma^{(n)}(\mathcal{O})$ for all $n$. Let $\mathcal{U}=\mathcal{U}_{1}\cap \mathcal{U}_{2}\cap \mathcal{U}_{3}$. 
Suppose that $f\colon \Gamma^{(n)}(\mathcal{O})\to \Delta$ is $(\varepsilon',F'|\theta',G')$-injective. 
Suppose that $x,y\in \Gamma^{(n)}(\mathcal{U})$ and $f(Q(x))=f(Q(y))$ and that $\|x-y\|_{2,G}<\theta$. Then $\|Q(x)-Q(y)\|_{2,G'}\leq L\|x-y\|_{2,G}<\theta'$. Hence by choice of $f$ we have that $\|Q(x)-Q(y)\|_{2,F'}<\varepsilon'$. Our choice of $\mathcal{U}_{1},\mathcal{U}_{2}$ thus force:
\begin{align*}
\|x-y\|_{2,F}&\leq \varepsilon+\|P(Q_{0}|_{I_{1}'\sqcup I_{2'}}(x))-P(Q_{0}|_{I_{1}'\sqcup I_{2}'}(y))\|_{2,F}\\
&\leq 4\varepsilon+L\left[\|Q_{0}|_{I_{1}'\sqcup I_{2}'}(x)-Q|_{I_{1}'\sqcup I_{2}'}(x)\|_{2,F'}+\|Q_{0}|_{I_{1}'\sqcup I_{2}'}(y)-Q|_{I_{1}'\sqcup I_{2}'}(y)\|_{2,F'} \right]\\
&+\|P(Q|_{I_{1}'\sqcup I_{2}'}(x))-P(Q|_{I_{1}'\sqcup I_{2}'}(y))\|_{2,F}\\
&\leq 4\varepsilon+5L\varepsilon'<9\varepsilon.    
\end{align*}
Thus $f\circ Q$ is $(9\varepsilon,F|G,\theta)$-injective. So we have shown that 
\[\md_{9\varepsilon,F|\theta,G}(\mathcal{U}:c)\leq \md_{\varepsilon',F'|\theta',G'}(\mathcal{O}:c').\]
A fortiori,
\[\md_{9\varepsilon,F}(a|b:c)\leq \md_{\varepsilon',F'|\theta',G'}(\mathcal{O}:c').\]
Observe that $\varepsilon',F'$ depend upon $\varepsilon,F$ but do not depend upon $\mathcal{O},\theta',G'$. Moreover, $\theta',G'$ also do not depend upon $\mathcal{O}$. Thus we may fix $\varepsilon',F',\varepsilon,F,\theta',G'$ and take the infimum over $\mathcal{O}$ to obtain
\[\md_{9\varepsilon,F}(a|b:c)\leq \md_{\varepsilon',F'|\theta',G'}(a'|b':c').\]
Since $G',\theta'$ do not depend upon $\varepsilon,F,\varepsilon',F'$ and $\varepsilon,F$ do not depend upon $\varepsilon',F',\theta',G'$ we can take the infimum over $G',\theta'$ to see  that
\[\md_{9\varepsilon,F}(a|b:c)\leq \md_{\varepsilon',F'}(a'|b':c').\]
A fortiori,
\[\md_{9\varepsilon,F}(a|b:c)\leq \delta_{MD}(a'|b':c').\]
Taking the supremum over $\varepsilon,F$ shows that 
\[\delta_{MD}(a|b:c)\leq \delta_{MD}(a'|b':c').\]

\end{proof}

That $\delta_{MD}(a|b:c)$ is an invariant of the generated algebras follows immediately. 

\begin{cor}
Let $(M,X)$ be a tracially complete $C^{*}$-algebra, and $(I_{i})_{i=1}^{3},(I_{i}')_{i=1}^{3}$ be sets. Suppose that $a\in M^{I_{1}},b\in M^{I_{2}},c\in M^{I_{3}}$, $a'\in M^{I_{1}'},b'\in M^{I_{2}'},c\in M^{I_{3}'}$ are given and that $C^{*}_{X}(b)=C^{*}_{X}(b')$, $C^{*}_{X}(a,b)=C^{*}_{X}(a',b')$, $C^{*}_{X}(a,b,c)=C^{*}_{X}(a',b',c')$. Then
\[\delta_{MD}(a|b:c)=\delta_{MD}(a'|b':c').\]
\end{cor}

Because of the above Theorem, we may unambiguously define $\delta_{MD}(Q|N:M)$ whenever $(M,X)$ is a tracially complete $C^{*}$-algebra and $Q\leq N\leq M$ as follows. Choose tuples $a,b,c$ in $M$ with $C^{*}_{X}(b)=N,C^{*}_{X}(a,b)=Q,C^{*}_{X}(a,b,c)=M$. We then set
\[\delta_{MD}(Q|N:M;X)=\delta_{MD}(a|b:c),\]
and $\delta_{MD}(N:Q;X)=\delta_{MD}(b:a)$.
This invariant enjoys the following monotonicity properties. 
\begin{cor}\label{cor:basic properties}
Let $(M;X)$ be a tracially complete $C^{*}$-algebra.
\begin{enumerate}[(i)]
    \item If $N_{1}\leq N_{2}\leq Q_{2}\leq Q_{1}\leq M_{1}\leq M_{2}$ are tracially complete sub-$C^{*}$-algebras of $M$, then
\[\delta_{MD}(Q_{2}|N_{2}:M_{2};X)\leq \delta_{MD}(Q_{1}|N_{1}:M_{1};X).\]
\item If $N\leq Q\leq M$, then
\[\delta_{MD}(Q:M;X)\leq \delta_{MD}(Q|N:M;X)+\delta_{MD}(N:M;X).\]
\end{enumerate}
\end{cor}

\begin{proof}
Now that we know $\delta_{MD}$ is independent of the choice of generators, the first item follows from Theorem \ref{thm: invariance of MD} and the second from Proposition \ref{prop: relative inequlaity}.
\end{proof}

Finally, we note the following limit property for $\delta_{MD}$ analogous to \cite[Lemma A.10]{Hayes2018}, which in many cases allows us to reduce to the case that the algebras in question are separable with respect to $\|\cdot\|_{2,X}$. 

\begin{prop}\label{prop: limiting formula for MD}
 Let $(M,X)$ be a tracially complete $C^{*}$-algebra, and $N\leq M$. Let $N_{\alpha}$ be an increasing net of tracially complete sub-$C^{*}$-algebras of $N$ with $\overline{\bigcup_{\alpha}N_{\alpha}}^{\|\cdot\|_{2,X}}=N$. Then:
 \[\delta_{MD}(N:M;X)=\sup_{\alpha}\delta_{MD}(N_{\alpha}:M;X).\]
\end{prop}

\begin{proof}
The fact that 
 \[\delta_{MD}(N:M;X)\geq \sup_{\alpha}\delta_{MD}(N_{\alpha}:M;X)\]
follows from Corollary \ref{cor:basic properties}.
For the reverse inequality, we simply take generating tuples for each $N_{\alpha}$, combined them to get a generating tuple of $N$, and use that $\delta_{MD}(N:M;X)$ does not depend upon the choice of generating tuple for $N$.

\end{proof}

\section{Sample computations of $\delta_{MD}$ and permanence properties}

\subsection{Technical preliminary results}

In general, proving lower bounds on $\dim_{\varepsilon}$ is extremely difficult. The following is the main case where we have a good lower bound.

\begin{prop}[Proposition 2.1 of \cite{Tsuka} and Proposition \ref{prop: dim TFAE}]\label{prop:Banach space computation}
Let $(V,\|\cdot\|)$ be a finite dimensional real Banach space. If $\varepsilon<1$, then $\dim_{\varepsilon}(\Ball(V),\|\cdot\|)=\dim(V)$.
\end{prop}

In our setting, we will have to lower bound $\dim_{\varepsilon}((\Ball(\bM_{n}(\bC)_{s.a.}))^{r}),\|\cdot\|_{2})$. Thus we need to take the unit ball of a Banach space in some norm, and measure $\dim_{\varepsilon}$ of this Ball with respect to a \emph{different} norm (in particular, in our setting the two norms will be not uniformly equivalent). The following classical result of Szarek \cite{SzarekAlmostEuc, STALmostEuc} provides a useful criterion in terms of volumes to say that on ``large" subspaces two norms on a finite-dimensional space are comparable. For a normed vector space $(V,\|\cdot\|)$, we use $\mathcal{O}(V,\|\cdot\|)$ for the group of linear maps $V\to V$ which preserve $\|\cdot\|$.  We use the formulation given in \cite{Pis}.

\begin{thm}[Theorem 6.1 of \cite{Pis}]\label{lemm: isomorphism on large subspace}
Let $\|\cdot\|$ be a norm on a finite-dimensional vector space $V$ and $\||\cdot\||$ a Hilbert space norm on a finite dimensional vector space $V$. Let $B$ be the  ball with respect to $\|\cdot\|$, and let $D$ be the ball with respect to $|\|\cdot\||$. Suppose that $ \|\cdot\|\leq |\|\cdot\||$, and let 
\[A=\left(\frac{\vol(D)}{\vol(B)}\right)^{1/\dim(V)}.\]
Then for all $1\leq k\leq \dim(V)-1$ there is a subspace $W\subseteq V$ with $\dim(V)=k$ and so that 
\begin{equation}\label{eqn: isomorphism on large subspace}
|\|x\||\leq \left(4\pi A\right)^{\frac{n}{n-k}}\|x\|\mbox{ for all $x\in W$.}
\end{equation}
Moreover, if $m_{k}$ is the unique $\mathcal{O}(V,\||\cdot\||)$-invariant probability measure on the space $G_{V,k}$ of $k$-dimensional subspaces of $V$, then
\[m_{k}\left(\{W:W \textnormal{ satisfies (\ref{eqn: isomorphism on large subspace})}\}\right)\geq 1-2^{-\dim(V)}.\]
\end{thm}

Note that in our setting, $\Ball(\bM_{n}(\bC),\|\cdot\|_{\infty})\subseteq \Ball(\bM_{n}(\bC),\|\cdot\|_{2})$, which is precisely the opposite inclusion in Szarek's theorem. We thus work with a dual statement, where we find quotients by ``small" subspaces on which the two norms are comparable. 
If $V$ is a real locally convex space, and $C\subseteq V$ we use $C^{o}$ for the polar of $C$ given by
\[\{\phi\in V^{*}:\phi(v)\leq 1\textnormal{ for all $v\in C$}\}.\]
Given a Banach space $(V,\|\cdot\|)$ and a close linear subspace $W\leq V$, we will abuse notation and still use $\|\cdot\|$ for the quotient norm on $V/W$. From Theorem \ref{lemm: isomorphism on large subspace}, we have the following. 

\begin{lem}\label{lemm:isomorphism on large quotient}
Let $\|\cdot\|$ be a norm on a finite-dimensional vector space $V$ and $\||\cdot\||$ a Hilbert space norm on a finite dimensional vector space $V$. Let $B$ be the  ball with respect to $\|\cdot\|$, and let $D$ be the ball with respect to $\||\cdot\||$. Suppose that $ |\|\cdot\||\leq \|\cdot\|$, and let 
\[A=\left(\frac{\vol(D^{o})}{\vol(B^{o})}\right)^{1/\dim(V)},\]
Then for all $1\leq k\leq \dim(V)-1$ there is a subspace $W\subseteq V$ with $\dim(V)=k$ and so that 
\begin{equation}\label{eqn: isomorphism on large subspace 2}
|\|x\||\leq \left(4\pi A\right)^{\frac{n}{k}}\|x\|\mbox{ for all $x\in V/W$.}
\end{equation}
Moreover, if $m_{k}$ is the unique $\mathcal{O}(V,\||\cdot\||)$-invariant probability measure on the space $G_{V,k}$ of $k$-dimensional subspaces of $V$, then
\[m_{k}\left(\{W:W \textnormal{ satisfies (\ref{eqn: isomorphism on large subspace 2})}\}\right)\geq 1-2^{-\dim(V)}.\]
\end{lem}

\begin{proof}
This is precisely the dual of Lemma \ref{lemm: isomorphism on large subspace}.
\end{proof}

In our specific case, we our concerned with comparing finite-dimensional noncommutative $L^{p}$-norms. For an integer $n\in \bN$, and $p\in [1,\infty]$, we let $S^{p}(n,\tr)$ be the space of $n\times n$ matrices equipped with the norm $\|A\|_{p}=\tr(|A|^{p})^{1/p}$. When $p=\infty$, we will use $M_{n}(\bC)$ instead of $S^{\infty}(n,\tr)$.
For $r\in \bN$, we will give $S^{p}(n,\tr)^{\oplus r}$ the $\ell^{p}$-direct sum norm of the noncommutative $L^{p}$-norm with respect to $\tr$.
For a linear subspace $W\subseteq S^{p}(n,\tr)^{\oplus r}$ we use $\|\cdot\|_{S^{p}(n,\tr)^{\oplus r}/W}$ for the quotient norm induced on $S^{p}(n,\tr)^{\oplus r}/W$ from the norm on $S^{p}(n,\tr)^{\oplus r}$.

Our ultimate goal will be to use Lemma \ref{lemm:isomorphism on large quotient} to show that $\dim_{\varepsilon}((\Ball(\bM_{n}(\bC)_{s.a.})^{r},\|\cdot\|_{2})$ is $(1-o(1))n^{2}$ as $n\to\infty$ and then $\varepsilon\to 0$.  Since we will use the norms being comparable on a quotient by a small-dimensional subspace, we will need to produce certain continuous sections of quotient maps. The following two lemmas will allow us to choose such continuous sections.

\begin{lem}\label{lemm:convex sets have small boundary}
Let $V$ be a finite-dimensional real vector space, and let $C\subseteq V$ be a compact convex set with $\dim(C)=\dim(V)$, and suppose that $0\in \Int(C)$. Then, for every linear subspace $W\leq V$ we have that $(\partial C)\cap W$ has measure zero with respect to any Lebesgue measure on $W$.

\end{lem}

\begin{proof}
Fix a Hilbert space norm  on $V$.
Choose $r>0$ so that $r\Ball(V)\subseteq C$. Given $p\in \partial C$, we know that $\{tp+(1-t)y:t\in [0,1),y\in r\Ball(V)\}$ is inside $C$. We leave it as an exercise to verify that $\{tp+(1-t)y:t\in [0,1),y\in r\Ball(v)\}$ is open, and so $\{tp+(1-t)y:t\in [0,1),y\in r\Ball(v)\}\subseteq \Int(C)$. In particular, $tp\in \Int(C)$ for all $t\in [0,1)$. Since $p$ was arbitrary, we see that $t\partial C\subseteq  C$ for all $t\in [0,1)$, i.e. $\partial C\subseteq \frac{1}{t}C$ for all $t\in [0,1)$. Thus 
\[\partial C\cap W\subseteq \left[\left(\frac{1}{t}C\right)\setminus C\right]\cap W=\left[\frac{1}{t} C\cap W\right]\setminus \left[C\cap W\right], \textnormal{ for all $t\in [0,1)$.}\]
Fix some choice $m_{W}$ of Lebesgue measure on $W$. Then,
\[m_{W}\left(\partial C\cap W\right)\leq t^{-\dim(W)}m_{W}( C\cap W)-m_{W}(C\cap W),\textnormal{ for all $t\in [0,1)$.}\]
The proof is completed by letting $t\to 1$. 
\end{proof}

\begin{lem}\label{lem: continuous sections on convex sets}
Let $V,W$ be finite-dimensional real vector space, and let $C\subseteq V$ be a compact convex set with $\dim(C)=\dim(V)$. Suppose that $T\colon V\to W$ is a surjective, linear map. Then there is a continuous section $\phi\colon T(\Int(C))\to \Int(C)$ of $T\big|_{\Int(C)}$.
\end{lem}

\begin{proof}

Since $\dim(C)=\dim(W)$, we know that $C$ is the closure of its interior, and in particular $\Int(C)$ is not empty.
Note that 
$(\Int(C)-x)\cap \ker(T)$
is a nonempty open set in $\ker(T)$ for all $x\in \Int(C)$. In particular, $(\Int(C)-x)\cap \ker(T)$ has positive measure with respect to any volume on $\ker(T)$.
Fix a choice $m_{\ker(T)}$ of Lebesgue measure on $\ker(T)$. 
Define $\phi\colon T(\Int(C))\to \Int(C)$ by
\[\phi(v)=\frac{1}{m_{\ker(T)}((\Int(C)-x)\cap \ker(T))}\int_{(\Int(C)-x)\cap\ker(T)}x+w\,dw.\]
where $x\in \Int(C)$ satisfies $T(x)=v$. This does not depend upon the choice of $x$, since $m_{\ker(T)}$ is invariant under translation by elements of $\ker(T)$.
Since $T$ is linear and $T(x+w)=v$ for all $w\in \ker(T)$, we see that 
\[T(\phi(v))=\frac{1}{m_{\ker(T)}((\Int(C)-x)\cap \ker(T))}\int_{(\Int(C)-x)\cap \ker(T)}v\,dw=v.\]

Let $v_{n}$ be a sequence in $T(\Int(C))$ with $v_{n}\to v\in T(\Int(C))$. Choose $x\in \Int(C)$ with $T(x)=v$. Since $T$ is a quotient map in the topological sense, we may choose $x_{n}\in V$ with $T(x_{n})=v_{n}$ and $x_{n}\to x$. For all large $n$ we have that $x_{n}\in \Int(C)$. Note that
\[\vol((\Int(C)-x_{n})\cap \ker(T))=\int_{\ker(T)}1_{\Int(C)}(x_{n}+w)\,dw.\]
We have that $1_{\Int(C)}(x_{n}+w)\to 1_{\Int(C)}(x+w)$ provided $w\notin \partial [(\Int(C)-x)]\cap \ker(T)$. 
Since $(\Int(C)-x)\cap \ker(T)$ is a convex set of dimension equal to dimension of the $\ker(T)$ with $0$ in its interior, we know by Lemma \ref{lemm:convex sets have small boundary} that $\partial(\Int(C)-x)\cap \ker(T)$ has measure zero. 
Thus $1_{\Int(C)}(x_{n}+w)\to 1_{\Int(C)}(x+w)$ almost everywhere and the dominated convergence theorem implies that \[\vol((\Int(C)-x_{n})\cap \ker(T))\to \vol((\Int(C)-x)\cap \ker(T)).\]
Similarly,
\[\int_{\ker(T)}1_{\Int(C)}(x_{n}+w)(x_{n}+w)\,dw\to \int_{\ker(T)}1_{\Int(C)}(x+w)(x+w)\,dw.\]
Thus $\phi(v_{n})\to \phi(v)$. This shows that $\phi$ is continuous.

\end{proof}

Finally, we can put all the above results together to produce a lower bound on $\dim_{\varepsilon}((\Ball(\bM_{n}(\bC)_{s.a.})^{r},\|\cdot\|_{2})$.

\begin{thm}\label{thm:computation of dim for op norm vs 2-norm}
Fix $r\in \bN$. Then
\[\sup_{\varepsilon>0}\liminf_{n\to\infty}\frac{1}{n^{2}}\dim_{\varepsilon}(\Ball((\bM_{n}(\bC)_{s.a.})^{r}),\|\cdot\|_{2})\]

\end{thm}

\begin{proof}
Let
\[A=\sup_{n}\left(\frac{\vol(\Ball(S^{1}(n,\tr)_{s.a.}^{\oplus r}))}{\vol(\Ball(S^{2}(n,\tr)_{s.a.}^{\oplus r}))}\right)^{1/rn^{2}}.\]
By \cite[Corollaire 8]{SPVolumes} we know that $A<\infty$.
Fix $\kappa>0$. 
Let $\varepsilon\in (0,1)$ be sufficiently small in a manner depending upon $\kappa$ which we will determine later. 
If we give $M_{n}(\bC)_{s.a.}^{\oplus r}$ the inner product coming from the norm on $S^{2}(n,\tr)_{s.a.}^{\oplus r}$, then this allows us to identify $M_{n}(\bC)_{s.a.}^{\oplus r}$  with its dual. 
Under this identification, the polar of $\Ball(M_{n}(\bC)_{s.a.})^{r}$ is $\Ball(S^{1}(n,\tr)_{s.a.}^{\oplus r})$.
So by Lemma \ref{lemm:isomorphism on large quotient}, for all sufficiently large $n$ we may choose a subspace $W\subseteq S^{1}(n,\tr)_{s.a.}^{\oplus r}$ with $\dim(W)=\lceil{\kappa n^{2}r\rceil}$ and such that 
\[\|\cdot\|_{M_{n}(\bC)^{\oplus r}/W}\leq \left(4\pi A\right)^{\frac{2}{\kappa}}\|\cdot\|_{S^{2}(n,\tr)^{\oplus r}/W}.\]
Let $\pi\colon M_{n}(\bC)_{s.a.}\to M_{n}(\bC)_{s.a.}/W$ be the quotient map.
Since we give $M_{n}(\bC)_{s.a.}^{\oplus r}/W$ the quotient norm, we know that $\pi$ maps the open unit ball of $M_{n}(\bC)_{s.a.}^{\oplus r}$ onto the open unit ball of $M_{n}(\bC)_{s.a.}^{\oplus r}/W$. By compactness, we in fact have that $\pi(\Ball(M_{n}(\bC)_{s.a.})^{r})=\Ball(M_{n}(\bC)_{s.a.}^{\oplus r}/W)$. By Lemma \ref{lem: continuous sections on convex sets}, we may find a continuous section $\phi\colon \Int(\Ball(M_{n}(\bC)_{s.a.}^{\oplus r}/W))\to \Int(\Ball(M_{n}(\bC)_{s.a.})^{r})$ of $\pi\big|_{\Int(\Ball(M_{n}(\bC)_{s.a.})^{r})}$.

Suppose that $Y$ is a $d$-dimensional compact Hausdorff space and $f\colon M_{n}(\bC)_{s.a.}^{\oplus r}\to Y$ is $\varepsilon$-injective with respect to $\|\cdot\|_{2}$. Let $g=f\circ \phi$. If $g(x)=g(y)$, then we have that $\|\phi(x)-\phi(y)\|_{2}<\varepsilon.$ Since $\pi$ is contractive with respect to the quotient norm, we have that 
\[\|x-y\|_{S^{2}(n,\tr)^{\oplus r}/W}=\|\pi(\phi(x)-\phi(y))\|_{S^{2}(n,\tr)^{\oplus r}/W}<\varepsilon.\]
By our choice of $W$,
\[\|x-y\|_{M_{n}(\bC)_{s.a.}^{\oplus r}/W}\leq \left(4\pi A\right)^{\frac{2}{\kappa}}\|x-y\|_{S^{2}(n,\tr)^{\oplus r}/W}<\left(4\pi A\right)^{\frac{2}{\kappa}}\varepsilon.\]
If $\varepsilon<(4\pi A)^{-\frac{2}{\kappa}}$, then $\Int(\Ball(M_{n}(\bC)_{s.a.}^{\oplus r}/W))$  contains a closed ball of radius $\rho\in \left(\left(4\pi A\right)^{\frac{2}{\kappa}}\varepsilon,1\right)$ with respect to $\|\cdot\|_{M_{n}(\bC)_{s.a.}^{\oplus r}/W}$ and by Proposition \ref{prop:Banach space computation} this forces $\dim(Y)\geq n^{2}r-\dim(W)\geq n^{2}r(1-\kappa)$. Thus 
\[\liminf_{n\to\infty}\frac{1}{n^{2}}\dim_{\varepsilon}((\Ball(\bM_{n}(\bC)_{s.a.})^{r}),\|\cdot\|_{2})\geq (1-\kappa), \textnormal{ if $\varepsilon<(4\pi A)^{-\frac{2}{\kappa}}$.}\]
This implies that 
\[\sup_{\varepsilon>0}\liminf_{n\to\infty}\frac{1}{n^{2}}\dim_{\varepsilon}((\Ball(\bM_{n}(\bC)_{s.a.})^{r}),\|\cdot\|_{2})\geq r(1-\kappa).\]
Since $\kappa>0$ was arbitrary, the proof is completed by letting $\kappa\to 0$.

\end{proof}

In many cases, it is more useful to use unitaries instead of self-adjoints. The following lemma allows us to show that the two settings are comparable.

\begin{lem}\label{lem: coercive unitary}
Suppose $A,B\in \bM_{n}(\bC)_{s.a.}$ and $\|A\|,\|B\|\leq 1/4$. Then $\frac{1}{2}\|A-B\|_{2}\leq \|e^{iA}-e^{iB}\|_{2}\leq \|A-B\|_{2}$. 

\end{lem}

\begin{proof}
Given $C\in \bM_{n}(\bC)_{s.a.}$, we define $T_{C}\colon \bM_{n}(\bC)_{s.a.}\to \bM_{n}(\bC)$ by 
\[T_{C}(B)=i\int_{0}^{1}e^{it C}Be^{i(1-t)C}\,dt.\]
Define $\gamma\colon [0,1]\to \bM_{n}(\bC)_{s.a}$ by $\gamma(t)=(1-t)A+tB$. Since $T_{C}$ is the differential of $C\mapsto e^{iC}$ at $C$ (see \cite[Chapter, Section 14 on perturbation theory]{Bellman}), we have $e^{iB}-e^{iA}=\int_{0}^{1}T_{\gamma(t)}(B-A)\,dt$. So:
\[\|e^{iB}-e^{iA}\|_{2}\leq \int_{0}^{1}\|T_{\gamma(t)}(B-A)\|_{2}\,dt\leq \|B-A\|_{2},\]
the last estimate following as $T_{C}$ is $\|\cdot\|_{2}$-$\|\cdot\|_{2}$ for contractive for all $C\in \bM_{n}(\bC)_{s.a.}$. 
So:
\[\|e^{iB}-e^{iA}-i(B-A)\|_{2}\leq \int_{0}^{1}\|(T_{\gamma(t)}-\id)(B-A)\|_{2}\,dt.\]
Note that:
\[\|T_{\gamma(t)}-\id\|_{B(S^{2}(N,\tr))}\leq 2\sup_{X\in \bM_{n}(\bC)_{s.a},\|X\|\leq 1/4}\|e^{iX}-1\|\leq 2\sup_{x\in [-1/4,1/4]}|e^{ix}-1|\leq \frac{1}{2}.\]
Thus:
\[\|e^{iA}-e^{iB}\|_{2}\geq \frac{1}{2}\|A-B\|_{2}.\]

\end{proof}

From Lemma \ref{lem: coercive unitary} and Theorem \ref{thm:computation of dim for op norm vs 2-norm} we deduce the following. 

\begin{cor}\label{cor: epsilon dim of unitaries}
Fix an integer $r\in \bN$. Then
\[\sup_{\varepsilon>0}\liminf_{n\to\infty}\frac{1}{n^{2}}\dim_{\varepsilon}(\cU(\bM_{n}(\bC))^{r},\|\cdot\|_{2})=r.\]

\end{cor}

It an also be helpful to use projections instead of arbitrray self-adjoints. As with the case of unitaries, this will follow from once we prove the following Lemma.

\begin{lem}\label{lem: case of projections}
Let $P\in \bM_{n}(\bC)$ with $\tr(P)=k/n$, for some $k\in \bN$. Suppose that $X,Y\in \bM_{n}(\bC)_{s.a.}$ and that $(1-P)X(1-P)=PXP=0=PYP=(1-P)Y(1-P)$ have $\|X\|,\|Y\|\leq 1/8$ then
\[\frac{1}{2}\|Y-X\|_{2}\leq \|e^{iY}Pe^{-iY}-e^{iX}Pe^{-iX}\|_{2}\leq 2\|Y-X\|_{2}.\]
\end{lem}

\begin{proof}
Let $T_{C}$ be given as in the proof of Lemma \ref{lem: coercive unitary}
Define $\gamma(t)=(1-t)X+tY$, and $\sigma(t)=e^{i\gamma(t)}Pe^{-i\gamma(t)}$. Then:
\[\sigma'(t)=T_{\gamma(t)}(Y-X))Pe^{-i\gamma(t)}-e^{i\gamma(t)}PT_{\gamma(t)}(Y-X).\]
So $\|\sigma'(t)\|_{2}\leq 2\|Y-X\|_{2}$, by the same estimates as in the proof of Lemma \ref{lem: coercive unitary}.

Since $X,Y$ are self-adjoint and $(1-P)X(1-P)=PXP=0=PYP=(1-P)Y(1-P)$, we have that $Y-X=P(Y-X)+(Y-X)P$. 
Hence 
\begin{align*}
 \|\sigma'(t)-i(Y-X)\|_{2}&\leq \left(2\|(T_{\gamma(t)}-id)\|_{B(S^{2}(N,\tr)}+2\|e^{i\gamma(t)}-1\|\right)\|Y-X\|_{2},\\
 &\leq \frac{1}{2}\|Y-X\|_{2},
\end{align*}
where in the last step we use that $\|\gamma(t)\|\leq 1/8$ and the same estimates as in Lemma \ref{lem: coercive unitary}.
Thus
\[\|e^{iY}Pe^{-iY}-e^{iX}Pe^{-iX}-i(Y-X)\|_{2}\leq \int_{0}^{1}\|\sigma'(t)\|_{2}\,dt\leq \frac{1}{2}\|Y-X\|_{2}.\]

\end{proof}

Again this allows us to deduce appropriate estimates for $\dim_{\varepsilon}$ for projections.
For integers $1\leq k\leq n$, we let 
\[\Gr(k,n)=\left\{P\in \Proj(\bM_{n}(\bC)):\tr(P)=\frac{k}{n}\right\}\]
\begin{cor}\label{cor: projection lower bound}
Fix $r\in \bN$, and for $n\in\bN$ and $j=1,\cdots,r$ let $k_{j,n}\in \bN$ be given. Suppose that 
\[\lim_{n\to\infty}\frac{k_{j,n}}{n}=\lambda_{j}\]
exists for all $j=1,\cdots,r$.
Then
\[\sup_{\varepsilon>0}\liminf_{n\to\infty}\frac{1}{n^{2}}\dim_{\varepsilon}\left(\prod_{j=1}^{r}\Gr(k_{j,n},n),\|\cdot\|_{2}\right)= 2\sum_{j=1}^{r}\lambda_{j}(1-\lambda_{j}).\]

\end{cor}
Note that Lemma \ref{lem: case of projections} gives us the lower bound, where the upper bound follows from the fact that $\Gr(k,n)$ has dimension $2k(n-k)$. Moreover, the conclusion of Corollary \ref{cor: projection lower bound} holds with $\liminf_{n\to\infty}$ replace with $\limsup_{n\to\infty}$. 

\subsection{Cases where $\delta_{MD}\leq 1$}

In this section, we give many examples of tracially complete $C^{*}$-algebras $(M,X)$ with $\delta_{MD}(M;X)\leq 1$. The main point is to prove more that if $a\in M$ has $C^{*}_{X}(a)=M$, then $\delta_{MD}(M;X)$ is at most the microstates free entropy dimension of $a$ with respect to $X$. 
We recall the definition of  microstates free entropy dimension with respect to a family of traces here. 

For a space $X$ equipped with a pseudometric $\rho$, we say that $A\subseteq X$ is $\varepsilon$-dense if for all $x\in X$, there is an $a\in A$ with $\rho(a,x)<\varepsilon$. We let $K_{\varepsilon}(A,\rho)$ be the minimum cardinality of an $\varepsilon$-dense subset of $X$. 
For an index set $I$, a precompact $\Omega\subseteq \bM_{n}(\bC)^{I}$, an $\varepsilon>0$, and a finite $F\subseteq I$, we use $K_{\varepsilon,F}(\Omega)$ for $K_{\varepsilon}(\Omega,\|\cdot\|_{2,F})$. 
We say that $\Omega_{0}\subseteq\Omega$ is $(\varepsilon,F)$-orbit dense if for all $X\in \Omega$, there is a $Y\in \Omega_{0}$ and a $U\in \cU(\bM_{n}(\bC))$ so that $\|X-UYU^{*}\|_{2,F}<\varepsilon$. 
We let $K_{\varepsilon,F}^{orb}(\Omega)$ be the minimal cardinality of an $(\varepsilon,F)$-orbit dense subset of $\Omega_{0}$.

\begin{defn}
Let $(M,X)$ be a tracially complete $C^{*}$-algebra. Let $I,J$ be index sets and $a\in M^{I},b\in M^{J}$. Let $R\in [0,\infty)^{I\sqcup J}$ be a cutoff function for $a,b$. For a weak$^{*}$-neighborhood $\mathcal{O}$ of $L_{a,b:X}$, an $\varepsilon>0$, and a $k\in \bN$,  we iteratively define:
\[\delta_{\varepsilon,F,R}(\mathcal{O})=\limsup_{k\to\infty}\frac{\log K_{\varepsilon,F}(\Gamma^{(k)}_{R}(\mathcal{O}))}{k^{2}|\log(\varepsilon)|},\]
\[\delta_{\varepsilon,F,R}(a,b;X)=\inf_{\mathcal{O}\supseteq L_{a,b;X}}\delta_{\varepsilon,F,R}(\mathcal{O}),\]
\[\delta_{0}(a|_{F}:b;X)=\limsup_{\varepsilon\to 0}\delta_{\varepsilon,F,R}(\mathcal{O}),\]
\[\delta_{0}(a:b;X)=\sup_{F}\delta_{0}(a|_{F}:b;X),\]
where the infimum is over all weak$^{*}$-neighborhoods $\mathcal{O}$ of $L_{a,b;X}$, and the supremum is over all finite subsets $F$ of $I$. 
Standard methods show that this quantity does not depend upon the choice of cutoff function $R,$ and this justifies dropping the $R$ in the last two equations. We call $\delta_{0}(a:b;X)$ the \textbf{microstates free entropy dimension of $a$ in the presence of $b$ with respect to $X$.} If $X$ is clear from context, we usually just call this the \emph{microstates free entropy dimension of $a$ in the presence of $b$.}
\end{defn}

In the case $X=\{\tau\}$ for some $\tau\in \cT(M)$, so that $(M,\tau)$ is a tracial von Neumann algebra, \cite[Corollary 2.4]{JungLemma} guarantees this agrees with Voiculescu's free entropy dimension defined in \cite[Section 6]{VoiculescuFreeEntropy2}, \cite[Section6]{Voiculescu1996}.

We also recall the $1$-bounded entropy of $h(a:b;X)$ defined in \cite{macmahon20261boundedentropycalgebras}. For $\varepsilon>0$,$F\subseteq I$ finite and a weak$^{*}$-neighborhood $\mathcal{O}$ of $L_{a,b;X}$ we iteratively define
\[h_{\varepsilon,F}(\mathcal{O};X)=\limsup_{k\to\infty}\frac{1}{k^{2}}\log K_{\varepsilon}^{orb}(\Gamma^{(k)}_{R}(\mathcal{O})),\]
\[h_{\varepsilon,F}(a,b;X)=\inf_{\mathcal{O}}h_{\varepsilon,F}(\mathcal{O};X),\]
\[h(a:b;X)=\sup_{\varepsilon>0,F\subseteq I}h_{\varepsilon,F}(a,b;X),\]
where the infimum is over all weak$^{*}$-neighborhoods $\mathcal{O}$ of $L_{a,b;X}$, and the supremum is over all $\varepsilon>0$ and all finite $F\subseteq I$. By \cite{macmahon20261boundedentropycalgebras}, this is an actually an invariant of the generated tracially complete $C^{*}$-algebra, and so we can set 
\[h(C^{*}_{X}(a):C^{*}_{X}(b);X)=h(a:b;X).\]

For $k\in \bN$, we use $u_{k}$ for the uniform measure on $[k]$. The following proposition will quickly allow us to show the microstates free entropy dimension of $a$ in the presence of $b$ is an upper bound for $\delta_{MD}(C^{*}_{X}(a):C^{*}_{X}(b)).$

\begin{prop}\label{prop: width and packing dim comparison}
Let $\cH_{n}$ be a sequence of finite-dimensional Hilbert spaces with $\dim(\cH_{n})$ increasing and going to infinity, and let $C\in [0,+\infty)$ be fixed. Suppose that $\Gamma_{n}\subseteq C\Ball(\cH_{n})$ are given. Then, for any $t>0$ we have:
\[\limsup_{n\to\infty}\frac{\dim_{t}(\Gamma_{n},\|\cdot\|_{\cH_{n}})}{\dim(\cH_{n})}\leq \liminf_{\varepsilon\to 0}\limsup_{n\to\infty}\frac{\log K_{\varepsilon}(\Gamma_{n},\|\cdot\|_{\cH_{n}})}{\dim(\cH_{n})|\log(\varepsilon)|}.\]

\end{prop}

\begin{proof}
Set $k_{n}=\dim(\cH_{n})$. We may, and will, take $\cH_{n}=\ell^{2}(k_{n},u_{k_{n}})$, and $C=1$. Fix $\eta>0$, and let $f\in C^{\infty}(\bR)$ be nonnegative function satisfying $1_{[-\frac{1}{\eta},\frac{1}{\eta}]}\leq f\leq 1_{[-\frac{2}{\eta},\frac{2}{\eta}]}$. Define $h_{i}\in C^{\infty}(\bR),i=1,2$ by $h_{1}(x)=f(x)x$, and $h_{2}(x)=(1-f(x))x$. Define continuous maps $\Phi_{i,n}\colon \Ball(\ell^{2}(k_{n},u_{k_{n}}))\to \bR^{k_{n}}$ by $\Phi_{i,n}(a)(j)=h_{i}(a(j))$. Set $\Xi_{n}=\Phi_{1,n}(\Gamma_{n})$. Letting $L=\|h_{1}'\|_{\infty}$, we see that $\Phi_{1,n}$ is an $L$-Lipschitz map if the domain and codomain are given the $\|\cdot\|_{\ell^{2}(u_{k_{n}})}$ norm, and that $\Xi_{n}\subseteq [-\frac{2}{\eta},\frac{2}{\eta}]^{k_{n}}$. Since $\Phi_{1,n}$ is $L$-Lipschitz, it follows that $K_{\varepsilon}(\Xi_{n},\|\cdot\|_{\ell^{2}(u_{k_{n}})})\leq K_{\varepsilon/L}(\Gamma_{n},\|\cdot\|_{\ell^{2}(u_{k_{n}})})$. Thus 
\[\liminf_{\kappa\to 0}\limsup_{n\to\infty}\frac{1}{k_{n}}K_{\kappa}(\Xi_{n},\|\cdot\|_{\ell^{2}(u_{k_{n}})})\leq \liminf_{\varepsilon\to 0}\limsup_{n\to\infty}\frac{1}{k_{n}}K_{\varepsilon}(\Gamma_{n},\|\cdot\|_{\ell^{2}(u_{k_{n}})}).\]
We may thus apply \cite[Lemma 2.18]{Hayes4} with $X=[-\frac{2}{\eta},\frac{2}{\eta}]$ to see that 
\[\limsup_{n\to\infty}\frac{1}{k_{n}}\dim_{t}(\Xi_{n},\|\cdot\|_{\ell^{2}(u_{k_{n})}})\leq \limsup_{n\to\infty}\frac{1}{k_{n}}\dim_{t}(\Xi_{n},\|\cdot\|_{\ell^{\infty}(k_{n})}) \leq \liminf_{\varepsilon\to 0}\limsup_{n\to\infty}\frac{1}{k_{n}}K_{\varepsilon}(\Gamma_{n},\|\cdot\|_{\ell^{2}(u_{k_{n}})})\]
Note that if $f\colon \Xi_{n}\to \Delta$ is $t$-injective, then so is the map $(f\circ \Phi_{1,n})\times \Phi_{2,n}\colon \Gamma_{n}\to \Delta\times \Phi_{2,n}(\Ball(\ell^{2}(k_{n},u_{k_{n}})))$. 

Moreover 
\[\Phi_{2,n}(\Ball(\ell^{2}(k_{n},u_{k_{n}}))\subseteq \bigcup_{A\subseteq [k_{n}]:|A|\leq \eta^{2}k_{n}}\{a\in \Ball(\ell^{2}(k_{n},u_{k_{n}})):a=1_{A}a\},\]
and by \cite[Page 30 and Theorem V.8]{DimTheory} the right-hand side has covering dimension at most $\eta^{2}k_{n}^{2}$. Altogether, we have shown that 
\[\limsup_{n\to\infty}\frac{1}{k_{n}}\dim_{t}(\Gamma_{n},\|\cdot\|_{\ell^{2}(u_{k_{n}})}\leq \eta^{2}+\liminf_{\varepsilon\to 0}\limsup_{n\to\infty}\frac{1}{k_{n}}K_{\varepsilon}(\Gamma_{n},\|\cdot\|_{\ell^{2}(u_{k_{n}})}).\]
The proof is now completed by sending $\eta\to 0$.

\end{proof}

From Proposition \ref{prop: width and packing dim comparison}, we immediately deduce the following.

\begin{cor}\label{cor: delta MD and MFED}
Let $(M,X)$ be a tracially complete $C^{*}$-algebra, and let $a\in M^{I},b\in M^{J}$ be given for some index sets $I,J$. Then:
\[\delta_{MD}(C^{*}_{X}(a):C^{*}_{X}(b))\leq \delta_{0}(a:b;X).\]
\end{cor}

Analogous to Jung's theorem that strong $1$-boundedness implies microstates free entropy dimension at most $1$ with respect to every set of generators, we have the following, which is the main way we will produce examples of tracially complete $C^{*}$-algebras $(M,X)$ which have $\delta_{MD}(M;X)\leq 1$.

\begin{prop}\label{prop: delta MD and MFED}
Let $(M,X)$ be a tracially complete $C^{*}$-algebra. Let $I,J$ be index sets and $a\in M^{I},b\in M^{J}$. Suppose that 
\[\sup_{\tau\in X}h(W^{*}(\ell_{a,\tau}):W^{*}(\ell_{b,\tau});\tau)<+\infty.\]
Then $\delta_{0}(a:b;X)\leq 1$.
\end{prop}

\begin{proof}
By the variational principle, \cite[Theorem 6.1]{macmahon20261boundedentropycalgebras} we have that 
\[h(C^{*}_{X}(a):C^{*}_{X}(b);X)<+\infty.\]
Fix a cutoff function $R\in [0,\infty)^{I\sqcup J}.$   
 Let $\varepsilon>0$ and a finite $F\subseteq I$ be given. Fix a weak$^{*}$-neighborhood $\mathcal{O}$  of $L_{a,b;X}$. Let $R_{F}=\max_{i\in F}R_{i}$. Let $\Xi\subseteq \Gamma^{(k)}_{R}(\mathcal{O})$ be $(\varepsilon,F)$-orbit dense. By \cite[Theorem 7]{Szarek}, we may find $\Omega\subseteq \cU(\bM_{k}(\bC))$ which is $\frac{\varepsilon}{R_{F}}$-dense in $\|\cdot\|_{\infty}$ and has $|\Omega|\leq \left(\frac{CR_{F}}{\varepsilon}\right)^{k^{2}}$ for some universal constant $C$. Then $\Gamma^{(k)}_{R}(\mathcal{O})$ is contained in the $3\varepsilon$-neighborhood of $\bigcup_{U\in \Omega}U\Xi U^{*}$ with respect to $\|\cdot\|_{2,F}$.
Thus 
\[\delta_{6\varepsilon,F}(\mathcal{O})\leq \frac{\log(CR_{F}/\varepsilon)}{|\log(\varepsilon)|}+\frac{h_{\varepsilon,F}(\mathcal{O};X)}{|\log(\varepsilon)|}.\]
Taking the infimum over $\mathcal{O}$ we find that
\[\delta_{6\varepsilon,F}(a,b;X)\leq \frac{\log(CR_{F}/\varepsilon)}{|\log(\varepsilon)|}+\frac{h_{\varepsilon,F}(a,b;X)}{|\log(\varepsilon)|}.\]
Taking the limit suprema as $\varepsilon\to 0$ and then the supremum over $F$ completes the proof.

\end{proof}

Proposition \ref{prop: delta MD and MFED} and Corollary \ref{cor: delta MD and MFED} allow us to give many examples of tracially complete $C^{*}$-algebras which $\delta_{MD}\leq 1$.
For terminology, given a tracial von Neumann algebra $(M,\tau)$ an $x\in M^{r}$ we say that $x$ is a \emph{Kazhdan tuple with constant $C>0$} if for any $M$-$M$ bimodule $\cH$, and any $\xi\in \cH$ we have 
\[\|\xi-P_{\Cent_{M}(\cH)}(\xi)\|\leq C\left(\sum_{j=1}^{r}\|[\xi,x_{j}]\|^{2}\right)^{1/2},\]
where $\Cent_{M}(\cH)=\{\xi\in \cH:x\xi=\xi x \text{ for all $x\in M$}\}$. We call $(x,C)$ a \emph{Kazhdan pair} for $M$.

If $G$ is a group and $S$ is a finite subset for $G$, we say that $S$ is a Kazhdan tuple with constant $C>0$ if for every unitary representation $\pi\colon G\to \cU(\cH)$, and for every $\xi\in \cH$ we have that 
\[\|\xi-P_{\Fix(\pi)}(\xi)\|\leq C\left(\sum_{s\in S}\|\xi-\pi(s)\xi\|^{2}\right)^{1/2},\]
where $\Fix(\pi)=\bigcap_{g\in G}\ker(\pi(g)-1).$ We call $(S,C)$ a \emph{Kazhdan pair} for $G$. We say $G$ has \emph{Property (T)} if it has a Kazhdan pair. 

\begin{cor}\label{cor: basic examples of vanishing}

\begin{enumerate}[(i)]
\item \label{item: 1 bdd entropy 0} If $h(M_{\tau},\tau)\leq 0$ for every $\tau\in X$, then $\delta_{MD}(M;X)\leq 1$. In particular, if for every $\tau\in X$ we have that $M_{\tau}$ satisfies one of the following conditions:
\begin{itemize}
    \item If $M_{\tau}$ is hyperfinite,
    \item If $M_{\tau}$ has a diffuse, regular, hyperfinite subalgebra,
    \item $M_{\tau}$ is generated by two commuting diffuse subalgebras,
    \item $M_{\tau}$ has diffuse central sequence algebra,
    \item $M_{\tau}$ has a single sequential commutation orbit in the sense of  \cite{SeqCommutation},
\end{itemize}
then $\delta_{MD}(M;X)\leq 1$. 
\item \label{item: uniformly strongly 1-bounded} If, more generally, $\sup_{\tau\in X}h(M_{\tau},\tau)<+\infty$, then $\delta_{MD}(M;X)\leq 1$.  
\item \label{item: M property (T)} 
If $(M,X)$ satisfies the following uniform Property (T) assumption: there is an $r\in \bN$ and a $C>0$ so that for every $\tau\in X$, there is an $x\in \Ball(M)^{r}$ which is a Kazhdan tuple with constant $C>0$,  then $(M;X)$ satisfies the hypotheses of (\ref{item: uniformly strongly 1-bounded}), and thus $\delta_{MD}(M;X)\leq 1$. 

In particular, this holds if $(M,X)$ is the tracial completion $(C^{*}(G),K)$ where $G$ is any Property (T) group and $K$ is a compact, convex subset of $\cT(C^{*}(G))$.

\end{enumerate}

\end{cor}

\begin{proof}

(\ref{item: uniformly strongly 1-bounded}): This follows from  Proposition \ref{prop: delta MD and MFED} and Corollary and \ref{cor: delta MD and MFED}.

The fact that each class listed in (\ref{item: 1 bdd entropy 0}) has $h(M_{\tau},\tau)\leq 0$ follows from the results in \cite{Jung2007, Hayes2018} (see e.g. \cite[Section 1.2]{FreePinsker} for a detailed discussion of the first four items, and see \cite[Corollary 4.8]{Hayes2018} or \cite[Fact 2.9]{SeqCommutation} for the fifth).

(\ref{item: M property (T)}): The proof of \cite[Theorem 1.1]{HJKEPropT}, shows that there is a function $f\colon (0,\infty)\times \bN\to [0,+\infty)$ if $(M,\tau)$ is a tracial von Neumann algebra generated by a Kazhdan tuple $x\in \Ball(M)^{r}$, with constant $C>0$, then 
\[h(M,\tau)\leq f(C,r).\]
Thus our hypothesis on $(M,X)$ imply that the hypotheses of (\ref{item: uniformly strongly 1-bounded}) hold. Hence, $\delta_{MD}(M;X)\leq 1$. 

To see that these hypothesis hold in the case that $(M;X)$ is the tracial completion of $(C^{*}(G),K)$ with $G$ property (T) and $K\subseteq \cT(C^{*}(G))$ compact and convex, fix a finite generating set $S$ for $G$. Suppose that $\tau\in K$, that $M_{\tau}=\overline{\pi_{\tau}(C^{*}(G))}^{SOT}$, let $\cH$ be an $M_{\tau}$-$M_{\tau}$ bimodule. Define $\phi\colon G\to \cU(H)$ by $\phi(g)\xi=\pi_{\tau}(g)\xi\pi_{\tau}(g)^{-1}$. Then, the $M_{\tau}$-central vectors in $\cH$ are precisely $\Fix(\phi)$, and since $G$ has Property (T), there is some uniform $C>0$ so that 
\[||\xi-P_{\Fix(\phi)}(\xi)\|\leq C\left(\sum_{s\in S}\|\phi(s)\xi-\xi\|_{2}^{2}\right)^{1/2}=C\left(\sum_{s\in S}\|[\pi_{\tau}(s),\xi]\|_{2}^{2}\right)^{1/2}.\]

\end{proof}

\section{$\delta_{MD}$ as an obstruction to finite generation}

\subsection{$\delta_{MD}>1$ obstructs stable single generation for tracial von Neumann algebras}

Since $\delta_{MD}$ is an invariant of tracially complete $C^{*}$-algebras, it is in particular an invariant of tracial von Neumann algebras. It is thus natural to ask what values it can take on tracial von Neumann algebras. In this section, we show that any nontrivial computation of $\delta_{MD}(M,\tau)$ for a tracial von Neumann algebra $(M,\tau)$ would solve the generator problem in the negative. Namely, we show that following.

\begin{thm}\label{thm: nontrivial computations for tracial vNas are hard}
Suppose that there exists a tracial von Neumann $(M,\tau)$ with $\delta_{MD}(M,\tau)>1$, then there is a $\textrm{II}_{1}$-factor $N$ which is not singly generated. If, in addition, we assume that $M_{*}$ is separable, there is a nonzero projection $p\in M$ so that $pMp$ is a $\textrm{II}_{1}$-factor which is not singly generated.  
\end{thm}

Given a tracial von Neumann algebra $(M,\tau)$ an index set $I$ and an $a\in M^{I}$, we say that  a sequence $(A^{(k)})\in \prod_{k}\bM_{k}(\bC)^{I}$ is a \emph{microstates sequence} for $a$ if:
\begin{enumerate}
    \item for all $i\in I$ we have $\sup_{k}\|A^{(k)}_{i}\|<+\infty$,
    \item $\ell_{A^{(k)}}\to \ell_{a}$ in the weak$^{*}$-topology.
\end{enumerate}
We start with the following observation.

\begin{lem}\label{lem: small neigbhorhoods concentrate on orbits}
Let $(M,\tau)$ be a tracial von Neumann algebra, $p\in \Proj(M)$, and $b\in pMp$, and $v_{1},\cdots,v_{n}$ partial isometries in $M$ with $v_{1}=p,$ $v_{i}^{*}v_{i}=p$ and $\sum_{i}v_{i}v_{i}^{*}=1$. Fix $R\geq \max(\|b\|,1)$, and let $(V_{i}^{(k)})_{i=1}^{n}\in \bM_{k}(\bC)^{n}$
be a microstates sequence with $\|V_{i}^{(k)}\|\leq 1$ and so that $P_{k}=V_{1}^{(k)}$ is a projection for all $k$. Set $a=(b,v_{1},\cdots,v_{n})$.  Then for every $\varepsilon>0$, there is a neighborhood $\mathcal{O}$ of $\ell_{a}$ so that for every $k\in \bN$ and every $A\in \Gamma^{(k)}_{R}(\mathcal{O})$, there is a unitary $U\in \cU(\bM_{k}(\bC))$ and an $X\in P_{k}R\Ball(\bM_{k}(\bC))P_{k}$ with
\[\|A-U(X,P_{k},(V_{i}^{(k)})_{i=2}^{n})U^{*}\|_{2}<\varepsilon.\]
\end{lem}

\begin{proof}
Set $N=W^{*}(a)$.
Suppose the claim is false, then we can find an $\varepsilon>0$, a decreasing sequence $\mathcal{O}_{n}$ of neighborhoods of $\ell_{a}$ in $\Sigma_{R,n+1}$ with $\bigcap_{n}\mathcal{O}_{n}=\ell_{a}$, a sequence $(k_{n})_{n}$ of integers and $A^{n}\in \Gamma^{(k_{n})}(\mathcal{O}_{n})$ so that 
\[\|A^{(n)}-U(X,P_{k_{n}},(V_{i}^{(k_{n})})_{i=2}^{n})U^{*}\|_{2}\geq \varepsilon, \textnormal{ for all $X\in P_{k_{n}}R\Ball(\bM_{k}(\bC))P_{k_{n}}, U\in \cU(\bM_{k_{n}}(\bC))$.}\]
Fix a free ultrafilter $\omega\in\beta\bN$, and set $\cM=\prod_{n\to\omega}\bM_{k_{n}}(\bC)$. Since $\ell_{A^{(n)}}\to\ell_{a}$, there is an embedding $\Theta\colon N\to \cM$ so that 
\[\Theta(b)=(A_{1}^{(n)})_{n\to\omega},\]
\[\Theta(v_{j})=(V_{j}^{(k_{n})})_{n\to\omega}.\]
Since $W^{*}(v_{1},\cdots,v_{n})\cong \bM_{n}(\bC)$ any two embeddings of $W^{*}(v_{1},\cdots,v_{n})$ are unitarily conjugate. Thus we may find a unitary $u\in \cM$ so that 
\[\Theta(v_{j})=u(V_{j}^{(k_{n})})_{n\to\omega}u^{*}.\]
By \cite[Proposition II.5.1.5]{BlackadarOA}  we can find $X^{(n)}\in P_{k_{n}}R\Ball(\bM_{k_{n}}(\bC))P_{k_{n}}$ so that:
\[\Theta(b)=\Theta(p)\Theta(b)\Theta(p)=u(X^{(n)})_{n\to\omega}u^{*}.\]
By the proof of \cite[Theorem XIV.4.6]{TakesakiIII} we may write $U=(U^{(n)})_{n\to\omega}$ for some $(U^{(n)})_{n}\in \prod_{n}\cU(\bM_{k_{n}}(\bC))$. Then by our choice of $\varepsilon$, $A^{(n)}$,
\[\varepsilon\leq \lim_{n\to\omega}\|A^{(n)}-U^{(n}(X^{(n)},P_{k},(V_{i}^{(k_{n})})_{i=2}^{n})(U^{(n)})^{*}\|_{2}=\|\Theta(a)-u((X^{(n)})_{n\to\omega},((V^{(k_{n})}_{i})_{i=1}^{r})_{n\to\omega}u^{*}\|_{2}=0,\]
a contradiction.

\end{proof}

 Recall that a $\textrm{II}_{1}$-factor is \emph{stably singly generated} (as defined in \cite{SorinSSG}) if for every nonzero projection $p\in M$ we have that $pMp$ is singly generated. 
The following is the main result that allows us to prove Theorem \ref{thm: nontrivial computations for tracial vNas are hard}.
This result can be proved using known results \cite[Corollary 5]{DimaNonMS} (see also \cite[Theorems A-C and Remark 6.9]{AldoNONMS}) in non-microstates free entropy dimension theory, which for a finite group $G$ acting in a trace-preserving manner on a tracial von Neumann algebra $M$ bounds the infimum of the non-microstates free entropy dimension of generating tuples for $M\rtimes G$ in terms of $|G|$ and the minimal number of generators for $M$. We have elected to give the following proof, since it is relatively short, self-contained, and transparently explains why the microstates space asymptotically has small packing number. The reader should also compare this result with \cite[Corollary 8.5 and 8.6]{DSSW}.

\begin{prop}\label{prop: gen with small entropy}
Let $M$ be a $\textrm{II}_{1}$-factor, and suppose that $M$ is stably singly generated. Then for every $\kappa>0$, there is an $r\in \bN$ and an $a\in M^{r}$ with $\delta_{0}(a)\leq 1+\kappa$. 
\end{prop}

\begin{proof}
Let $\tau$ be the unique tracial state on $M$. 
Fix an $n\in \bN$ and a projection $p\in M$ with $\tau(p)=\frac{1}{n}$. and let $b\in pMp$ generate $pMp$ with $\|b\|\leq 1$. Since $M$ is a factor, we may then find partial isometries $v_{1},\cdots,v_{n}\in M$ with $v_{1}=p$, $v_{i}^{*}v_{i}=p$ and $\sum_{i}v_{i}v_{i}^{*}=1$.  Set $a=(b,v_{1},\cdots,v_{n})\in M^{n+1}$ and note that $a$ generates $M$. 

Since $W^{*}(v_{1},\cdots,v_{n})\cong \bM_{n}(\bC)$ we may find a microstates sequence $(V_{i}^{(k)})_{i=1}^{n}$ for $(v_{1},\cdots,v_{n})$ with $\|V_{i}^{(k)}\|\leq 1$. Since $v_{1}=p$, we may also assume that $V_{1}^{(k)}$ is a projection, call it $P_{k}$. 
Let $\varepsilon>0$, and for $n\in \bN$, choose an $X\subseteq P_{k}\Ball(\bM_{k}(\bC))P_{k}$ which is $\varepsilon$-dense with respect to $\|\cdot\|_{2}$ and has $|X|\leq \left(\frac{C}{\varepsilon}\right)^{\tr(P_{k})^{2}k^{2}}$ for some universal constant $C>0$.
By Lemma \ref{lem: small neigbhorhoods concentrate on orbits} if $\mathcal{O}$ is a sufficiently small neighborhood of $\ell_{a}$, then $\Gamma^{(k)}_{1}(\mathcal{O})$ is contained in the $\varepsilon$-neighborhood of $\bigcup_{U\in \cU(\bM_{k}(\bC))}UXU^{*}$.
By \cite[Theorem 7]{Szarek}, we may choose an $\Omega\subseteq \cU(\bM_{k}(\bC))$ which is $\varepsilon$-dense in operator norm and has $|\Omega|\leq \left(\frac{C_{1}}{\varepsilon}\right)^{k^{2}}$ for some universal constant $C_{1}$. Then $\bigcup_{U\in \Omega_{k}}UXU^{*}$ contain $\Gamma^{(k)}_{1}(\mathcal{O})$ in its $3\varepsilon$-neighborhood with respect to $\|\cdot\|_{2}$. Hence
\[\frac{1}{k^{2}}\log K_{6\varepsilon}(\Gamma^{(k)}_{1}(\mathcal{O}))\leq \tr(P_{k})^{2}\log(C/\varepsilon)+\log(C_{1}/\varepsilon).\]
Since $\lim_{k\to\infty}\tr(P_{k})=\tau(p)=\frac{1}{n}$, it follows that
\[\delta_{6\varepsilon}(a)\leq \delta_{6\varepsilon}(\mathcal{O})\leq 1+n^{-2}.\]
Taking the limit supremum in $\varepsilon$, we find that 
\[\delta_{0}(a)\leq 1+n^{-2}.\]
\end{proof}

\begin{proof}[Proof of Theorem \ref{thm: nontrivial computations for tracial vNas are hard}]
Splitting the center of $M$ into diffuse and atomic pieces, we may write 
\[M=M_{0}\oplus \bigoplus_{j\in J}M_{j},\]
where:
\begin{itemize}
    \item $M_{0}$ is either $\{0\}$ or has diffuse center,
    \item $J$ is a countable set (potentially empty),
    \item for all $j\in J$ we have that $M_{j}$ is a $\textrm{II}_{1}$-factor.
\end{itemize}

First, assume $M_{*}$ is separable. 
Fix $\kappa>0$.
Let $\lambda_{j}=\tau(1_{M_{j}})$. Fix a countable generating tuple $a_{0}$ for $M_{0}$. If $M_{0}\ne 0$, then by \cite[Example 3]{FreePinsker}, we have that $h(M_{0},\frac{\tau|M_{0}}{\tau(1_{M_{0}})})\leq 0$, thus $\delta_{0}(a_{0},\frac{\tau|_{M_{0}}}{\tau(1_{M_{0}})})\leq 1$.
For $j\in J$, by Proposition \ref{prop: gen with small entropy} we may find a finite tuple $a_{j}\in M_{j}^{r_{j}}$ with $\delta_{0}\left(a_{j},\frac{\tau|_{M_{j}}}{\tau(1_{M_{j}})}\right)\leq 1+\kappa$. 
Set $a=(a_{0},(a_{j})_{j\in J})$. Then, by \cite[Lemma 4.8]{HJKEPropT} (which builds off of \cite[Lemma 3.2 and Corollary 4.3]{JungHyperFiniteIneq}) we have 
\[\delta_{0}(a)-1\leq \tau(1_{M_{0}})^{2}\left(\delta_{0}\left(a_{0},\frac{\tau|_{M_{0}}}{\tau(1_{M_{0}})}\right)-1\right)+\sum_{j\in J}\tau(1_{M_{j}})^{2}\left(\delta_{0}\left(a_{j},\frac{\tau|_{M_{j}}}{\tau(1_{M_{j}})}\right)-1\right)\leq \kappa\sum_{j\in J}\tau(1_{M_{j}})^{2}\leq \kappa.\]
By Corollary \ref{cor: delta MD and MFED}, it follows that
\[\delta_{MD}(M,\tau)\leq 1+\kappa.\]
The proof is completed by sending $\kappa\to 0$. In this case, we see by contraposition that if $\delta_{MD}(M,\tau)>1$, then one of the $M_{j}$ for $j\in J$ has a corner which is not singly generated, giving the desired $N$ as a corner of $M$.

Now consider the general case. For $j\in J\sqcup\{0\}$, let $M_{j,\alpha}$ be increasing nets of von Neumann subalgebras with separable predual $M_{j}=\bigvee_{\alpha}M_{j,\alpha}$ (e.g. because each $M_{j}$ has a faithful, normal, tracial state the set of subalgebras with separable predual is directed). Set 
\[M_{\alpha}=M_{0,\alpha}\bigoplus_{j\in J}M_{j,\alpha}.\]
Suppose that for each $j,\alpha$ we have that $M_{j,\alpha}$ is stably singly generated. Then by the case where $M_{*}$ is separable, we have $\delta_{MD}(M_{\alpha})\leq 1$. Thus,
\[\delta_{MD}(M_{\alpha}:M,\tau)\leq \delta_{MD}(M_{\alpha},\tau)\leq 1.\]
Taking the supremum over $\alpha$ and applying Proposition \ref{prop: limiting formula for MD} we see that $\delta_{MD}(M,\tau)\leq 1$.

\end{proof}

\subsection{$\delta_{MD}$ as an obstruction to generators for tracially complete $C^{*}$-algebras}

Recall that our purpose in this paper for defining the quantity $\delta_{MD}$ is that it provides an obstruction to a bundle being fiberwise finitely generated. The following proposition is how we will use $\delta_{MD}$ to obstruct fiberwise finite generation. 

\begin{prop}\label{prop: md bounds finite generation}
Let $(M,X)$ be a tracially complete $C^{*}$-algebra and $Q\leq N\leq M$. Suppose that $a_{1},\cdots,a_{n}\in M_{s.a}$ and that $C^{*}_{X}(\{a_{1},\cdots,a_{n}\}\cup N)\supseteq Q$. Then $\delta_{MD}(Q|N:M)\leq n$.   
In particular, if $b_{1},\cdots,b_{k}\in M$ satisfy that $C^{*}_{X}(\{b_{1},\cdots,b_{k}\}\cup N)\supseteq Q$, then $\delta_{MD}(Q|N:M)\leq 2k.$ 
\end{prop}

\begin{proof}
Note that if $\widetilde{Q}=C^{*}_{X}(\{a_{1},\cdots,a_{n}\}\cup N)$, then by Corollary \ref{cor:basic properties} we have that $\delta_{MD}(Q|N:M)\leq \delta_{MD}(\widetilde{Q}|N:M)$. Hence, replacing $Q$ with $\widetilde{Q}$, we may (and will) assume that $C^{*}_{X}(\{a_{1},\cdots,a_{n}\}\cup N)=Q$. Similarly, since $\delta_{MD}(Q|N:M)\leq \delta_{MD}(Q|N:Q)$, we may (and will) assume that $M=Q$. Let $I$ be a set and $c\in Q^{I}$ a generating tuple. Fix $R\in [0,\infty)^{[n]\sqcup I}$ a cutoff function for $a,c$. Suppose we are given an $\varepsilon>0$ and a finite $F\subseteq [n]\sqcup I$.
We may choose a weak$^{*}$-neighborhood $\mathcal{O}$ of $L_{a,c;X}$ so that $\|A_{i}-\frac{A_{i}+A_{i}^{*}}{2}\|_{2}<\frac{\varepsilon}{4\sqrt{n}}$ for all $i\in [n]$ and all $A=(A_{i})_{i\in [n]\sqcup I}\in \Gamma^{(n)}_{R}(\mathcal{O})$. 
Set $G=F\cap I$, and fix any $\theta<\varepsilon/2$. Then the  map $f\colon \prod_{i\in [n]\sqcup I}R_{i}\Ball(\bM_{k}(\bC))\to \prod_{i=1}^{n}R_{i}\Ball(\bM_{k}(\bC))_{s.a}$ given by
\[f(A)=\left(\frac{A_{i}+A_{i}^{*}}{2}\right)_{i=1}^{n}\] is $(\varepsilon,F|\theta,G)$-injective.
Hence
\[\md_{\varepsilon,F}(a|c)\leq \md_{\varepsilon,F|\theta,G}(a|c)\leq n.\]
Taking  the supremum over $\varepsilon,F$ proves that 
\[\delta_{MD}(Q|N)\leq n.\]
The ``in particular" part follows by setting 
\[a_{j}=\frac{b_{j}+b_{j}^{*}}{2},a_{j+k}=\frac{b_{j}-b_{j}^{*}}{2i}\]
for $j\in [k]$. 

\end{proof}

For $W^{*}$-bundles we can connect the concept of finitely generated over $N$, as in the above Proposition, with being fiberwise finitely generated over $N$.

\begin{lem}
Let $(M,\bE,K)$ be a $W^{*}$-bundle with $K$ compact Hausdorff. Let $N\leq M$ with $C(K)\leq N$. Given a family $(a_{i})_{i\in I}\in M^{I}$ for some set $I$, the following are equivalent:
\begin{enumerate}[(i)]
    \item for every $p\in K$ we have that $M_{p}=W^{*}(\{a_{i,p}:i\in I\}\cup N_{p})$, \label{item: fiberwise generation lem}
    \item $C^{*}_{X}(N\cup \{a_{i}:i\in I\})=M$ with $X=\overline{\co}^{wk^{*}}(\{\tau_{p}:p\in K\})$. \label{eqn: relative generation lem}
\end{enumerate}
\end{lem}

\begin{proof}

That (\ref{eqn: relative generation lem}) implies (\ref{item: fiberwise generation lem}) follows by $\|\cdot\|_{\bE}$-$\|\cdot\|_{2}$ continuity of $\pi_{p}$ for $p\in K$.

To see that (\ref{item: fiberwise generation lem}) implies (\ref{eqn: relative generation lem}), fix a cutoff function $R\in [0,+\infty)^{I}$. Let $x\in M$ and $\varepsilon>0$ be given. Given $p\in K$, apply Proposition \ref{prop: better KDT} and (\ref{item: fiberwise generation lem}) to choose  finite $F_{p}\subseteq N$  and a $P_{p}\in \bC^{*}\ip{(T_{i})_{i\in I\cup F_{p}}}$ so that $\|P_{p}\|_{R}\leq \|x\|$ and $\|P_{p}((a_{i,p})_{i\in I},(b_{p})_{b\in F_{p}})-x_{p}\|_{L^{2}(\tau_{p})}<\varepsilon$. Since $q\mapsto \|P_{p}((a_{i,q})_{i\in I},(b_{q})_{b\in F_{p}})-x_{q}\|_{L^{2}(\tau_{q})}$ is continuous, we can find a neighborhood $\mathcal{O}_{p}$ of $p$ so that $\|P_{p}((a_{i,q})_{i\in I},(b_{q})_{b\in F_{p}})-x_{q}\|_{L^{2}(\tau_{q})}<\varepsilon$ for all $q\in \mathcal{O}_{p}$. By compactness of $K$, we may find $p_{1},\cdots,p_{n}$ so that $K=\bigcup_{i=1}^{n}\mathcal{O}_{p_{i}}$. Let $(\varphi_{i})_{i=1}^{n}$ be a partition of unity subordinate to $(\mathcal{O}_{p_{i}})_{i=1}^{n}$, and set $y=\sum_{j=1}^{n}\varphi_{j}P_{p_{j}}((a_{i})_{i\in I},F_{p_{j}})$. Since $C(K)\leq N$, we have $y\in C^{*}_{X}(N\cup \{a_{i}:i\in I\}\}$, and for every $p\in K$ we have:
\[\|y_{p}-x_{p}\|_{L^{2}(\tau_{p})}\leq \sum_{j:p\in \mathcal{O}_{p_{j}}}\varphi_{p_{j}}(p)\|P_{p_{j}}((a_{i})_{i\in I},F_{p_{j}})-x_{p}\|_{2}<\varepsilon.\]
Thus $\|y-x\|_{\bE}<\varepsilon$, and this proves that $x\in C^{*}_{X}(a_{i}:i\in I)$. 
\end{proof}

We say that the family $(a_{i})_{i\in I}$ \textbf{fiberwise generates $M$ over $N$} if either/both of the above equivalent conditions hold. If $N=C(K)$, we say that $(a_{i})_{i\in I}$ \textbf{fiberwise generates $M$}. 
In Section \ref{sec: finally it is done}, we will exhibit a $W^{*}$-bundle $(M,\bE,K)$ which has a sub-bundle which is the trivial bundle over $K$ with fiber a nonamenable $\textrm{II}_{1}$-factor $Q$, with $Q'\cap M_{p}=\bC1$ for all $p\in K$, and which has $\delta_{MD}(M;X)=+\infty$, where $X=\overline{\co}^{wk^{*}}(\{\tau_{p}:p\in K\})$. The following Corollary will then be used to show that $M$ cannot be fiberwise finitely generated over a sub-bundle $N$ which satisfying that $\sup_{p\in K}h(N_{p},\tau_{p})<+\infty$, and this will prove Theorem \ref{thm: main thm intro}.

\begin{cor} \label{cor: better main thm}
Let $(M,\bE,K)$ be a $W^{*}$-bundle with $\delta_{MD}(M;X)=+\infty$ where $X=\overline{\co}^{wk^{*}}(\{\tau_{p}:p\in K\})$. Then $M$ cannot be fiberwise generated by a finite family of continuous sections. More generally if $N$ is a sub-bundle of $M$ and $\delta_{MD}(N;X)<+\infty$, then $M$ cannot be fiberwise generated by a finite family of continuous sections over $N$.  
In particular, this holds if 
\begin{enumerate}[(i)]
 \item $h(N_{p},\tau_{p})\leq 0$ for every $p\in K$. For example, if for every $p\in K$ we have that $N_{p}$ satisfies one of the following conditions:
\begin{itemize}
    \item If $N_{p}$ is hyperfinite,
    \item If $N_{p}$ has a diffuse, regular, hyperfinite subalgebra,
    \item $N_{p}$ is generated by two commuting diffuse subalgebras,
    \item $N_{p}$ has diffuse central sequence algebra,
    \item $N_{p}$ has a single sequential commutation orbit in the sense of  \cite{SeqCommutation},
\end{itemize}
\item \label{item: uniformly strongly 1-bounded_2}  more generally if $\sup_{p\in K}h(N_{p},\tau_{p})<+\infty$. 
\item \label{item: M property (T)_2} 
If $(N,X)$ satisfies the following uniform Property (T) assumption: there is an $r\in \bN$ and a $C>0$ so that for every $p\in K$, there is a tuple $x\in \Ball(N_{p})^{r}$ which is a Kazhdan tuple with constant $C>0$. 
For example, this holds if $(N,X)$ is the tracial completion $(C^{*}(G),\Upsilon)$ where $G$ is any Property (T) group and $\Upsilon$ is a compact, convex subset of $\cT(C^{*}(G))$.

\end{enumerate}

\end{cor}

\begin{proof}
By Theorem \ref{thm: infintie MD} and Corollary \ref{cor:basic properties}, we have that 
\[+\infty=\delta_{MD}(M;X)\leq \delta_{MD}(M|N;X)+\delta_{MD}(N;X).\]
Thus $\delta_{MD}(M|N;X)$ is infinite.
It follows by Proposition \ref{prop: md bounds finite generation} that we cannot find a finite family $(a_{i})_{i=1}^{n}\in M^{n}$ so that $M=C^{*}_{X}(\{a_{i}:i=1,\cdots,n\}\cup N)$. 
The ``in particular" part follows from Corollary \ref{cor: basic examples of vanishing}.
\end{proof}

\section{Proof of the main theorem}

In this section we prove Theorem \ref{thm: main thm intro}. We will use Corollary \ref{cor: better main thm} and so wish to build a $W^{*}$-bundle $(M,\bE,K)$ which has a sub-bundle the trivial bundle with fibers $Q$, where $Q$ is a nonamenable $\textrm{II}_{1}$-factor, and so that:
\begin{itemize}
    \item $Q'\cap M_{p}=\bC 1$ for all $p\in K$,
    \item $\delta_{MD}(M;X)=+\infty$, where $X=\overline{\co}^{wk^{*}}(\{\tau_{p}:p\in K\})$. 
\end{itemize}
First, notice that to build a factorial $W^{*}$-bundle $(M,\bE,K)$, it suffices by Proposition \ref{prop: creating faces}, \cite[Theorem 3]{OzawaBundle} (see also \cite[Theorem 3.37]{TraciallyComplete}), and \cite[Proposition 3.23 (ii) and (iv)]{TraciallyComplete} to find a $C^{*}$-algebra and a compact $K\subseteq \Ext(\cT(A))$. For a trace to be extremal just means it generates a factor, and so we wish to find a compact family of traces with factorial GNS completions. In order to force the first item above, this family of extremal traces should be constant on some sub-algebra $B$, and the GNS completion of $B$ with respect to all/any of the traces in $K$ should be a non-amenable $\textrm{II}_{1}$-factor $Q$. Moreover, this $Q$ should be an irreducible subfactor of any of the GNS completions of $A$ with respect to any of the traces in $K$. Finally, to guarantee that $\delta_{MD}$ of the bundle arising as the tracial completion is infinite, $K$ should be not so restrictive so as to guarantee that there are lots of homomorphisms from $A$ into a tracial ultraproduct of matrices whose traces pullback to elements of $K$. 

The core idea is to start with some nonamenable $\textrm{II}_{1}$-factor $Q$ which has \emph{some embedding} into a tracial ultraproduct of matrices with trivial relative commutant. We give examples in the next section. Let $B$ be a unital and weak$^{*}$-dense subalgebra of $Q$. We then take a universal free product of $B$ with a suitable universal $C^{*}$-algebra (e.g. $C^{*}(\bF_{\infty})$, or the universal $C^{*}$-algebra generated by countably many self-adjoint contractions) and form a $C^{*}$-algebra $\cA$. Since the embedding of $Q$ into a tracial ultraproduct of matrices is irreducible, \emph{any} intermediate von Neumann subalgebra will be a factor. We thus define a compact (compactness will be guaranteed via a diagonal argument) family of extremal traces by considering all traces on $\cA$ which (roughly speaking) come from homomorphisms on $\cA$ into a tracial ultraproduct of matrices which extend the given embedding of $B$ with trivial relative commutant. The desired $(A,K)$ is then given as the separation of $\cA$ by $\cK$.

\subsection{von Neumann algebras which embed into ultraproducts of matrices with trivial relative commutant}\label{sec: trivial rel comm examples}

This subsection is largely folklore, but we include complete proofs for convenience of the reader.

Let $\bF_{2}$ be the free group on two letters $a,b$. Let us first show that $L(\bF_{2})$ has an embedding into a matrix ultraproduct with trivial relative commutant. 
By \cite{Hastings} (see also \cite{PisierQuantumExpander} and \cite{CollinsBordenave}) and \cite[Theorem 3.8]{VoicAsyFree}, almost every (with respect to the infinite product of Haar measure on  $\cU(\bM_{N}(\bC))^{2}$) sequence$(U_{1}^{(N)},U_{2}^{(N)})\in \cU(\bM_{N}(\bC))^{2}$ satisfies the following:
\begin{itemize}
\item for every free ultrafilter $\omega\in\beta\bN\setminus \bN$ there is a trace-preserving embedding $\Theta_{\omega}\colon L(\bF_{2})\to \prod_{n\to\omega}^{tr}\bM_{N}(\bC)$ with $\Theta_{\omega}(\lambda(a))=(U_{1}^{(N)})_{N\to\omega},\Theta_{\omega}(\lambda(b))=(U_{2}^{(N)})_{N\to\omega}.$ 
\item We have
\[\limsup_{n\to\infty}\left\|\sum_{i=1}^{2}\Re(U_{i}^{(N)}\otimes \overline{U_{i}^{(N)}})\right\|_{\bM_{N}(\bC)\ominus \bC 1}<1.\]
\end{itemize}

Suppose $x=(x_{N})_{N\to\omega}\in \cM\ominus\bC 1.$ We may, and will, assume that $\tr(x_{N})=0$ for all $N$. Then
\[\|[\Theta_{\omega}(\lambda(a)),x]\|_{2}^{2}+\|[\Theta_{\omega}(\lambda(b)),x]\|_{2}^{2}=\lim_{N\to\omega}2\|x_{N}\|_{2}^{2}-2\ip{T_{N}(x_{N}),x_{N}},\]
where $T_{N}=\sum_{i=1}^{2}\Re(U_{i}^{(N)}\otimes \overline{U_{i}^{(N)}}).$ Since $\lim_{N\to\infty}\|T_{N}\|_{\bM_{N}(\bC)\ominus\bC 1}<1$, we see that \[\|[\Theta_{\omega}(\lambda(a)),x]\|_{2}^{2}+\|[\Theta_{\omega}(\lambda(b)),x]\|_{2}^{2}>0.\]
Thus $\Theta_{\omega}(L(\bF_{2}))'\cap \cM=\bC 1$. 
Another application of \cite[Theorem 3.8]{VoicAsyFree} tells us that for a.e. such sequence satisfies that for every free ultrafilter $\omega$ there is a extension of $\Theta_{\omega}$ to an embedding $L(\bF_{\infty})\to \cM$. Hence every $L(\bF_{2})\leq Q\leq L(\bF_{\infty})$ has an embedding into an ultraproduct of matrices with trivial relative commutant. This property is also seen to be preserved by taking corners of a given von Neumann algebra or matrices over a given von Neumann algebra. Thus every interpolated free group factor in the sense of \cite{DykmaIFGF, RadulescuIFGF} has an embedding into an ultraproduct of matrices with trivial relative commutant.

Another example is given by the following. 
\begin{thm}\label{thm: Property T trivial rel comm}
Let $G$ be a Property (T) group, and let $\pi_{n}\colon G\to \cU(\bM_{k_{n}}(\bC))$ be a sequence of irreducible representations with $k_{n}\to\infty$. Fix a free ultrafilter $\omega$, and let $\pi_{\omega}\colon G\to \cU\left(\prod_{n\to\omega}\bM_{k_{n}}(\bC)\right)$ be given by $\pi_{\omega}(g)=(\pi_{n}(g))_{n\to\omega}$, and set $\cM=\prod_{n\to\omega}\bM_{k_{n}}(\bC)$. Then $W^{*}(\pi_{\omega}(G))'\cap \cM=\pi_{\omega}(G)'\cap \cM=\bC 1$.   
\end{thm}

\begin{proof}
Fix $x\in \pi_{\omega}(G)'\cap \cM$ with $\|x\|\leq 1$. By \cite[Proposition II.5.1.5]{BlackadarOA} we may write $x=(x_{n})_{n\to\omega}$ with $\|x_{n}\|\leq 1$. Consider the representations $\phi_{n}\colon G\to \cU(S^{2}(k_{n},\tr))$ given by $\phi_{n}(g)A=\pi_{n}(g)A\pi_{n}(g)^{-1}$. Since each $\pi_{n}$ is irreducible, it follows that 
$\Fix(\phi_{n})=\bC1$. Let $(S,C)$ be a Kazhdan pair for $G$. Then:
\[\|x_{n}-\tr(x_{n})\|_{2}=\|x_{n}-P_{\Fix(\phi_{n})}(x_{n})\|_{2}\leq C\left(\sum_{s\in S}\|\phi_{n}(s)x_{n}-x_{n}\|_{2}^{2}\right)^{1/2}=C\left(\sum_{s\in S}\|[x_{n},\pi_{n}(s)]\|_{2}^{2}\right)^{1/2}.\]
Taking $\lim_{n\to\omega}$, we see that $x\in \bC 1$.

\end{proof}

For a representation $\pi\colon G\to \cU(A)$ where is a $C^{*}$-algebra, we continue use $\pi$ for the unique $*$-homomorphism $C^{*}(G)\to A$ induced by $\pi$. 

Given a Property (T) group and a $\chi\in \Ext(\cT(C^{*}(G)))$, it is natural to ask when there is such a sequence $\pi_{n}$ with $W^{*}(\pi_{\omega}(G))=W^{*}(\pi_{\chi}(C^{*}(G)))$, with notation as in the above Theorem.  We proceed to show that this is less restrictive than one might expect, and amounts to the existence of finite-dimensional representations (irreducible or not) whose normalized characters converge to $\chi$. The essential idea will be that if we decompose such finite-dimensional representations into irreducible pieces, then the corresponding character will be a convex combination of characters associated with irreducible representations. Since factoriality of $W^{*}(\pi_{\chi}(C^{*}(G)))$ is equivalent to extremality of $\chi$, if a convex combination is close to $\chi$, then some (in some sense most) of the terms in the convex combination will have to be close to $\chi$. 
For this, it will be helpful to use the following folklore result.

\begin{lem}\label{lem: perturb extremality}
Let $V$ be a locally convex space, $C\subseteq V$ a compact convex set, and $p\in \Ext(C)$. Then given any neighborhood $\mathcal{O}$ of $p$ and a $\eta>0$, there is a neighborhood $\mathcal{V}$ of $p$ so that for every $\mu\in \Prob(C)$ with $\int x\,d\mu(x)\in \mathcal{V}$, we have $\mu(\mathcal{O})\geq 1-\eta$.       
\end{lem}

\begin{proof}
If the lemma were false, then we could find a neighborhood $\mathcal{O}$ of $p$, an $\eta>0$ and a net $\mu_{\alpha}\in \Prob(C)$ with $\int x\,d\mu_{\alpha}(x)\to p$, but $\mu_{\alpha}(\mathcal{O})<1-\eta$. Passing to a subsequence we may, and will, assume that there is a $\mu\in \Prob(C)$ with $\mu_{\alpha}\to\mu$ weak$^{*}$. Then $p=\int x\,d\mu(x)$, and extremality of $p$ forces $\mu=\delta_{p}$. Thus, by the Portmanteau theorem \cite[Theorem 11.1.1]{DudleyProb}, we have 
\[1=\delta_{p}(\mathcal{O})\leq \liminf_{\alpha}\mu_{\alpha}(\mathcal{O})\leq 1-\eta,\]
a contradiction.

\end{proof}

Using this, we can show that (under the assumption of factoriality) the hypotheses in Theorem \ref{thm: Property T trivial rel comm} are no different than just assuming the existence of finite-dimensional representations which model $W^{*}(\pi_{\omega}(G)).$

\begin{cor}\label{cor:removing irrep assumption}
Let $G$ be a group, and $\chi\in \Ext(\cT(C^{*}(G)))$. Set $M_{\chi}=W^{*}(\pi_{\chi}(C^{*}(G)))$. The following are equivalent:
\begin{enumerate}[(i)]
    \item there is a sequence $\pi_{n}\colon G\to \cU(\bM_{k_{n}}(\bC))$ of finite-dimensional representations with $\tr\circ \pi_{n}\to_{n\to\infty}\chi$ weak$^{*}$, \label{item: irrep approximation}
    \item there is a sequence $\rho_{n}\colon G\to \cU(\bM_{m_{n}}(\bC))$ of irreducible representations with $\tr\circ \rho_{n}\to \chi$ weak$^{*}$. \label{item: general approximation}
\end{enumerate}
\end{cor}

\begin{proof}
That (\ref{item: general approximation}) implies (\ref{item: irrep approximation}) is direct. For the reverse implication, let $\mathcal{O}$ be a weak$^{*}$-neighborhood of $\chi$. It suffices to find an irreducible representation $\pi\colon G\to \cU(\cH)$ with $\cH$ finite-dimensional and so that $\tr\circ \pi\in \mathcal{O}$. Let $\mathcal{V}$ be as in Lemma \ref{lem: perturb extremality} for this $\mathcal{O}$ and $\eta=\frac{1}{2}$. By (\ref{item: general approximation}), we may find a finite-dimensional representation $\rho\colon G\to \cU(\bM_{m}(\bC))$ so that $\tr\circ \rho\in \mathcal{V}$. Write $\rho=\bigoplus_{j=1}^{d}\rho_{j}$ where the $\rho_{j}$ are irreducible, and set $m_{j}$ to the dimension of $\rho_{j}$. Then
\[\tr\circ \rho=\sum_{j=1}^{d}\frac{m_{j}}{m}\tr\circ \rho_{j}.\]
Hence, by Lemma \ref{lem: perturb extremality}, we have that 
\[\sum_{j:\tr\circ\rho_{j}\in \mathcal{O}}\frac{m_{j}}{m}\geq \frac{1}{2}.\]
Thus we can find a $j$ with $\tr\circ\rho_{j}\in \mathcal{O}$.
\end{proof}

Combined with our earlier work, we can give more examples of Property (T) von Neumann algebras which embed with trivial relative commutant into an ultraproduct of matrices.

\begin{cor}
Let $G$ be a Property (T) group, and let $\chi\in\Ext(\cT(C^{*}(G)))$. Suppose that there is a sequence $\pi_{n}\colon G\to \cU(\bM_{k_{n}}(\bC))$ of representations with $\tr\circ\pi_{n}\to \chi$. Then there is an embedding of $W^{*}(\pi_{\chi}(G))$ into an ultraproduct of matrices with trivial relative commutant. In particular, this holds for $\chi=\delta_{e}$ (and hence $W^{*}(\pi_{\chi}(G))=L(G)$)  if $G$ is an i.c.c, residually finite, Property (T) group.  
\end{cor}

\begin{proof}
That there is such an embedding of $W^{*}(\pi_{\chi}(G))$ follows from  Corollary \ref{cor:removing irrep assumption} and Theorem \ref{thm: Property T trivial rel comm}. For the ``in particular part" first note that if $G$ is i.c.c, then $L(G)$ is a factor, so $\delta_{e}$ is an extremal character. If $G$ is residually finite, then we may find a decreasing sequence $G_{n}$ of finite index, normal subgroups of $G$ with $\bigcap_{n}G_{n}=\{e\}$. The corresponding representations $\lambda_{G/G_{n}}\colon G\to \cU(B(\ell^{2}(G/G_{n})))$ then satisfy that $\tr\circ \lambda_{G/G_{n}}\to \delta_{e}$. 
\end{proof}

\subsection{Construction and proof of the main theorem} \label{sec: finally it is done}
It will be helpful to give an equivalent asymptotic criterion to the condition that $Q$ has an embedding into an ultraproduct of matrices with trivial relative commutant. 

\begin{prop}\label{prop: asymptotic rephrase}
Let $(Q,\tau)$ be a tracial von Neumann algebra with separable predual, and let $x=(x_{j})_{j\in J}$ be a countable generating tuple for $Q$. Then the following are equivalent:
\begin{enumerate}[(i)]
    \item there is a free ultrafilter $\omega\in \bN\setminus \bN$ and a trace-preserving embedding $\Theta\colon Q\to\prod_{k\to\omega}\bM_{k}(\bC):=\cM_{\omega}$, so that $\Theta(Q)'\cap \cM_{\omega}=\bC1$, \label{item: trivial relative commutant ultrapower}
    \item there is a sequence $n_{1}<n_{2}<\cdots$ of positive integers, and $X^{(k)}=(X^{(k)}_{j})_{j\in J}\in \prod_{j\in J}\|x_{j}\|\Ball(\bM_{n_{k}}(\bC))$ so that $\ell_{X^{(k)}}\to \ell_{x}$, which satisfy the following property. For every $\varepsilon>0$, there is a $\delta>0$, finite $F\subseteq I$ and a $K\in \bN$, so that whenever $k\geq K$, and $A\in \Ball(\bM_{n_{k}}(\bC))$ satisfies $\|A\|\leq 1$ and $\|[X^{(k)},A]\|_{2,F}<\delta$, then $\|A-\tr(A)\|_{2}<\varepsilon$. \label{item: asymptotic criterion}
\end{enumerate}

\end{prop}

\begin{proof}
(\ref{item: asymptotic criterion}) implies (\ref{item: trivial relative commutant ultrapower}): Fix a free ultrafilter $\widetilde{\omega}\in \bN\setminus \bN$, and define a free ultrafilter $\omega\in \bN\setminus\bN$ by
\[\lim_{n\to\omega}f(n)=\lim_{k\to\widetilde{\omega}}f(n_{k})\]
for all $f\in \ell^{\infty}(\bN)$. Then $\cM_{\omega}\cong \prod_{k\to\widetilde{\omega}}^{tr}\bM_{n_{k}}(\bC)$. Since $\ell_{X^{(k)}}\to_{k\to\infty}\ell_{x}$, we have a trace-preserving embedding $\Theta_{\omega}\colon Q\to \prod_{k\to\widetilde{\omega}}\bM_{n_{k}}(\bC)$ satisfying $\Theta_{\omega}(x_{j})=(X^{(k)}_{j})_{k\to\widetilde{\omega}}$ for all $j\in J$. Suppose $a\in \Theta_{\omega}(Q)'\cap \prod_{k\to\widetilde{\omega}}\bM_{n_{k}}(\bC)$  with $\|a\|\leq 1$. Then by \cite[Proposition II.5.1.5]{BlackadarOA}, we may write $a=(A^{(k)})_{k\to\widetilde{\omega}}$ with $\|A^{(k)}\|\leq 1$ for all $k\in \bN$. Let $\varepsilon>0$, and choose a $\delta>0$, a finite $F\subseteq I$, and a $K\in \bN$ as in (\ref{item: asymptotic criterion}). Then, by assumption,
\[\|A^{(k)}-\tr(A^{(k)})\|_{2}\leq \varepsilon+\frac{2}{\delta}\|[A^{(k)},X^{(k)}]\|_{2,F}, \textnormal{ for all $k\geq K$}.\]
Taking $\lim_{k\to\widetilde{\omega}}$ and using that $[a,x_{j}]=0$ for all $j\in F$, we have 
\[\|a-\tr_{\omega}(a)\|_{2}\leq \varepsilon.\]
Since $\varepsilon>0$ is arbitrary, this implies that $a=\tr_{\omega}(a)\in \bC1$.

(\ref{item: trivial relative commutant ultrapower}) implies (\ref{item: asymptotic criterion}):  Define $R\in [0,\infty)^{J}$ by $R_{j}=\|x_{j}\|$. By \cite[Proposition II.5.1.5]{BlackadarOA}, we may write $\Theta(x_{j})=(X^{(n)}_{j})_{n\to\omega}$ with $\|X^{(n)}_{j}\|\leq \|x_{j}\|$ for all $n\in \bN$. 

By a diagonal argument, it suffices to prove the following:

\emph{Claim: given a weak$^{*}$-neighborhood $\mathcal{O}$ of $\ell_{x}$ in $\Sigma_{R}$, an $\varepsilon>0$ and an $N\in \bN$, there is a $\delta>0$, a finite $F\subseteq I$ and an $n\in \bN$ with $n\geq N$ with $\ell_{X^{(n)}}\in \mathcal{O}$ and so that  $\|A-\tr(A)\|_{2}<\varepsilon$ for all $A\in \Ball(\bM_{n}(\bC))$ with $\|[A,X^{(n)}]\|_{2,F}<\delta$.}

Choose a finite $G\subseteq \bC^{*}\ip{(T_{j})_{j\in J}}$ and an $\eta>0$ so that 
\[\mathcal{O}\supseteq \bigcap_{P\in G}\{\ell\in \Sigma_{R}:|\ell(P)-\tau(P(x))|<\eta\}.\]
Let $E=\{k\in \bN:k\geq N, \text{ and }\ell_{X^{(k)}}\in \mathcal{O}\}$, then
\[\lim_{n\to\omega}1_{E^{c}}(n)\leq \frac{1}{\eta}\lim_{n\to\omega}\sum_{P\in G}|\tr(P(X^{(n)}))-\tau(P(x))|=0,\]
so
\[\lim_{n\to\omega}1_{E}(n)=1.\]
Write $J=\bigcup_{n}F_{n}$ where $F_{n}$ is an increasing sequence of finite subsets of $J$. Assuming the claim is false, for each $n\in E$  we may find an $A^{(n)}\in \Ball(\bM_{n}(\bC))$ so that $\|[A^{(n)},X]\|_{2,F_{n}}<n^{-1}$ and $\|A^{(n)}-\tr(A^{(n)})\|_{2}\geq \varepsilon$. Let $A^{(n)}=0$ for $n\notin E$.
Set $a=(A^{(n)})_{n\to\omega}$. Since $\lim_{n\to\omega}1_{E}(n)=1$, we have that $a=(A^{(n)}1_{E}(n))_{n\to\omega}$, so $\|a-\tr_{\omega}(a)\|_{2}\geq \varepsilon$, and since $F_{n}$ is increasing and has union $J$ we have that $a\in \Theta_{\omega}(Q)'\cap \cM_{\omega}$ but $a\notin \bC1$, a contradiction.

\end{proof}

\begin{cons}\label{cons: preliminary construction}
Suppose $(Q,\tau_{Q})$ is a tracial von Neumann algebra so that there is a free ultrafilter $\omega_{0}\in \beta\bN\setminus\bN$ and a trace-preserving embedding $\Theta_{\omega_{0}}\colon Q\to \cM_{\omega_{0}}$ so that $\Theta_{\omega_{0}}(Q)'\cap \cM_{\omega_{0}}=\bC1$. Fix generators $(x_{j})_{j\in J}$ for $Q$ with $J$ countable, and set $B=C^{*}(1,(x_{j})_{j\in J})$. Let $n_{1}<n_{2}<\cdots$ and $X^{(k)}\in \prod_{j\in J}\|x_{j}\|\Ball(\bM_{n_{k}}(\bC))$ be as in (\ref{item: asymptotic criterion}). 
Let $\cA=B*_{u}C([-1,1])^{*_{u}\bN}$, and let $\cK$ be set the traces $\tau$ on $\cA$ satisfying the following: there is a free ultrafilter $\omega\in \beta\bN\setminus\bN$ and a unital $*$-homomorphism 
\[\Psi\colon \cA\to \prod_{k\to\omega}\bM_{n_{k}}(\bC)\] 
with $\Psi(x_{j})=(X^{(k)}_{j})_{k\to\omega}$ and with $\tau=\tr_{\omega}\circ \Psi$.  
\end{cons}


For $i\in\bN$, we use $y_{i}$ for the identity function on $[-1,1]$ in the $i^{th}$ copy of $C([-1,1])$ inside $\cA$. 
It will be helpful to rephrase being an element of $\cK$ in the following way.

\begin{prop}\label{prop: usual diagaonl stuff}
Let $\cK$ be as defined above. Let $R\in [0,+\infty)^{J\sqcup \bN}$ be given by $R_{j}=\|x_{j}\|$ for $j\in J$, and $R_{i}=1$ for $i\in \bN$. 
Given $\ell\in \Sigma_{R}$, the following are equivalent:
\begin{enumerate}[(i)]
\item there is a $\tau\in K$ so that $\ell=\tau\circ \ev_{(x_{j})_{j\in J},(y_{i})_{i=1}^{\infty}}$, \label{item: ultraproduct reformulation trace}
\item there is a strictly increasing sequence $k_{1}<k_{2}<\cdots$ of natural numbers, and $A^{(s)}=(A_{i}^{(s)})_{i=1}^{\infty}\in \Ball(\bM_{n_{k_{s}}}(\bC)_{s.a.})^{\bN}$ for $p\in \bN$ so that $\ell_{(X^{(k_{s})}_{j})_{j\in J},A^{(s)}}\to_{s\to\infty} \ell$.
\label{item: sequential refomrulation trace}
\end{enumerate}

\end{prop}

\begin{proof}

(\ref{item: sequential refomrulation trace}) implies (\ref{item: ultraproduct reformulation trace}): Let $\omega$ be any free ultrafilter on $\bN$, and define another free ultrafilter $\widetilde{\omega}$ on $\bN$ by 
\[\lim_{s\to\widetilde{\omega}}f(s)=\lim_{n\to\omega}f(k_{s}), \text{ for all $f\in \ell^{\infty}(\bN).$}\]
Then $\prod_{s\to\omega}\bM_{n_{k_{s}}}(\bC)\cong \prod_{k\to\widetilde{\omega}}\bM_{n_{k}}(\bC)$.
Our hypothesis implies that we have a unital $*$-homomorphism
\[\Psi\colon \cA\to \prod_{s\to\omega} \bM_{n_{k_{s}}}(\bC)\]
so that:
\begin{itemize}
    \item $\Psi_(x_{j})=(X_{i}^{(k_{s})})_{s\to\omega}$, for $j\in J$,
    \item $\Psi(y_{i})=(A_{i}^{(s)})_{s\to\omega}$, for $i\in \bN$,
    \item $\ell=\tr_{\omega}\circ \Psi\circ \ev_{(x_{j})_{j\in J},(y_{i})_{i=1}^{\infty}}$. 
\end{itemize}
By design $\Psi(x_{j})=(X^{(k)}_{j})_{j\to\widetilde{\omega}}$ under the identification $\prod_{k\to\widetilde{\omega}}\bM_{n_{k}}(\bC)=\prod_{s\to\omega}\bM_{n_{k_{s}}}(\bC).$ Thus, $\tau=\tr_{\widetilde{\omega}}\circ \Psi\in \cK$ and $\ell=\tau\circ \ev_{(x_{j})_{j\in J},(y_{i})_{i=1}^{\infty}}$.

(\ref{item: ultraproduct reformulation trace}) implies (\ref{item: sequential refomrulation trace}):
Let $\bC^{*}\ip{(T_{j})_{j\in J\sqcup \bN}}'$ denote the space of linear functions on $\bC^{*}\ip{(T_{j})_{j\in J\sqcup \bN}}$ endowed with the weak$^{*}$-topology. Since $\bC^{*}\ip{(T_{j})_{j\in J\sqcup \bN}}'$ is metrizable,  we may find a decreasing sequence $\cV_{m}$ of weak$^{*}$-neighborhoods in $\bC^{*}\ip{(T_{j})_{j\in J\sqcup \bN}}'$ of $\{0\}$ so that $\{0\}=\bigcap_{m=1}^{\infty}\cV_{s}$. Choice a free ultrafilter $\omega$ and a $*$-homomorphism $\Psi\colon \cA\to \prod_{k\to\omega}\bM_{n_{k}}(\bC)$ so that $\tau=\tr_{\omega}\circ \Psi$, and $\Psi(x_{j})=(X^{(k)}_{j})_{k\to\omega}$ for $j\in J$. Write $\Psi(y_{i})=(A_{i}^{(k)})_{k\to\omega}$, with $A_{i}^{(k)}$ self-adjoint.  Since continuous functional calculus commutes with $*$-homomorphisms, replacing $A_{i}^{(n)}$ with $f(A_{i}^{(k)})$ where $f\colon \bR\to [-1,1]$ is any continuous function with $f|_{[-1,1]}=\id$, we may, and will, assume that $\|A_{i}^{(k)}\|\leq 1$ for all $i,k\in \bN$. 

We then have that 
\[\lim_{k\to\omega}\ell_{(X^{(k)}_{j})_{j\in J},(A^{(k)}_{i})_{i=1}^{\infty}}=\tau\circ \ev_{(x_{j})_{j\in J},(y_{i})_{i=1}^{\infty}}.\]
In particular, for every $m\in \bN$, 
\[\{k:\ell_{(X^{(k)}_{j})_{j\in J},(A_{i}^{(k)})_{i=1}^{\infty}}\in \mathcal{V}_{m}+\tau\circ \ev_{(x_{j})_{j\in J},(y_{i})_{i=1}^{\infty}}\}\]
is infinite. As such, we can construct a strictly increasing sequence $k_{1}<k_{2}<\cdots$ so that 
\[\ell_{(X^{(k_{s})})_{j\in J},(A_{i}^{(k_{s})})_{i=1}^{\infty}}\in \mathcal{V}_{s}+\tau\circ \ev_{(x_{j})_{j\in J}),(y_{i})_{i=1}^{\infty}}, \text{ for all $p\in \bN$.}\]
Thus 
\[\ell_{(X^{(k_{s})}_{j})_{j\in J},(A^{(s)}_{i})_{i=1}^{\infty}}\to_{s\to\infty}\ell.\]

\end{proof}

From this, it is direct to prove that $\cK$ is a compact family of extremal traces. 

\begin{prop}\label{prop: get that Bauer simplex}
Let $\cA$ and $\cK$ be defined as above. Then $\cK$ is a weak$^{*}$-compact family of traces and $\cK\subseteq \Ext(\cT(A))$.    
\end{prop}

\begin{proof}
Given $\tau\in \cK$, choose a free ultrafilter    $\omega\in \beta\bN\setminus \bN$ be and a unital $*$-homomorphism
\[\Psi\colon \cA\to \prod_{k\to\omega}\bM_{n_{k}}(\bC)\]
with $\Psi(x_{j})=(X^{(k)}_{j})$ for all $j\in J$ and so that $\tau=\tr_{\omega}\circ \Psi$. Set $\cM=\prod_{k\to\omega}^{tr}\bM_{n_{k}}(\bC).$ Then 
\[Z(\overline{\Psi(\cA)}^{SOT})\subseteq \Psi(B)'\cap \cM.\]
Since $\tau|_{B}=\tau_{Q}|_{B}$, we may extend $\Psi|_{B}$ to a trace-preserving $*$-homomorphism $\Theta\colon Q\to \cM$. By density of $B$ in $Q$, we have that 
\[\Psi(B)'\cap \cM=\Theta(Q)'\cap \cM.\]
But, by the proof of Proposition \ref{prop: asymptotic rephrase}, we see that $\Theta(Q)'\cap M=\bC1$.

Thus we have shown that $Z(\overline{\Psi(\cA)}^{SOT})=\bC1$, and uniqueness of GNS representations show that
\[\overline{\pi_{\tau}(\cA)}^{SOT}\cong \overline{\Psi(\cA)}^{SOT}.\]
So $\overline{\pi_{\tau}(\cA)}^{SOT}$ is a factor, i.e. $\tau\in \Ext(\cT(\cA))$. 

To show that $\cK$ is compact, since $\cT(\cA)$ is compact, it suffices to show that $\cK$ is closed in $\cT(\cA)$. Since $\cA$ is separable, we know that $\cT(\cA)$ is metrizable. Thus given a sequence $(\tau_{n})_{n=1}^{\infty}\in \cK$ with $\tau_{n}\to_{n\to\infty}\tau$ it suffices to show that $\tau\in \cK$. Let $(\cV_{m})_{m=1}^{\infty}$ be as in the proof of Proposition \ref{prop: usual diagaonl stuff}. Choose weak$^{*}$-open neighborhoods $\cW_{m}$ of $\{0\}$ so that $\cW_{m}+\cW_{m}\subseteq \cV_{m}$.  Passing to a subsequence we may, and will, assume that $(\tau_{s}-\tau)\circ \ev_{(x_{j})_{j\in J},(y_{i})_{i=1}^{\infty}}\in \cW_{s}$ for all $s\in \bN$. Iteratively applying Proposition \ref{prop: usual diagaonl stuff}, we may find a strictly increasing sequence $k_{1}<k_{2}<\cdots$ and $A^{(n)}=(A_{i}^{(s)})_{i=1}^{\infty}\in \Ball(\bM_{n_{k_{s}}}(\bC)_{s.a})^{\bN}$ so that \[\ell_{(X^{(k_{s})}_{j})_{j\in J},(A^{(s)}_{i})_{i=1}^{\infty}}-\tau_{s}\circ \ev_{(x_{j})_{j\in J},(y_{i})_{i=1}^{\infty}}\in  \cW_{s}, \text{ for all $n\in \bN$.}\]
Thus
\[\ell_{(X^{(k_{s})}_{j})_{j\in J},(A_{i}^{(s)})_{i=1}^{\infty}}-\tau\circ \ev_{(x_{j})_{j\in J},(y_{i})_{i=1}^{\infty}}\in \mathcal{V}_{s}\]
for all $n$, and thus $\ell_{(X^{(k_{s})}_{j})_{j\in J},(A^{(s)}_{i})_{i=1}^{\infty}}\to\tau\circ \ev_{(x_{j})_{j\in J},(y_{i})_{i=1}^{\infty}}.$ 
Hence by Proposition \ref{prop: usual diagaonl stuff} it follows that $\cK$ is closed.

\end{proof}

\begin{cons}\label{cons: defining the bundle}
Let $\cA,Q,\cK$ be as Construction \ref{cons: preliminary construction}.
We let $A=\cA/\{a\in \cA:\tau(a^{*}a)=0 \text{ for all $a\in \cK$}\}$, and we let $K$ be the compact family of extremal traces on $A$ induced by $\cK$. Then $K$ is a faithful family of extremal traces on $A$. Set $X=\overline{\co}(K)$, then by Proposition \ref{prop: creating faces}, we know that $X$ is a face in $\cT(A)$, and $\Ext(X)=K$.  Thus by \cite[Theorem 3]{OzawaBundle} (see also \cite[Theorem 3.37]{TraciallyComplete}) and \cite[Proposition 3.23 (ii) and (iv)]{TraciallyComplete} if $(M,X)$ is the universal tracial completion of $(A,X)$ we have that $M$ is a $W^{*}$-bundle over $K$. 
\end{cons}

\begin{prop}\label{prop: relative commutant condition bundle}
Let $(M,\bE,K)$ be the bundle given in Construction \ref{cons: defining the bundle}. Then the $M$ admits a sub-bundle $N$ which is the trivial bundle over $K$ with fibers $Q$. Moreover, $Q'\cap M_{p}=\bC$ for every $p\in K$.     
\end{prop}

\begin{proof}
Let $X\subseteq \cT(M)$ be as above. By \cite[Theorem 3]{OzawaBundle} (see also \cite[Theorem 3.37]{TraciallyComplete}) for every $\tau\in X$ there is a unique $\mu_{\tau}\in \Prob(K)$ so that 
\[\tau(fa)=\int f(p)p(a)\,d\mu_{\tau}(a)\textnormal{ for all $f\in C(K),a\in M$}.\]
In particular, applying this to $\tau_{p}=\mu_{\delta_{p}}$ we see that 
\[\tau_{p}(fa)=f(p)p(a) \textnormal{ for all $p\in K$}\]
Let $\cA$ be as above and $q\colon \cA\to A$. By design, for each $p\in K$, there is a $\chi_{p}\in \cK$ so that $\chi_{p}=p\circ q$. Moreover, by definition of $\cK$, we may choose a free ultrafilter $\omega\in \beta\bN\setminus\bN$ and a $*$-homomorphism $\Psi\colon\cA\to \prod_{k\to\omega}\bM_{n_{k}}(\bC)$ with $\chi_{p}=\tr_{\omega}\circ \Psi$, and $\Psi(x_{j})=(X^{(k)}_{j})_{k\to\omega}$. By uniqueness of GNS representations, we thus see that 
\[M_{p}=\pi_{\tau_{p}}(M)\cong \overline{\pi_{p}(A)}^{SOT}\cong\overline{\pi_{\chi_{p}}(\cA)}\cong \overline{\Psi(M)}^{SOT}.\]
Moreover, since $\chi_{p}|_{B}=\tau_{Q}|_{B}$ is independent of $p$, we have that 
\[\tau(fb)=\tau_{Q}(b)\int f\,d\mu_{\tau} \text{ for all $b\in B,f\in C(K)$}\]
which implies by uniqueness of GNS representations that the $\|\cdot\|_{\bE}$-closure of $C^{*}(C(K),B)$ inside $M$ is $C_{\sigma}(K,Q)$. It follows that the image of $\overline{\Psi(B)}^{SOT}$ under each of these isomorphism gives us a fixed copy of $Q$. So we have 
\[Q'\cap M_{p}\cong Q'\cap \overline{\Psi(M)}^{SOT},\]
and in the proof of Proposition \ref{prop: get that Bauer simplex}, we see that this relative commutant is trivial. 
\end{proof}

The following lemma ultimately shows that our construction is universal enough to allow $(M,X)$ to admit a very large set of microstates. 
\begin{lem}\label{lem: getting many microstates}
Let $(M,X)$ be the bundle constructed in Construction \ref{cons: defining the bundle}. Let $R\in [0,+\infty)^{J\sqcup \bN}$ be given by $R_{j}=\|x_{j}\|$ for $j\in J$ and $R_{i}=1$ for $i\in \bN$.  Let $(y_{i})_{i=1}^{\infty}$ be the natural generators of $C([-1,1])^{*\bN}$ arising from the $i^{th}$ copy of the identity function on $[-1,1]$. Then for any neighborhood $\mathcal{O}$ of $L_{(x_{j})_{j\in J},(y_{i})_{i=1}^{\infty});X}$ there is a $K\in \bN$ so that for all $k\geq K$ we have that $((X^{(k)}_{j})_{j\in J},(A_{i})_{i=1}^{\infty})\in \Gamma^{(n_{k})}_{R}(\mathcal{O})$ for all $(A_{i})_{i=1}^{\infty}\in (\Ball(\bM_{n_{k}}(\bC)_{s.a.}))^{\bN}$.

\end{lem}

\begin{proof}
If the lemma were false, then we could find a neighborhood $\mathcal{O}$ of $L_{(x_{j})_{j\in J},(y_{i})_{i=1}^{\infty};X}$, a strictly increasing sequence $k_{s}\in \bN$ and $(A_{i}^{(s)})_{i=1}^{\infty}\in \bM_{n_{k_{s}}}(\bC)_{s.a.}^{\bN}$ so that $\ell_{(X^{(k_{s})}_{j})_{j\in J},(A_{i}^{(s)})_{i}}\notin \mathcal{O}$. Passing to a further subsequence we may, and will, assume that $\ell_{(X^{(k_{s})}_{j})_{j\in J},(A_{i}^{(s)})_{i}}\to_{s\to\infty}\ell\in \Sigma_{R}$. By Proposition \ref{prop: usual diagaonl stuff}, it follows that $\ell=\tau\circ \ev_{(x_{j})_{j\in J},(y_{i})_{i=1}^{\infty}}$ for some $\tau\in K$. Thus $\ell\in L_{(x_{j})_{j\in J},(y_{i})_{i=1}^{\infty};X}.$ Since $\mathcal{O}$ is a neighborhood of $L_{(x_{j})_{j\in J},(y_{i})_{i=1}^{\infty};X}$ it follows that $\ell_{(X^{(k_{s})}_{j})_{j\in J},(A_{i}^{(s)})_{i}}\in \mathcal{O}$ for all large $s$. This is a contradiction.

\end{proof}

\begin{thm}\label{thm: infintie MD}
Let $(M,X)$ be the bundle defined in Construction \ref{cons: defining the bundle}. Then $\delta_{MD}(M,X)=+\infty$.

\end{thm}

\begin{proof}
Define $R\in [0,+\infty)^{J\sqcup \bN}$ by $R_{j}=\|x_{j}\|$ for $j\in J$ and $R_{i}=1$ for $i\in \bN$. 
Fix a neighborhood $\mathcal{O}$ of $L_{(y_{j})_{j\in J},(x_{i})_{i=1}^{\infty};X}$, an $\varepsilon>0$ and a finite $\widetilde{F}\subseteq J\sqcup \bN$, and set $F=\widetilde{F}\cap \bN$. By Lemma \ref{lem: getting many microstates}, for all sufficiently large $k$, there is a well-defined map
\[f\colon \Ball(\bM_{n_{k}}(\bC)_{s.a})^{\bN}\to \Gamma_{R}^{(n_{k})}(\mathcal{O})\]
given by $f((A_{i})_{i=1}^{\infty})=((X^{(k)}_{j})_{j\in J},(A_{i})_{i=1}^{\infty})$.
If $g\colon \Gamma_{1}^{(n)}(\mathcal{O})\to \Delta$ is an $(\widetilde{F},\varepsilon)$-injective function, then $g\circ f$ is an $(F,\varepsilon)$-injective function. Hence
\[\dim_{\varepsilon,F}(\Ball(\bM_{n_{k}}(\bC)_{s.a})^{\bN},\|\cdot\|_{2})\leq \dim_{\varepsilon,\widetilde{F}}(\Gamma_{R}^{(n)}(\mathcal{O}),\|\cdot\|_{2})\]
for all sufficiently large $k$. It follows that 
\[\limsup_{k\to\infty}\frac{1}{n_{k}^{2}}\dim_{\varepsilon,F}(\Ball(\bM_{n_{k}}(\bC)_{s.a})^{\bN},\|\cdot\|_{2}))\leq \inf_{\mathcal{O}}\limsup_{n\to\infty}\frac{1}{n^{2}}\dim_{\varepsilon,\widetilde{F}}(\Gamma_{R}^{(n)}(\mathcal{O}),\|\cdot\|_{2}),\]
where the infimum is over all weak$^{*}$-neighborhoods $\mathcal{O}$ of $L_{(X^{(k)_{j}}))_{j\in J},(y_{i})_{i=1}^{\infty};X}$. 
Taking the supremum over $\varepsilon$ and applying Theorem \ref{thm:computation of dim for op norm vs 2-norm} we thus see that $\sup_{\varepsilon>0}\md_{\varepsilon,\widetilde{F}}((x_{j})_{j\in J},(y_{i})_{i=1}^{\infty};X)\geq |F|$ for all finite $\widetilde{F}\subseteq J\sqcup \bN$. The proof is now completed by taking the supremum over finite subsets of $J\sqcup \bN$.

\end{proof}

We now prove the main theorem of the paper.

\begin{proof}[Proof of Theorem \ref{thm: main thm intro}]
Let $Q$ be any von Neumann algebra which has an embedding into an ultraproduct of matrices with trivial relative commutant, e.g. any of the examples given in Section \ref{sec: trivial rel comm examples}.
Let $(M,\bE,K)$ be the $W^{*}$-bundle defined in Construction \ref{cons: defining the bundle}. By Proposition \ref{prop: relative commutant condition bundle}, we have that the trivial bundle over $K$ with fiber $Q$ embeds as a sub-bundle of $M$ and that $Q'\cap M_{p}=\bC1$ for all $p\in K$. 
Since any two embeddings of the hyperfinite $\textrm{II}_{1}$-factor are unitarily conjugate, and at least of them does not have trivial relative commutant (we leave this verification as an exercise to the reader), none of them have trivial relative commutant. Thus $Q$ is a non-amenable $\textrm{II}_{1}$-factor. Since $Q$ is a subfactor of $M_{p}$, it follows that each $M_{p}$ is a nonamenable $\textrm{II}_{1}$-factor. 
By Theorem \ref{thm: infintie MD}, we have that $\delta_{MD}(M;X)=+\infty$, and so items (\ref{item: intro classes}), (\ref{item: unfiromly S1B INTRO}) follow from Corollary \ref{cor: better main thm}.

\end{proof}


Constructions \ref{cons: preliminary construction}, \ref{cons: defining the bundle} use the universal $C^{*}$-algebra associated with countably many self-adjoint contractions for concreteness. However, it is frequently useful (e.g. in terms of connections to group theory), to work with unitaries instead of self-adjoints. For this reason, we state a version of this construction for unitaries. Fix again a tracial von Neumann algebra $(Q,\tau_{Q})$ so that there is some free ultrafilter $\omega\in \beta\bN\setminus\bN$ and a trace-preserving embedding $\Theta\colon Q\to\cM_{\omega_{0}}$ with $\Theta(Q)'\cap \cM_{\omega_{0}}=\bC$. Let $(x_{j})_{j\in J}$,$B$, $(n_{k})_{k=1}^{\infty},X^{(k)}$ be as in Construction \ref{cons: preliminary construction}. Now set $\cA_{0}=B*_{u}C^{*}(\bF_{\infty})$, and let $\cK_{0}$ be the set of all traces $\tau$ on $\cA_{0}$ satisfying the following: there is a free ultrafilter $\omega\in \beta\bN\setminus\bN$ and a $*$-homomorphism
\[\Psi\colon \cA_{0}\to \prod_{k\to\omega}\bM_{n_{k}}(\bC)\]
with $\Psi(x_{j})=(X_{j}^{(k)})_{k\to\omega}$ and $\tau=\tr_{\omega}\circ \Psi$. We again set $A_{0}=\cA_{0}/\{x\in \cA_{0}:\tau(x^{*}x)=0, \text{ for all $\tau\in \cK_{0}$}\}$, and let $K_{0}$ be the traces on $A_{0}$ induced from $\cK_{0}$. Exactly as in Proposition \ref{prop: get that Bauer simplex}, we have that $K_{0}$ is a compact subset of extremal traces on $A_{0}$ and hence if $X_{0}=\overline{\co}^{wk^{*}}(K_{0})$, then $X_{0}$ is a face in $\cT(A_{0})$ and a Bauer simplex. Hence, the tracial completion $(M_{0},X_{0})$ of $(A_{0},X_{0})$ is a factorial $W^{*}$-bundle over $K_{0}$. The following is proved exactly as in the self-adjoint case, the only difference being that we use Corollary \ref{cor: epsilon dim of unitaries} to show that $\delta_{MD}(M_{0};X_{0})=+\infty$.

\begin{thm}
Let $(M_{0},\bE_{0},K_{0})$ be the $W^{*}$-bundle described above. Then:
\begin{enumerate}[(i)]
    \item $\delta_{MD}(M_{0};X_{0})=+\infty$,
    \item if $N\leq M_{0}$ is a sub-bundle so that $\sup_{p\in K_{0}}h(N_{p},\tau_{p})<+\infty$, then $M_{0}$ cannot be fiberwise generated over $N$ by a finite family of continuous sections. 
\end{enumerate}
\end{thm}

Finally, we can also use projections instead of  $C^{*}(\bF_{\infty})$ or the universal $C^{*}$-algebra associated with countably many self-adjoint contractions. 
Let $Q$ be a $\textrm{II}_{1}$-factor with separable predual which has an embedding into an ultraproduct of matrices with trivial relative commutant. Pick a countable family $(x_{j})_{j\in J}$ of projections in $Q$ which generate $Q$. Set $B=C^{*}((x_{j})_{j\in J})$. 
Fix $(X_{j}^{(k)})_{k\in \bN,j\in J},(n_{k})_{k}$ as in the conclusion to Proposition \ref{prop: asymptotic rephrase} for $(x_{j})_{j\in J}$. 

Set $\cA_{1}=B*_{u}C^{*}((\bZ/2\bZ)^{*\infty})$. For $i\in \bN$, let $g_{i}$ be the $i^{th}$ generator of $(\bZ/2\bZ)^{*\infty}$, and let $y_{i}=\frac{1+g_{i}}{2}$. Set $\cK_{1}$ to be all traces $\tau$ on $\cA_{1}$ so that there is a free ultrafilter $\omega\in \beta\bN\setminus\bN$ and a $*$-homomorphism $\Psi\colon \cA_{1}\to \prod_{k\to\omega}\bM_{n_{k}}(\bC)$ with
\begin{itemize}
    \item $\tau=\tr_{\omega}\circ \Psi$, 
    \item $\Psi(x_{j})=(X_{j}^{(k)})_{k\to\omega}$ for $j\in J$,
    \item $\tau(y_{i})=\frac{1}{2}$.
\end{itemize}
Let $A_{1}$ be the separation of $\cA_{1}$ by $\cK_{1}$ and $K_{1}$ the induced family of traces. Set $X_{1}=\overline{\co}^{wk^{*}}(K_{1})$, and $(M_{1},X_{1})$ be the tracial completion. Then again we have that $M_{1}$ is a factorial $W^{*}$-bundle over $K_{1}$, and that $Q$ is a sub-bundle with fiberwise trivial relative commutant.
Moreover, using Corollary \ref{cor: projection lower bound} with $k_{j,n_{k}}=\lfloor{n_{k}/2\rfloor}$ instead of Corollary \ref{cor: epsilon dim of unitaries} we deduce that $\delta_{MD}(M_{1};X_{1})=+\infty$. Hence we have the following, which immediately implies Theorem \ref{thm: main theorem intro but with projections}.
\begin{thm}\label{thm: case of projections}
Let $(M_{1},\bE_{1},K_{1})$ be the $W^{*}$-bundle described above. Then:
\begin{enumerate}[(i)]
    \item $M_{1}$ is generated by a countably family of projections,
    \item $\delta_{MD}(M_{1};X_{1})=+\infty$,
    \item if $N\leq M_{1}$ is any sub-bundle so that $\sup_{p\in K_{1}}h(N_{p},\tau_{p})<+\infty,$ then $M_{1}$ cannot be fiberwise generated over $N$ by a finite family of continuous sections. 
\end{enumerate}

\end{thm}

\section{Relations to the generator problem for von Neumann algebras}

\subsection{Universally measurable and Borel functions on bundles}\label{Sec:universmally measurable bundle}

We recall the construction of Evington-Pennig \cite{EPBundles}, where we also the refer the reader to for complete proofs of everything claimed here.

Given a $W^{*}$-bundle $(M,\bE,K)$, set $B=\bigsqcup_{p\in K}M_{p}$. We define $\pi\colon  B\to K$ by $\pi(b)=p$ if $b\in M_{p}$. For $U\subseteq K$ open, $\varepsilon>0$, and $a\in M$, we let $V(a,\varepsilon,U)=\{b\in B:\pi(b)\in U, \|a_{\pi(b)}-b\|_{2}<\varepsilon\}$. This forms a basis for a topology on $B$. For $X\subseteq B$, we set 
$X\times_{\pi}X=\{(x,y)\in X^{2}:\pi(x)=\pi(y)\}$, and for $r>0$, we let $r\Ball(B)=\{x\in B:\pi(x)=p, \|x\|_{M_{p}}\leq r\}$.
The triple $(B,\pi,K)$ has the following properties:
\begin{enumerate}[(i)]
    \item the fiberwise addition map $B\times_{p}B\to B$ is continuous,
    \item the fiberwise scaling map $\bC\times B\to B$ is continuous,
    \item the fiberwise $*$-operation $B\to B$ is continuous,
    \item the sections $0\colon K\to B,1\colon K\to B$ which send each $p\in K$ to the $0,1$ elements of $M_{p}$ are continuous,
    \item the fiberwise trace and $\|\cdot\|_{2}$-maps $B\to \bC,B\to [0,+\infty)$ are continuous,
    \item given a net $(b_{i})_{i}$ in $B$ and $p\in K$, we have that $b_{i}$ converges to the zero element of $M_{p}$ if and only if $\pi(b_{i})\to p$ and $\|b_{i}\|_{L^{2}(\tau_{\pi(b_{i})})}\to_{i}0$.
    \item the fiberwise product map $(r\Ball(B))\times_{\pi}(r\Ball(B))\to B$ is continuous for all $r>0$,
    \item $\pi|_{r\Ball(B)}$ is an open map for all $r>0$. 
\end{enumerate}

For a Polish space $X$, we let $\cB_{X}$ be the Borel $\sigma$-algebra of $X$. For $\mu\in \Prob(X)$, we let $\cG_{\mu}$ be the completion of $\cB_{X}$ with respect to $\mu$. We define the $\sigma$-algebra of \emph{universally measurable sets} to be 
\[\cG_{u,X}=\bigcap_{\mu\in \Prob(X)}\cG_{\mu}.\]
While not obvious, there are many elements of $\cG_{u,X}$ that are not in $\cB_{X}$, see e.g \cite[Theorems 14.2 and 21.10]{KechrisClassic}.
Recall that if $Y$ is a topological space, we say that $f\colon X\to Y$ is \emph{Borel} if the inverse image of any open subset of $Y$ is Borel. Similarly, we say that $f$ is \emph{universally measurable} if the inverse image of any open subset of $Y$ is universally measurable.
\begin{prop}\label{prop:Borel TFAE}
 Let $(M,\bE,K)$ be a $W^{*}$-bundle which is $\|\cdot\|_{\bE}$-separable. Let $\pi\colon B\to K$ be the corresponding topological bundle of \cite{EPBundles} as describe above. Fix $\cF\in \{\cB_{K},\cG_{u,K}\}$. Given a section $a\colon K\to B$ of $\pi$ (which we denote as $p\mapsto a_{p}$, the following are equivalent:
 \begin{enumerate}[(i)]
 \item for every $b\in M$ we have that $p\mapsto \tau_{p}(a_{p}b_{p})$ is $\cF$-measurable, \label{item: pointwise traces}
 \item for every $b\in M$, we have that $p\mapsto \|a_{p}-b_{p}\|_{L^{2}(\tau_{p})}$ is $\cF$-measurable. \label{item: distance is measurable}

 \end{enumerate}
\end{prop}

\begin{proof}
(\ref{item: pointwise traces}) implies (\ref{item: distance is measurable}):
First note that by $\|\cdot\|_{\bE}$-separability, we may find a countable family $(c_{j})_{j\in J}$ in $M$ so that $\overline{\{c_{j,p}:j\in J\}}^{\|\cdot\|_{2}}=M_{p}$ for all $p\in K$. 
Then:
 \[\|a_{p}-b_{p}\|_{2}=\sup_{j\in J}1_{(0,\infty)}(\|c_{j,p}\|_{2})\frac{\Re(\tau((a_{p}-b_{p})c_{j,p}^{*}))}{\|c_{j,p}\|_{2}}.\]
 This is a countable supremum of $\cF$-measurable functions, hence $\cF$-measurable.
 
 (\ref{item: distance is measurable}) implies (\ref{item: pointwise traces}): This follows from the fact that
 \[\tau(a_{p}b_{p})=\frac{1}{2}\left(\sum_{j=0}^{4}i^{j}\|a_{p}+i^{j}b_{p}^{*}\|_{2}^{2}\right).\]

\end{proof}

From this, we can give an alternate description of being Borel/universally measurable in terms of inner products against continuous sections.

\begin{thm}
Let $(M,\bE,K)$ be a $W^{*}$-bundle which is $\|\cdot\|_{\bE}$-separable, and let $\pi\colon B\to K$ be the corresponding topological bundle of \cite{EPBundles}, as describe above. Let $a\colon K\to B$ be a section of $B$. Then
$a$ is Borel (resp. universally measurable) if and only if for every $b\in M$ we have that $p\mapsto \tau_{p}(a_{p}b_{p})$ is Borel (resp. universally measurable).
\end{thm}

\begin{proof}
We do the proof for Borel, the proof for universally measurable is the same. 

First, suppose that $a\colon K\to B$ is Borel. 
Let $t>0$, then by Proposition \ref{prop:Borel TFAE} we have to show that for $b\in M$ we have that
\[\{p\in K:\|a_{p}-b_{p}\|\geq t\}\]
is Borel. But this is 
\[a^{-1}(\{x\in B:\|x-b_{\pi(x)}\|_{2}\geq t\}),\]
and by definition, the set $\{x\in B:\|x-b_{\pi(x)}\|_{2}\geq t\}$ is closed in $B$.

Conversely, suppose that for all $b\in M$ we have that $p\mapsto \tau_{p}(a_{p}b_{p})$ is Borel. Let $b\in M$, $\varepsilon>0$ and an open $U\subseteq K$ be given. Define $f\colon K\to [0,\infty)$ by $f(p)=\|a_{p}-b_{p}\|_{2}$. Note that $f$ is Borel, by Proposition \ref{prop:Borel TFAE}.
Then:
\[a^{-1}(V(b,\varepsilon,U))=\{p\in K:\|a_{p}-b_{p}\|_{2}<\varepsilon\}=f^{-1}([0,\varepsilon))\]
is Borel.

\end{proof}

\subsection{A weak topology on Evington-Pennig's topological bundle}

In order to have access to descriptive set-theoretic machinery, we need to prove that $\Ball(B)$ is a standard Borel space when $M$ is $\|\cdot\|_{\bE}$-separable. We give a proof of this by using an alternate topology (analogous to the weak$^{*}$-topology) which has the same Borel sets and which makes $\Ball(B)$ compact metrizable.

Let $(M,\bE,K)$ be a $W^{*}$-bundle, and let $B$ be the topological bundle of Evington-Pennig described in Section \ref{Sec:universmally measurable bundle}. Given $b\in B$, $a\in M$, and $\varepsilon>0$ and $U\subseteq K$ open, we let 
\[\widetilde{V}(b,a,U,\varepsilon)=\{b'\in B:|\tau(ba_{\pi(b)})-\tau(b'a_{\pi(b')})|<\varepsilon\}.\]
We define a new topology on $B$ by saying that for $b\in B$ the sets
\[\widetilde{\mathcal{V}}_{F,U,\varepsilon}(b)=\bigcap_{a\in F}\widetilde{\mathcal{V}}(b,a,U,\varepsilon)\]
ranging over finite $F\subseteq M$, $\varepsilon>0$, and open neighborhoods $U\subseteq K$ of $\pi(b)$ form a neighborhood basis at $b$. Let $\widetilde{\cT}$ be the generated topology.

\begin{prop}\label{prop: new topology who dis}
 Let $(M,\bE,K)$ be a $W^{*}$-bundle, and let $B$ be the topological bundle of Evington-Pennig described in Section \ref{Sec:universmally measurable bundle}. Then:
 \begin{enumerate}[(i)]
     \item $\Ball(B)$ is $\widetilde{\cT}$-compact, \label{item: new wk* compact}
     \item if $M$ is $\|\cdot\|_{\bE}$-separable, then $\Ball(B)$ is metrizable and generates the same Borel $\sigma$-algebra as the original topology on $\Ball(B)$. \label{item: same borels}
 \end{enumerate}
\end{prop}
 \begin{proof}
 (\ref{item: new wk* compact}): Define 
 \[\iota\colon \Ball(B)\to \left(\prod_{a\in M}\{z\in \bC:|z|\leq \|a\|\}\right)\times K\]
 by
 \[\iota(b)=((\tau(ba))_{a\in M},\pi(b)).\]
 It is direct to see that for $b\in M$ the map $a\mapsto \tau(ab)$ is $\widetilde{\cT}$-continuous, hence $\iota$ is a continuous function. We claim that its image is closed and that $\iota$ is a homeomorphism onto its image. First note that $\iota$ is injective, simply because its second coordinate is the projection map $\pi$, and from the fact that if $p\in K$ and $x,y\in M_{p}$ have $\tau_{p}(xa)=\tau_{p}(ya)$ for all $a\in M_{p}$, then $x=y$.

Since $\iota$ is  continuous injective, to show that $\iota$ is a homeomorphism onto its image, it suffices to show that $\iota(\widetilde{\mathcal{V}}(b,a,U,\varepsilon))$ is an open subset of $\iota(\Ball(B))$. But,
\[\iota(\widetilde{\mathcal{V}}(b,a,U,\varepsilon))=\iota(\Ball(B))\cap \left[\{z\in \bC^{M}:|z_{a}-\tau(ba)|<\varepsilon\} \times U\right]
,\]
and this is an open subset of $\iota(\Ball(B))$.

Finally, let us show that the image is closed. Suppose that $x_{\alpha}\in\Ball(B)$ is a net and that $\iota(x_{\alpha})\to (z,p)$. Necessarily this forces that $\pi(x_{\alpha})\to p$, and that $\lim_{\alpha}\tau_{\pi(x_{\alpha})}(x_{\alpha}b_{\pi(x_{\alpha})})$ converges for all $b\in M$. Define $\widetilde{B}\colon M\times M\to \bC$ by
\[B(a,c)=\lim_{\alpha}\tau_{\pi(x_{\alpha})}(c_{\pi(x_{\alpha})}^{*}x_{\alpha}a_{\alpha}).\]
Then
\[|B(a,c)|\leq \liminf_{\alpha}\|c_{\pi(x_{\alpha})}\|_{2}\|a_{\pi(x_{\alpha})}\|_{2}=\|c_{p}\|_{2}\|a_{p}\|_{2}.\]
This implies that $\widetilde{B}$ descends to a unique bounded sesquilinear form $B\colon M_{p}\times M_{p}\to \bC$ which satisfies that $B(a_{p},c_{p})=\widetilde{B}(a,c)$ for all $a,c\in M$. Thus there is a unique operator $T\in B(L^{2}(M_{p}))$ with $B(a,c)=\ip{Ta,c}$. Moreover, since $B(ay,c)=B(a,cy^{*})$ it is direct to show that $T$ commutes with right multiplication by elements of $M_{p}$, and hence $T$ is given as left multiplication by some element $x$ in $M_{p}$. By definition of $\widetilde{\cT}$ we have that $\iota(x_{\alpha})\to \iota(x)$  and hence the image of $\iota$ is closed.

(\ref{item: same borels}): The fact that $\Ball(B)$ has the same Borel sets as the original topology follows from Proposition \ref{prop:Borel TFAE}.
To show that $\Ball(B)$ is metrizable, we use a slightly different embedding then in (\ref{item: new wk* compact}). Fix a $\|\cdot\|_{\bE}$-dense sequence $(a_{n})_{n=1}^{\infty}$ in $M$. Define 
\[j\colon \Ball(B)\to \left(\prod_{n\in \bN}\{z\in \bC:|z|\leq \|a_{n}\|\}\right)\times K\]
by $j(b)=((\tau(ba_{n}))_{n=1}^{\infty},\pi(b)))$.
Since the countable product of metrizable spaces is metrizable, it suffices to show that $j$ is a homeomorphism onto its image. The fact that $j$ is continuous follows from the proof of (\ref{item: new wk* compact}). 
Suppose that $x_{\alpha}$ is a net in $\Ball(B)$ and that $j(x_{\alpha})$ converges to $j(x)$ for some $x\in\Ball(B)$. To show that $x_{\alpha}$ converges to $x$ with respect to $\widetilde{\cT}$ it suffices, from the proof of (\ref{item: new wk* compact}), to show that $\tau_{\pi(x_{\alpha})}(x_{\alpha}a_{\pi(x_{\alpha})})\to \tau_{\pi(x)}(xa_{\pi(x)})$ for all $a\in M$. Since $j(x_{\alpha})$ converges to $j(x)$ we know that 
$\tau_{\pi(x_{\alpha})}(x_{\alpha}a_{n,\pi(x_{\alpha})})\to \tau_{\pi(x)}(xa_{n,\pi(x)})$ for every $n\in \bN$. From the uniform estimate $\|x_{\alpha}\|\leq 1$, it is direct to show that 
\[\{a\in M:\tau_{\pi(x_{\alpha})}(x_{\alpha}a_{\pi(x_{\alpha})})\to \tau_{\pi(x)}(xa_{\pi(x)})\} \]
is a $\|\cdot\|_{\bE}$-closed subset of $M$, and since it contains $\{a_{n}:n\in \bN\}$, it must therefore be all of $M$.

 \end{proof}

The following corollary of the above material shows that if one could upgrade continuous in Theorem \ref{thm: main thm intro} to universally measurable, then this would amount to a disproof of the generator problem for von Neumann algebras. 
\begin{cor}
Let $(M,\bE,K)$ be the $W^{*}$-bundle which is $\|\cdot\|_{\bE}$-separable. Let $N$ be a sub-bundle of $M$, and let $\pi\colon B\to K$ be the topological bundle of \cite{EPBundles}, as described above. Fix $n\in \bN$. Then the following are equivalent:
\begin{enumerate}[(i)]
    \item for each $p\in K$, there are $a_{1},\cdots,a_{n}\in M_{p}$ so that $M=W^{*}(\{a_{1},\cdots,a_{n}\}\cup N_{p})$, \label{cor item: fiberwise finite generation over N iff}
    \item there are universally measurable sections $b_{1},\cdots,b_{n}\colon K\to B$ of $\pi$  so that $M_{p}=W^{*}(\{b_{j,p}\}_{j=1}^{n}\cup N_{p})$ for every $p\in K$. \label{cor item: finite generation univ measu}
\end{enumerate}
\end{cor}

\begin{proof}
The fact that (\ref{cor item: finite generation univ measu}) implies (\ref{cor item: fiberwise finite generation over N iff}) is direct. So we focus on the reverse implication.

Assume (\ref{cor item: fiberwise finite generation over N iff}) holds. Let $X=\{(p,x)\in K\times B^{n}:\pi(x_{j})=p \text{ for all $j\in [n]$ and } W^{*}(\{x_{j}\}_{j=1}^{n}\cup N_{p})=M_{p}\}$. We claim that $X$ is a Borel subset of $K\times \Ball(B)^{n}$. To see this, apply $\|\cdot\|_{\bE}$-separability of $M$ to find  countable families $(y_{j})_{j\in J}\in M^{J}$, $(c_{i})_{i\in I}\in N^{I}$ which are $\|\cdot\|_{\bE}$-dense in $M,N$. 
For $j\in J,k\in \bN,P\in \bC^{*}\ip{(T_{i})_{i\in I\sqcup [n]}}$, set 
\[X_{j,k,P}=\{(p,x)\in K\times \Ball(B)^{n}:\pi(x_{j})=p, \text{ for all $j\in [n]$ and } \|P(c_{p},x)-y_{j,p}\|_{2}<k^{-1}\}.\]
Then
\[X=\bigcap_{j\in J,k\in \bN}\bigcup_{P\in \bC^{*}\ip{(T_{i})_{i\in I\sqcup [n]}}}X_{j,k,P},\]
so $X$ is a $G_{\delta}$-subset of $\{(p,x)\in K\times B^{n}:\pi(x_{j})=p \text{ for all $j\in [n]$}\}$ (and in particular, is Borel). 
Since $\Ball(B)$ is standard Borel by Proposition \ref{prop: new topology who dis} (\ref{item: same borels}), we have that $X$ is a standard Borel space. By (\ref{cor item: fiberwise finite generation over N iff}) the first factor projection $\rho\colon X\to K$ is surjective. Since $X,K$ are standard Borel, the Jankov-von Neumann theorem \cite[Theorems 18.1, 21.10]{KechrisClassic} implies that there is a universally measurable section $s\colon K\to X$ of $\rho$. We write $s(p)=(p,b_{1,p},\cdots,b_{n,p})$. Then $p\mapsto b_{j,p}$ are universally measurable sections, and by design we have that 
$W^*{*}(\{b_{j,p}\}_{j=1}^{n}\cup N_{p})=M_{p}$ for all $p\in K$.

\end{proof}

\subsection{There is no continuous algorithm to solve the generator problem}

We recall the construction of Jekel's $L^{2}$-uniformly continuous functional calculus here as given in \cite{FreePinsker, JekelThesis}. Fix an index set $J,$ and $R\in [0,\infty)^{J}.$
Define
\[\mathcal{A}_{R}=C(\Sigma_{R})\otimes_{\textnormal{alg}}\bC^{*}\ip{(T_{j})_{j\in J}}.\]
Given a tracial von Neumann algebra $(M,\tau)$ and $x\in \prod_{j\in J}R_{j}\Ball(M)$, we let $\ev_{x}\colon \mathcal{A}_{R}\to M$ be the linear map satisfying $\ev_{x}(\phi\otimes P)=\phi(\ell_{x})P(x)$ for $\phi \in C(\Sigma_{R,J})$,$P\in \bC^{*}\ip{(T_{j})_{j\in J}}$. For $f\in \mathcal{A}_{R}$ we will use $f(x)$ for $\ev_{x}(f).$
Define a seminorm $\|\cdot\|_{R,2}$ on $\mathcal{A}_{R}$ by
\[\|f\|_{R,2}=\sup_{x,(M,\tau)}\|f(x)\|_{2},\]
where the supremum is over all tracial von Neumann algebras $(M,\tau)$ and all $x\in \prod_{j\in J}R_{j}\Ball(M).$ We then let $\mathcal{F}_{R,2}$ be the completion of
\[\mathcal{A}_{R}/\{f\in \mathcal{A}_{R}:\|f\|_{R,2}=0\}\]
under the norm induced by $\|\cdot\|_{R,2}.$  For $f\in \mathcal{F}_{R,2}$ we let
\[\|f\|_{R}=\sup_{x, (M,\tau)}\|f(x)\|_{\infty}\in [0,+\infty],\]
where again the supremum is over all tracial von Neumann algebras and all $x\in \prod_{j\in J}R_{j}\Ball(M)$. We set
\[\mathcal{F}_{R}=\{f\in \mathcal{F}_{R,2}:\|f\|_{R}<\infty\}.\]
Note that for a tracial von Neumann algebra $(M,\tau)$ and $x\in \prod_{j\in J}R_{j}\Ball(M),$ we then have that $f(x)\in M.$

By \cite[Section 3.1]{FreePinsker} (see also \cite[Section 2.4]{HayesPT}) the  product and   $*$-operation have a unique extension to product and $*$-operations on $\mathcal{F}_{R}$ which satisfy the axioms of a $*$-algebra as well as the following estimates
 \[\|f\|_{R}=\|f^{*}\|_{R},\,\,\ \|f\|_{R,2}=\|f^{*}\|_{R,2}\]
 \[\|fg\|_{R}\leq \|f\|_{R}\|g\|_{R},\,\,\,\, \|fg\|_{R,2}\leq \|f\|_{R}\|g\|_{R,2}.\]
Further, under these extended operations and the norm $\|\cdot\|_{R}$ we have that $\cF_{R}$ is a $C^{*}$-algebra.

We remark that $\cF_{R}$ may be viewed as a $W^{*}$-bundle over $\Sigma_{R}$. Indeed, for $\ell\in \Sigma_{R}$, choose an von Neumann algebra $(M,\tau)$ and an $a\in \prod_{j\in J}R_{j}\Ball(M)$ with $\ell_{a}=\ell$. We then define $\bE(f)(\ell)=\tau(f(a))$. It is direct to verify that this is independent of the choice of $(M,\tau,a)$ (indeed such a triple is unique up to isomorphism), that $\bE(f)\in C(\Sigma_{R})$, that $\|f\|_{R,2}=\bE(f^{*}f)^{1/2}$, and that $\Ball(\cF_{R})$ is complete with respect to $\|\cdot\|_{R,2}$. 

We now recall the basic properties of this highly general noncommutative functional calculus.
\begin{enumerate}[(i)]
\item Let $(M,\tau)$ be a tracial von Neumann  algebra and $x\in \prod_{j\in J}R_{j}\Ball(M)$. Then the map $\mathcal{F}_{R}\to W^{*}(x)$ given by $f\mapsto f(x)$ is surjective. In fact, for all $a\in W^{*}(x),$ there is an $f\in \mathcal{F}_{R}$ with $\|f\|_{R}\leq \|a\|$ and so that $f(x)=a.$ \label{item:surjective nc func calc}
\item Every $f\in \mathcal{F}_{R,2}$ is $\|\cdot\|_{2}$-uniformly continuous in the following sense. For every $\varepsilon>0,$ there is a $\delta>0$ and a finite $F\subseteq J$ so that if $(M,\tau)$ is any tracial von Neumann algebra and $x,y\in \prod_{j\in J}R_{j}\Ball(M)$ with $\|x_{j}-y_{j}\|_{2}<\delta$ for all $j\in J,$ we have $\|f(x)-f(y)\|_{2}<\varepsilon.$ \label{item:unif cont nc func calc}
\item \label{item:naturality of nc func calc}
Suppose $(M_{k},\tau_{k}),k=1,2$ are tracial von Neumann algebras and $x\in \prod_{j\in J}\{a\in M_{1}:\|a\|_{\infty}\leq R_{j}\},$ and that $\Theta\colon M_{1}\to M_{2}$ is a trace-preserving, unital, normal $*$-homomorphism. Then $f(\Theta(x))=\Theta(f(x))$ for all $f\in \mathcal{F}_{R,2}.$
\end{enumerate}

Using Jekel's $L^{2}$-uniformly continuous functional calculus, we can give a precise statement that it is not possible to continuously choose a single generator for von Neumann algebras. We reminder the reader of the notation $\theta_{\ell},\pi_{\ell},W^{*}(\ell)$ given in Section \ref{sec: all praise the OG}.

\begin{thm}\label{thm: cont func of law}
Fix a countable set $I$ and $R\in [0,\infty)^{I}$. Given $\ell\in \Sigma_{R}$, we let $\vartheta_{\ell}=(\pi_{\ell}(T_{i}))_{i\in i}$. 
Then given any $\|\cdot\|_{R}$-bounded, weak$^{*}$-$\|\cdot\|_{R,2}$ continuous function $\Sigma_{R}\to \cF_{R}$ (which we denote as $\ell\mapsto f_{\ell}$), there is an $\ell\in \Sigma_{R}$ so that  $W^{*}(f_{\ell}(\vartheta_{\ell}))\neq W^{*}(\ell)$.  
\end{thm}

The above theorem can be quickly deduced from the following.

\begin{thm}\label{thm: no continuous algorithm}
Let $(M,\bE,K)$ be a $W^{*}$-bundle satisfying the conclusion of Corollary \ref{cor: better main thm}, and let $N$ be a sub-bundle of $M$ so that $\sup_{p\in K}h(N_{p},\tau_{p})<+\infty$. Fix $R\in [0,+\infty)^{I}$ and $a\in \prod_{i}R_{i}\Ball(M)$. Then given any $\|\cdot\|_{R}$-bounded, and $\|\cdot\|_{R,2}$-continuous function $p\mapsto f_{p}$ from $K\to \cF_{R}$ (giving $\cF_{R}$ the $\|\cdot\|_{R,2}$-topology), there is a $p\in K$ so that $f_{p}(a)\cup N_{p}$ does not generate $M_{p}$.   
\end{thm}

In both cases, if the generator problem is true then we can find a universally measurable solution. E.g. in the context of Theorem \ref{thm: no continuous algorithm}, it is not hard to show that 
\[\{(\ell,f)\in\Sigma_{R}\times \Ball(\cF_{R}):W^{*}(f(\vartheta_{\ell}))=W^{*}(\ell)\}\]
is a $G_{\delta}$ (hence Borel) subset of the Polish space $\Sigma_{R}\times \Ball(\cF_{R})$. 

The following Proposition does the bulk of the work of proving Theorem \ref{thm: no continuous algorithm}.

\begin{prop}\label{prop: preserving continuous functions}
Let $(M,\bE,K)$ be a $W^{*}$-bundle which is $\|\cdot\|_{\bE}$-separable. Fix a countable set $I$, an $R\in [0,\infty)^{I}$ and a tuple $a\in \prod_{i\in I}R_{i}\Ball(M)$.
Suppose that $p\mapsto f_{p}$ is a $\|\cdot\|_{R}$-bounded,  and $\|\cdot\|_{R,2}$-continuous function $K\to \cF_{R}$. Then, the section $b_{p}=f_{p}(a_{p})$ is in $M$. 
\end{prop}

\begin{proof}
We use Ozawa's local criterion \cite[Theorem 11]{OzawaBundle}.
Fix $p\in K$ and $\varepsilon>0$.  Choose a neighborhood $\mathcal{O}$ of $p$ so that $\|f_{p}-f_{x}\|_{R,2}<\varepsilon$ for all $x\in \mathcal{O}$. Choose a $\Phi\in C(\Sigma_{R})\otimes_{\textnormal{alg}}\bC^{*}\ip{(T_{i})_{i\in i}}$ so that $\|\Phi-f_{p}\|_{R,2}<\varepsilon$. Then we may define $c\in M$ by $c_{x}=\Phi(a_{x})$. Moreover, we have for all $x\in \mathcal{O}$ that:
\[\|b_{x}-c_{x}\|_{2}\leq \|f_{x}-f_{p}\|_{R,2}+\|f_{p}-\Phi\|_{R,2}<2\varepsilon.\]

\end{proof}

\begin{proof}[Proof of Theorem \ref{thm: no continuous algorithm}]
Suppose $p\mapsto f_{p}$ is a bounded, continuous function $K\to \cF_{R}$ (giving $\cF_{R}$ the $\|\cdot\|_{R,2}$-topology). Set $b_{p}=f_{p}(a_{p})$, then Proposition \ref{prop: preserving continuous functions} implies $b\in M$. Since $M$ satisfies the conclusion of Corollary \ref{cor: better main thm},  there is some $p\in K$ with $W^{*}(b_{p}\cup N_{p})\neq M_{p}$.

\end{proof}

\begin{proof}[Proof of Theorem \ref{thm: cont func of law}]
Suppose there is such a function $\ell\mapsto f_{\ell}$. Let $a,R,I$ be as in the setup of Theorem \ref{thm: no continuous algorithm}. Consider the function $K\to \Sigma_{R}$ given by $p\mapsto f_{\ell_{a_{p}}}$. This is a continuous function since the maps $p\mapsto \ell_{a_{p}}$ and $\ell\mapsto f_{\ell}$ are continuous. 

Since 
\[\ell_{\vartheta_{\ell_{a_{p}}}}=\ell_{a_{p}},\]
there is a unique trace-preserving isomorphism $W^{*}(\ell)\cong M_{p}$ which sends $\pi_{\ell}(T_{i})$ to $a_{i,p}$ for $i\in I$. Thus the fact that $f_{\ell_{a_{p}}}(\vartheta_{\ell_{a_{p}}})$ generates $W^{*}(\ell)$ implies that $f_{\ell}(a_{p})$ generates $M_{p}$. This contradicts Theorem \ref{thm: no continuous algorithm}.

\end{proof}

%
%

\end{document}